\documentclass[11pt,twosides]{article}
\usepackage[a4paper,left=2.5cm,right=2.5cm, top=2.5cm, bottom=2.5cm]{geometry}

\usepackage[latin1]{inputenc}
\usepackage{cuted}
\usepackage{cite}
\usepackage{comment}
\usepackage{enumitem}
\usepackage{multirow}
\usepackage{pifont}%
\newcommand{\cmark}{\ding{51}}
\newcommand{\xmark}{\ding{55}}
\usepackage{graphicx} %
\usepackage{xcolor}
\usepackage{hyperref}
\usepackage{tikz}
\usepackage{orcidlink}

\usepackage{amsmath,amsfonts,amssymb,amsthm}
\numberwithin{equation}{section}

\usepackage{stackengine}
\usepackage{scalerel}
\usepackage{xcolor}
\usepackage{graphicx}
\newcommand\openbigstar[1][0.7]{%
  \scalerel*{%
    \stackinset{c}{-.125pt}{c}{}{\scalebox{#1}{\color{white}{$\bigstar$}}}{%
      $\bigstar$}%
  }{\bigstar}
}

\newtheorem{theorem}{Theorem}[section]

\newtheorem{lemma}[theorem]{Lemma}
\newtheorem{proposition}[theorem]{Proposition}
\newtheorem{corollary}[theorem]{Corollary}
\newtheorem{remark}[theorem]{Remark}
\newtheorem{definition}[theorem]{Definition}

\newcommand{\eremk}{\hbox{}\hfill\rule{0.8ex}{0.8ex}}

\definecolor{myblue}{rgb}{0,0,0.6}
\hypersetup{
    colorlinks,
    linktoc=section,
    citecolor=purple,
    linkcolor=myblue,
    urlcolor=magenta,
    pdfborder={0 0 0}
}

\newcommand{\Norm}[2]{\|#1\|_{#2}}
\newcommand{\Seminorm}[2]{|#1|_{#2}}

\newcommand{\R}{\mathbb{R}}
\newcommand{\N}{\mathbb{N}}
\newcommand{\QT}{Q_T}
\newcommand{\HH}{\mathcal{H}}

\newcommand{\dpt}{\partial_t}
\newcommand{\dptau}{\partial_{\tau}} 
\newcommand{\dps}{\partial_s}
\newcommand{\dptt}{\partial_{tt}}
\DeclareMathOperator{\Span}{span}
\newcommand{\esssup}{\operatorname*{ess\,sup}}
\newcommand{\diam}{\operatorname*{diam}}

\newcommand{\ST}{\Sigma_T}
\newcommand{\SO}{\Sigma_0}
\newcommand{\Sn}{\Sigma_n}
\newcommand{\Snmo}{\Sigma_{n-1}}
\newcommand{\dd}{\mathrm{d}}
\newcommand{\dt}{\,\mathrm{d}t}
\newcommand{\ds}{\,\mathrm{d}s}
\newcommand{\bx}{\mathbf{x}}
\newcommand{\dx}{\,\mathrm{d}\mathbf{x}}
\newcommand{\dS}{\,\mathrm{d}S}

\newcommand{\calD}{\mathcal{D}}

\newcommand{\Th}{\mathcal{T}_h}
\newcommand{\Tt}{\mathcal{T}_{\tau}}
\newcommand{\Nt}{N}
\newcommand{\tnmo}{t_{n-1}}
\newcommand{\tn}{t_n}
\newcommand{\In}{I_n}

\newcommand{\Qn}{Q_n}
\newcommand{\Qm}{Q_m}
\newcommand{\jump}[1]{[\![#1]\!]}

\newcommand{\Pp}[2]{\mathbb{P}^{#1}(#2)}
\newcommand{\Pcont}[2]{\mathbb{P}^{\mathsf{cont}(#1)}(#2)}

\newcommand{\oVhp}{\mathring{\mathcal{V}}_h^p}

\newcommand{\oVdisc}[1]{\mathring{\mathcal{V}}_{h\tau}^{\mathsf{disc}(#1)}}
\newcommand{\oVcont}[1]{\mathring{\mathcal{V}}_{h\tau}^{\mathsf{cont}(#1)}}

\newcommand{\uht}{u_{h\tau}}
\newcommand{\vh}{v_h}

\newcommand{\vht}{v_{h\tau}}

\newcommand{\wht}{w_{h\tau}}
\newcommand{\zht}{z_{h\tau}}

\newcommand{\hLag}{\widehat{\mathcal{L}}}
\newcommand{\Lag}{{\mathcal{L}}}
\newcommand{\Leg}{{\mathbb{L}}}

\newcommand{\Pih}{\Pi_h}
\newcommand{\Rt}{\mathcal{R}_{\tau}^{(q+1)}}

\newcommand{\Rtq}{\mathcal{R}_{\tau}^{(q)}}
\newcommand{\Rh}{\mathcal{R}_h}%
\newcommand{\Piht}{\Pi_{h\tau}}
\newcommand{\Id}{\mathsf{Id}}
\newcommand{\Pt}{\mathbb{P}_{\tau}}

\newcommand{\PtTh}{\mathcal{P}_{\tau}^{\mathsf{Th}}}
\newcommand{\PtThqmo}{\mathcal{P}_{\tau}^{\mathsf{Th}(q-1)}}
\newcommand{\PtThL}{\widetilde{\mathcal{P}}_{\tau}^{\mathsf{Th}}}
\newcommand{\tPtTh}{\widetilde{\mathcal{P}}_{\tau}^{\mathsf{Th}}}
\newcommand{\PtW}{\mathcal{P}_{\tau}^{\mathsf{Wa}}}
\newcommand{\PtAM}{\mathcal{P}_{\tau}^{\mathsf{AM}}}
\newcommand{\Ptdqmo}{\widetilde{\mathcal{P}}_{\tau}^{\mathsf{Th}(q-1)}}
\newcommand{\Cinv}{C_{\mathsf{inv}}}
\newcommand{\ItR}{\mathcal{I}_{\tau}^{\mathsf{R}}}

\newcommand{\CSI}{C_S^{\mathcal{I}}}
\newcommand{\CS}{C_{\Pi^t}}
\newcommand{\CCFL}{C_{\mathrm{CFL}}}

\newcommand{\Bht}{\mathcal{B}_{h\tau}}

\newcommand{\ev}{e_v}
\newcommand{\eu}{e_u}
\newcommand{\tvarphi}{\widetilde{\varphi}}
\newcommand{\tlambda}{\widetilde{\lambda}}

\allowdisplaybreaks

\title{Galerkin-type time discretizations for parabolic and hyperbolic problems: stability and \emph{a priori} error analysis}

\author{
Sergio G\'omez\thanks{Department of Mathematics and Applications, University of Milano-Bicocca, Via Cozzi 55, 20125 Milan, Italy
\href{mailto:sergio.gomezmacias@unimib.it}{sergio.gomezmacias@unimib.it}
}\ \thanks{IMATI-CNR ``E. Magenes", Via Ferrata 5, 27100 Pavia, Italy}\ \orcidlink{0000-0001-9156-5135}
}
\date{}

\begin{document}

\maketitle

\begin{abstract}
\noindent 
Galerkin-type time discretizations are variational time-stepping schemes in which the test and trial spaces consist of piecewise polynomials in time.
This paper reviews the schemes from this class available in the literature for parabolic and hyperbolic problems. In particular, we focus on the historical difficulties in their stability and \emph{a priori} error analysis that have limited their applicability, especially for wave propagation problems. 
We discuss how the limitations of standard energy arguments are behind these theoretical gaps, and present a unified framework that yields stability bounds and \emph{a priori} error estimates in~$L^{\infty}(0, T; \mathcal{X})$ norms. Crucially, the stability analysis relies on a suitable choice of nonstandard test functions.
This approach is  characterized by the absence of Gr\"onwall estimates, applicability to arbitrary approximation degrees, reduced regularity assumptions, and robustness with respect to model parameters. 
\end{abstract}

\paragraph{Keywords.} Discontinuous Galerkin time discretization, continuous Galerkin time discretization, unconditional stability, \emph{a priori} error estimates

\paragraph{Mathematics Subject Classification.} 65M12, 65M60, 65M15, 35K05, 35L05

\tableofcontents

\section{Introduction}
Galerkin-type time discretizations are formulated in a variational setting, with test and trial spaces consisting of functions that are piecewise polynomials in time. Several discontinuous and continuous Galerkin time discretizations  were introduced in the last century and have remained of interest to this day. This interest, as evidenced by the extensive literature on the subject, is due to the many theoretical and computational advantages of this class of methods. Specifically, in contrast to most standard time-stepping schemes, these methods:
\begin{itemize}
\item allow for arbitrary high-order approximations in a natural way;

\item are typically unconditionally stable;

\item offer greater flexibility and possess a structure closer to that of the continuous variational problem; and

\item require weaker regularity assumptions for their \emph{a priori} error estimates.
\end{itemize}
However, although substantial numerical evidence of their performance can be found in the literature, there remain gaps in the stability and convergence analysis of some of these methods, especially for wave propagation problems. This issue arises because standard energy arguments (with natural test functions) are insufficient to provide bounds on the discrete solution at arbitrary time instances for high-order approximations, which, in some cases, prevents even the guarantee of a unique discrete solution.
The approach adopted in this work is inspired by that of Walkington~\cite{Walkington:2014} in the analysis of a discontinuous--continuous Galerkin (DG--CG) time discretization of the wave equation, which relies on nonstandard test functions to obtain continuous dependence on the data in~$L^{\infty}(0, T; \mathcal{X})$ norms.
We show that this approach can be extended to other Galerkin-type time discretizations, thereby closing most of the aforementioned theoretical gaps.

To keep the presentation accessible to a broader audience, we focus on two prototypical model problems: the heat equation (Section~\ref{sec:model-heat}) and the wave equation~(Section~\ref{sec:model-wave}). These models are chosen not only for their simplicity, but also to highlight important robustness properties of Galerkin-type time discretizations: for the heat equation, the resulting estimates are independent of negative powers of $\nu$, while for the wave equation, unconditionally stable schemes do not require a damping term. The analysis of these model problems therefore provides the foundation for the treatment of more challenging models, where robustness with respect to model parameters is essential.

We first collect the different operators involved in the analysis of the methods considered, focusing on their properties and relations. Subsequently, we present a systematic and unified variational approach to the stability and convergence analysis in~$L^{\infty}(0, T; \mathcal{X})$ norms of Galerkin-type time discretizations. 
In particular, we focus on the space--time methods resulting from combining conforming finite element discretizations in space with
\begin{itemize}
    \item[$\blacklozenge$] the discontinuous Galerkin (DG) time discretization for the heat equation by Jamet in~\cite{Jamet:1978} (Section~\ref{sec:Jamet-heat});
    \item[$\blacklozenge$] the continuous Galerkin (CG) time discretization for the heat equation by Aziz and Monk in~\cite{Aziz_Monk:1989} (Section~\ref{sec:Aziz-Monk-heat});
    \item[$\bigstar$] the \emph{``plain vanilla"} time discretization for the second-order formulation of the wave equation introduced by Hughes and Hulbert in~\cite{Hughes_Hulbert:1988} (Section~\ref{sec:plain-vanilla-wave}), as well as the variants in~\cite{Hughes_Hulbert:1988,Hulbert_Hughes:1990,French:1993,Antonietti_Mazzieri_Migliorini:2020};
    \item[$\openbigstar$] the CG time discretization for the Hamiltonian formulation of the wave equation by French and Peterson in~\cite{French_Peterson:1991} (Section~\ref{sec:French-Peterson-wave}); 
    \item[$\bigstar$] the DG time discretization for the Hamiltonian formulation of the wave equation by Johnson in~\cite{Johnson:1993} (Section~\ref{sec:Johnson-wave}); and
    \item[$\openbigstar$] the DG--CG time discretization for the second-order formulation of the wave equation by Walkington in~\cite{Walkington:2014} (Section~\ref{sec:Walkington-wave}).
\end{itemize}
To clarify the novelty of the results, the symbols used above indicate the following: $\blacklozenge$ the proposed framework offers an alternative to the discrete characteristic test function technique of Chrysafinos and Walkington~\cite{Chrysafinos_Walkington:2006} for parabolic problems; $\bigstar$~completely new results that close remaining theoretical gaps in the analysis of the method in that section are presented; and $\openbigstar$~the proposed framework was applied to this method in previous works, but is presented here in a simplified setting for completeness.

The primary motivation for presenting these results together is to demonstrate that a single underlying argument, detailed in Section~\ref{sec:structure-theory} below, provides a unified framework for the analysis of all time discretizations in this family.

\paragraph{Function spaces.} We shall use standard notation for~$L^p$, Sobolev, and Bochner spaces. In particular, given an open set~$\calD \subset \R^d$ ($d \in \{1, 2, 3\}$) with Lipschitz boundary~$\partial \calD$, $p \in [1, \infty]$, and~$s \geq 0$, we denote by~$W^{s, p}(\calD)$ the corresponding Sobolev space with seminorm~$\Seminorm{\cdot}{W^{s, p}(\calD)}$ and norm~$\Norm{\cdot}{W^{s, p}(\calD)}$. For~$p = 2$, we set~$H^s(\calD) := W^{s, 2}(\calD)$ with seminorm~$\Seminorm{\cdot}{H^s(\calD)}$ and norm~$\Norm{\cdot}{H^s(\calD)}$, and we denote by~$H_0^1(\calD)$ the space of functions in~$H^1(\calD)$ with zero trace on~$\partial \calD$ and by~$H^{-1}(\calD)$ its dual space. For~$s = 0$, we write~$L^p(\calD) := W^{0, p}(\calD)$, where~$L^2(\calD)$ is the space of Lebesgue square-integrable functions in~$\calD$ with inner product~$(\cdot, \cdot)_{\calD}$ and norm~$\Norm{\cdot}{L^2(\calD)}$. 
The superscript~$d$ will be used to indicate spaces of~$d$-vector-valued functions, as well as their seminorms and norms. 

Given a degree~$\ell \in \N$, we denote by~$\Pp{\ell}{\calD}$ the space of polynomials of degree at most~$\ell$ defined on~$\calD$. Moreover, the algebraic tensor product~$V\otimes W$ of two vector spaces~$V$ and~$W$ is given by
\begin{equation*}
    V \otimes W := \Span \{vw \ : \ v \in V \text{ and } w \in W\}.
\end{equation*}

In addition, given a nonempty interval~$(a, b) \subset \R$, a Banach space~$(Z, \Norm{\cdot}{Z})$, and~$1 \le p \le \infty$, we consider the corresponding Bochner space
\begin{equation*}
L^p(a, b; Z) := \Big\{v: (a, b) \to Z \ : \  \text{$v$ is strongly measurable \,and\, }  \Norm{v}{L^p(a, b; Z)} < \infty \Big\},
\end{equation*}
where
\begin{equation*}
\Norm{v}{L^p(a, b; Z)} := \begin{cases}
\displaystyle \Big(\int_a^b \Norm{v(t)}{Z}^p \dt\Big)^{1/p} & \text{ if } 1 \le p < \infty,
\\
\esssup_{t \in (a, b)} \Norm{v(t)}{Z} & \text{ if } p = \infty.
\end{cases}
\end{equation*}
Analogously, if~$\Seminorm{\cdot}{Z}$ is a specific seminorm in~$Z$, we denote by~$\Seminorm{\cdot}{L^p(a, b; Z)}$ the corresponding seminorm in~$L^p(a, b; Z)$. Finally, for an integer~$m \geq 0$, we define
\begin{equation*}
    W^{m, p}(a, b; Z) := \Big\{v: (a, b) \to Z \ : \ \dpt^{(i)} v \in L^p(a, b; Z), \text{ for~$i = 0, \ldots, m$} \Big\},
\end{equation*}
where~$\dpt^{(i)}v$ denotes the~$i$th time derivative of~$v$. The characteristic function on~$(a, b)$ is denoted by~$\chi_{(a, b)}$.

\paragraph{Outline.} Section~\ref{sec:model-problems} introduces the continuous setting for the heat and the wave equations, which serve as our model problems. Section~\ref{sec:preliminaries} establishes the notation, discrete spaces, and operators required for the analysis. In particular, Section~\ref{sec:structure-theory} outlines the core strategy used for the stability and convergence proofs. Since Section~\ref{sec:preliminaries} provides the necessary mathematical foundation, the subsequent sections are self-contained and can be read independently.
\section{Model problems} 
\label{sec:model-problems}
To keep the presentation simple, we consider two prototypical model problems, namely the heat equation and the second-order wave equation; however, extensions to more complex parabolic and hyperbolic problems can also be obtained. For instance, this approach has been used to study space--time methods for the incompressible Navier--Stokes equations in~\cite{Beirao_Gomez_Dassi:2025} and for the Westervelt quasilinear wave equation in~\cite{Gomez-Nikolic:2025}, where a careful analysis evidenced the robustness of the schemes with respect to the model parameters.

In what follows, $\Omega \subset \R^d$ ($d \in \{1, 2, 3\}$) will be a polytopic domain with boundary~$\partial \Omega$. Moreover, for a prescribed final time~$T > 0$, we define the space--time cylinder~$\QT := \Omega \times (0, T)$, and the surfaces~$\SO := \Omega \times \{0\}$ and~$\ST := \Omega \times \{T\}$.

\subsection{The heat equation}
\label{sec:model-heat}
Given a positive diffusion coefficient~$\nu > 0$, a source term~$f : \QT \to \R$, and an initial datum~$u_0 : \Omega \to \R$, we consider the following initial and boundary value problem (IBVP): find  $u : \QT \to \R$ such that
\begin{subequations}
\label{eq:heat-equation}
\begin{alignat}{3}
\dpt u - \nu \Delta u & = f & & \qquad \text{ in } \QT, \\
u & = 0 & & \qquad \text{ on } \partial \Omega \times (0, T), \\
u & = u_0 & & \qquad \text{ on } \SO.
\end{alignat}
\end{subequations}

Since we are interested in analyzing the DG and CG time discretizations of~\eqref{eq:heat-equation}, with stability and \emph{a priori} error constants that do not degenerate for small diffusion coefficients~(i.e., for~$0 < \nu \ll 1$), we make the following assumptions on the data of problem~\eqref{eq:heat-equation}:
\begin{equation}
\label{eq:regularity-data-heat}
f \in L^2(\QT) \quad \text{ and } \quad u_0 \in L^2(\Omega).
\end{equation}

Defining the Bochner--Sobolev spaces
\begin{alignat*}{3}
X & := \{u \in L^2(0, T; H_0^1(\Omega)) \ : \ \dpt u \in L^2(0, T; H^{-1}(\Omega)) \} \quad \text{ and } \quad 
Y & := L^2(0, T; H_0^1(\Omega)),
\end{alignat*}
we consider the following standard Petrov--Galerkin weak formulation of~\eqref{eq:heat-equation}: find~$u \in X$ such that
\begin{subequations}
\label{eq:weak-formulation-heat}
\begin{alignat}{3}
\label{eq:weak-formulation-heat-1}
b(u, v) := \langle \dpt u, v \rangle_{Y', Y} + \nu (\nabla u, \nabla v)_{\QT} & =  (f, v)_{\QT}  & & \quad \forall v \in Y, \\
\label{eq:weak-formulation-heat-2}
u(\cdot, 0) & = u_0(\cdot) & & \quad \text{in~$L^2(\Omega)$}, 
\end{alignat}
\end{subequations}
where~$\langle \cdot, \cdot\rangle_{Y', Y}$ denotes the dual pairing between~$Y'$ and~$Y$, and the continuous embedding~$X \hookrightarrow C^0([0, T]; L^2(\Omega))$ (see, e.g., \cite[Thm.~3(i) in \S5.9.2]{Evans:2010}) guarantees that~\eqref{eq:weak-formulation-heat-2} makes sense.

The idea at the core of this analysis is that, for~$f \in L^2(\QT)$, choosing~$v = \chi_{(0, t)} u$ in the weak formulation~\eqref{eq:weak-formulation-heat-1} gives
\begin{alignat}{3}
\label{eq:aux-test}
\int_0^t \langle \dpt u, u \rangle_{H^{-1}(\Omega), H_0^1(\Omega)}\dt + \nu \int_0^t (\nabla u, \nabla u)_{\Omega} \dt = \int_0^t (f, u)_{\Omega} \dt,
\end{alignat}
which, combined with integration by parts in time, equation~\eqref{eq:weak-formulation-heat-2}, and the H\"older and the Young inequalities, leads to
\begin{alignat}{3}
\nonumber
\frac12 \Norm{u(\cdot, t)}{L^2(\Omega)}^2 + \nu \Norm{\nabla u}{L^2(0, t; L^2(\Omega)^d)}^2 & = \int_0^t (f, u)_{\Omega} \dt + \frac12 \Norm{u_0}{L^2(\Omega)}^2 \\
\nonumber
& \le \Norm{f}{L^1(0, t; L^2(\Omega))} \Norm{u}{L^{\infty}(0, t; L^2(\Omega))} + \frac12 \Norm{u_0}{L^2(\Omega)}^2.
\end{alignat}
Then, one can easily deduce the following continuous dependence of~$u$ on the data:
\begin{alignat}{3}
\label{eq:robust-continuous-dependence-heat}
\frac14 \Norm{u}{C^0([0, T]; L^2(\Omega))}^2 + \nu \Norm{\nabla u}{L^2(\QT)^d}^2  \le 2 \Norm{f}{L^1(0, T; L^2(\Omega))}^2 + \Norm{u_0}{L^2(\Omega)}^2.
\end{alignat}

\begin{remark}[Robust continuous dependence on the data]
Bound~\eqref{eq:robust-continuous-dependence-heat} has several important features.
\begin{itemize}
    \item It provides a uniform bound on~$u$ in the stronger norm~$\Norm{u}{C^0([0, T]; L^2(\Omega))}$.
    In fact, bound~\eqref{eq:robust-continuous-dependence-heat} holds even for pure homogeneous Neumann boundary conditions, for which the corresponding weak formulation is well posed (see~\cite[\S4.7.2 in Ch.~3]{Lions_Magenes-I:1972}).

    \item The stability constant is independent of the final time~$T$ and the diffusivity parameter~$\nu$. In contrast, using a Gr\"onwall estimate would lead to a constant that grows exponentially with respect to~$T$. 
    Some works in the literature of DG time discretizations employ Gr\"onwall estimates; however, this pessimistic artifact of the analysis can be avoided in many cases.

    \item An alternative approach is to bound the first term on the right-hand side of~\eqref{eq:aux-test} using the Cauchy--Schwarz inequality and the Poincar\'e--Steklov inequality in space (with constant~$C_P$) as follows:
    \begin{alignat}{3}
    \label{eq:ugly-continuous-dependence-heat}
    \int_0^t (f, u)_{\Omega} \dt & \le \Norm{f}{L^2(\QT)} \Norm{u}{L^2(\QT)}
    & \le C_P\Norm{f}{L^2(\QT)} \Norm{\nabla u}{L^2(\QT)^d},
    \end{alignat}
    but this leads to a stability constant proportional to~$\nu^{-1}$. 
\end{itemize}
These points motivate the search for discrete analogues of bound~\eqref{eq:robust-continuous-dependence-heat}.
As discussed in Sections~\ref{sec:Jamet-heat} and~\ref{sec:Aziz-Monk-heat}, reproducing~\eqref{eq:robust-continuous-dependence-heat} at the discrete level requires a more delicate analysis, whereas obtaining a discrete analogue of~\eqref{eq:ugly-continuous-dependence-heat} is straightforward.
\eremk 
\end{remark}

\begin{remark}[Additional regularity on~$f$]
It is well known that the continuous weak formulation of~\eqref{eq:heat-equation} is well posed for~$f \in L^2(0, T; H^{-1}(\Omega))$, which can be proven using either the Faedo--Galerkin approach (see, e.g., \cite[\S65.2]{Ern_Guermond:2021}) or the inf--sup stability approach (see~\cite[Thm.~5.1]{Schwab_Stevenson:2009}, \cite[Thm.~2.1]{Steinbach:2015}, and~\cite[\S71.1]{Ern_Guermond:2021}). However, in many applications, the focus is on the robustness of the scheme for vanishing parameters rather than on requiring minimal regularity. This is the case, for instance, in convection-dominated problems, the incompressible Navier--Stokes equations at high Reynolds numbers, and other models in electromagnetism and fluid dynamics. Consequently, the additional regularity on~$f$ in~\eqref{eq:regularity-data-heat} is used to obtain stability bounds and a priori error estimates that are robust for small diffusion~$(0 < \nu \ll 1)$. 

We assume that~$f \in L^2(\QT)$, whereas the stability bound~\eqref{eq:robust-continuous-dependence-heat} only involves the~$L^1(0,T; L^2(\Omega))$ norm of~$f$. This property has a discrete counterpart, as the well-posedness of the DG and CG time discretizations of the heat equation only requires~$f \in L^1(0, T; L^2(\Omega))$. This also allows us to obtain a priori error estimates in terms of~$W^{r, 1}(0, T; \mathcal{X})$ norms of the continuous solution~$u$ rather than the standard~$H^r(0, T; \mathcal{X})$ norms. A similar situation holds for the wave equation (see~\eqref{eq:continuous-dependence-wave}).
\eremk
\end{remark}

\begin{remark}[Space--time methods]
The larger class of space--time methods for the heat equation includes conforming finite element~\cite{Steinbach:2015,Steinbach-Zank:2020}, wavelet~\cite{Schwab_Stevenson:2009}, minimal residual~\cite{Stevenson_Westerdiep:2021}, discontinuous Galerkin~\cite{Cangiani_Dong_Georgoulis:2017,Gomez_Perinati_Stocker:2026}, least-squares~\cite{Fuhrer_Karkulik:2021,Gantner_Stevenson:2024,Fuhrer_Gantner:2025}, virtual element~\cite{Gomez_Mascotto_Moiola_Perugia:2024,Gomez_Mascotto_Perugia:2024}, isogeometric~\cite{Langer_Moore_Neumuller:2016}, and discontinuous Petrov--Galerkin~\cite{Diening_Storn:2022} methods.
\eremk
\end{remark}

\subsection{The wave equation}
\label{sec:model-wave}
Given a source term~$f \in L^2(\QT)$, initial data~$u_0 \in H_0^1(\Omega)$ and~$v_0 \in L^2(\Omega)$, and a positive propagation speed~$c > 0$, we consider the following second-order-in-time IBVP: find~$u : \QT \to \R$ such that
\begin{subequations}
\label{eq:second-order-wave}
\begin{alignat}{3}
\dptt u - c^2 \Delta u & = f & & \quad \text{in~$\QT$},\\
u & = 0 & & \quad \text{on~$\partial \Omega \times (0, T)$},\\
u = u_0 \ \text{ and } \ \dpt u & = v_0  & & \quad \text{on~$\SO$.}
\end{alignat}
\end{subequations}

We are interested in weak solutions to~\eqref{eq:second-order-wave} in the \emph{energy class} (see~\cite[\S4 in Ch.~4]{Ladyzhenskaya:1985}), i.e., a function~$u \in C^0([0, T]; H_0^1(\Omega))$ with~$\dpt u \in C^0([0, T]; L^2(\Omega))$ and~$\dptt u \in L^2(0, T; H^{-1}(\Omega))$ satisfying
\begin{subequations}
\label{eq:second-order-weak-formulation-wave}
\begin{alignat}{3}
\langle \dptt u, w \rangle_{Y', Y}  + c^2( \nabla u, \nabla w)_{\QT} & = (f, w)_{\QT} && \quad \forall w \in Y, \\
u(\cdot, 0) & = u_0(\cdot) & & \quad \text{in~$H_0^1(\Omega)$, }\\
\dpt u(\cdot, 0) & = v_0(\cdot) & & \quad \text{in~$L^2(\Omega)$},
\end{alignat}
\end{subequations}
where~$\langle \cdot, \cdot \rangle_{Y', Y}$ is the same dual pairing as in~\eqref{eq:weak-formulation-heat}.

The following continuous dependence on the data can be deduced as in~\eqref{eq:robust-continuous-dependence-heat} for the heat equation:
\begin{equation}
\label{eq:continuous-dependence-wave}
\begin{split}
\frac14 \Big( \Norm{\dpt u}{C^0([0, T]; L^2(\Omega))}^2 & + c^2 \Norm{\nabla u}{C^0([0, T]; L^2(\Omega)^d)}^2 \Big) \\
& \le 2 \Norm{f}{L^1(0, T; L^2(\Omega))}^2 + c^2 \Norm{\nabla u_0}{L^2(\Omega)^d}^2 + \Norm{v_0}{L^2(\Omega)}^2.
\end{split}
\end{equation}

The second-order formulation~\eqref{eq:second-order-wave} can also be rewritten as a first-order-in-time system as follows: find~$u : \QT \to \R$ and~$v : \QT \to \R$ such that
\begin{subequations}
\label{eq:first-order-wave}
\begin{alignat}{3}
v & = \dpt u & & \quad \text{in~$\QT$}, \\
\dpt v - c^2 \Delta u & = f & & \quad \text{in~$\QT$},\\
u & = 0 & & \quad \text{on~$\partial \Omega \times (0, T)$},\\
u = u_0 \ \text{ and } \ v & = v_0  & & \quad \text{on~$\SO$,}
\end{alignat}
\end{subequations}
which is commonly referred to as the \emph{Hamiltonian formulation} of the wave equation. 

We shall consider as a weak solution to~\eqref{eq:first-order-wave} the pair~$(u, v)$ with~$u$ being the weak solution to~\eqref{eq:second-order-weak-formulation-wave} and~$v = \dpt u \in C^0([0, T]; L^2(\Omega))$. 

\begin{remark}[Well-posedness of~\eqref{eq:second-order-weak-formulation-wave}]
The existence of a unique solution to~\eqref{eq:second-order-weak-formulation-wave} has been established in~\cite[Thms.~8.1 and 8.2 in Ch. 3]{Lions_Magenes-I:1972}, \cite[Thm.~29.1 in Ch. V]{Wloka:1987}, and~\cite[Thm.~24A in Ch. 24]{Zeidler:1990}.
The corresponding result for the weak formulation of~\eqref{eq:second-order-wave} with~$f \in L^1(0, T; L^2(\Omega))$ can be found in~\cite[Thm.~4.2 in Ch. IV]{Ladyzhenskaya:1985}  and~\cite[Thm.~2.2 in Part I]{Lasiecka_Triggiani:1994}.
\eremk
\end{remark}

\begin{remark}[Nonhomogeneous Dirichlet boundary conditions]
When~$H^1(\Omega)$-conforming approximations in space are combined with Galerkin-type time discretizations, as we do in this work, the strong imposition of nonhomogeneous Dirichlet boundary conditions in the fully discrete setting is nontrivial, but this issue is often neglected in the literature. 
In particular, to avoid a loss of accuracy, the discrete liftings of the Dirichlet datum must be compatible with the time discretization; see, e.g., \cite{Walkington:2014}, \cite[Ch.~3]{Voulis:2019}, and~\cite{Gomez:2025}. 
Further insight into this situation is given in Remark~\ref{rem:nonhomogeneous-Dirichlet-boundary} below.
\eremk
\end{remark}

\begin{remark}[Space--time methods]
The larger class of space--time methods for the wave equation includes 
collocation~\cite{Anselmann_Bause_Becher_Matthies:2020},
Morawetz-multiplier-based~\cite{Bignardi_Moiola:2025}, 
isogeometric~\cite{Ferrari_Perugia:2025,Ferrari_Fraschini:2026,Ferrari_Perugia_Zampa:2025,Ferrari_Fraschini_Loli_Perugia:2025}, discontinuous Galerkin~\cite{Shukla_Vegt:2022},
least-squares~\cite{Fuhrer_Gonzalez_Karkulik:2025}, mixed finite element~\cite{Bause_Radu_Kocher:2017},
Trefftz-like~\cite{Banjai_Georgoulis_Lijoka:2017,Imbert_Moiola_Stocker:2023,Moiola_Perugia:2018,Kretzschmar_Moiola_Perugia:2016}, 
and ultraweak~\cite{Henning_Palitta_Simoncini_Urban:2022} methods.
\eremk
\end{remark}
\section{Preliminaries}
\label{sec:preliminaries}
In this section, we introduce the notation, tools, and main results used for the formulation and analysis of the methods considered. 
\subsection{Mesh notation and discrete spaces}
\label{sec:DG-notation}
\paragraph{Space and time meshes.} Let~$\Th$ be a shape-regular, conforming simplicial partition of the spatial domain~$\Omega$, where~$h := \max_{K \in \Th} \diam(K)$ is the meshsize of~$\Th$.

Let also~$\Tt$ be a partition of the time interval~$(0, T)$ defined by the nodes
$$0 := t_0 < t_1 < \ldots < t_{\Nt} := T,$$
on which we do not assume any local or global quasi-uniformity.

For~$n = 1, \ldots, \Nt$, we define the time interval~$\In := (\tnmo, \tn)$, the time step~$\tau_n := \tn - \tnmo$, and the time slab~$\Qn := \Omega \times \In$. Moreover, for~$n = 1, \ldots, \Nt - 1$, we define the surface~$\Sn := \Omega \times \{\tn\}$. The maximum time step in~$\Tt$ is denoted by~$\tau$.

For any nonnegative integer~$m \geq 0$, $p \in [1, \infty]$, and any Banach space~$(Z, \Norm{\cdot}{Z})$, we also define the following broken Bochner spaces:
\begin{alignat*}{3}
W^{m, p}(\Tt; Z) := \big\{v : (0, T) \to Z  \ : \ v|_{\In} \in W^{m, p}(\In; Z), \ n =1, \ldots, N\big\}.
\end{alignat*}

Henceforth, we denote by~$\Id$ the identity operator in space and time. We write~$a \lesssim b$ to indicate the existence of a positive constant~$C$ that  is independent of the \textbf{meshsize~$h$}, the \textbf{maximum time step~$\tau$}, the \textbf{final time~$T$}, and the \textbf{model parameters}. Moreover, we write~$a \simeq b$ to denote that~$a \lesssim b$ and~$b \lesssim a$. 

\paragraph{Discrete spaces.} Let~$p \in \N$ with~$p \geq 1$, and~$q \in \N$ be given degrees of approximation in space and time, respectively.  

We define the following discrete spaces:
\begin{alignat*}{3}
\oVhp & := \big\{\vh \in H_0^1(\Omega) \ : \ \vh{}_{|_{K}}  \in \Pp{p}{K} \text{ for all~$K \in \Th$}\big\}, \\
\oVdisc{q} & := \big \{\vht \in L^2(0, T; H_0^1(\Omega)) \ : \ \vht{}_{|_{\Qn}} \in \Pp{q}{\In} \otimes \oVhp, \text{ for~$n = 1, \ldots, N$}\big\}, \\
\oVcont{q} & := \big \{\vht \in C^0([0, T]; H_0^1(\Omega)) \ : \ \vht{}_{|_{\Qn}} \in \Pp{q}{\In} \otimes \oVhp, \text{ for~$n = 1, \ldots, N$} \big\}.
\end{alignat*}
The notation for~$\oVdisc{q}$ and~$\oVcont{q}$ highlights the regularity and degree of approximation in time, as the same $H_0^1(\Omega)$-conforming spatial discretization is used for the test and trial spaces of all methods considered.

We denote by~$\Pp{q}{\Tt}$ the broken space of polynomials of degree at most~$q$ on each~$\In \in \Tt$, and by~$\Pcont{q}{\Tt} := C^0[0, T] \cap \Pp{q}{\Tt}$. 

\paragraph{Time jumps.} 
For~$n = 1, \ldots, \Nt - 1$, we define the time jumps~$(\jump{\cdot}_n)$ of a piecewise smooth function in time~$v$ as
\begin{equation*}
\jump{v}_n := v(\tn^+) - v(\tn^-),
\end{equation*}
where
\begin{equation*}
v(\tn^+) := \lim_{\varepsilon \to 0^+} v(\tn + \varepsilon) \quad \text{ and } \quad v(\tn^-) := \lim_{\varepsilon \to 0^+} v(\tn - \varepsilon).
\end{equation*}

The following identities will be extensively used in the analysis of DG time discretizations:
\begin{subequations}
\begin{alignat}{3}
\label{eq:jumps-in-time-identity-1}
v(\tn^+) \jump{w}_n + w(\tn^-) \jump{v}_n & = \jump{vw}_n, \\
\label{eq:jumps-in-time-identity-2}
-\frac12 \jump{v^2}_n + v(\tn^+) \jump{v}_n & = \frac12 \jump{v}_n^2, \\
\label{eq:jumps-in-time-identity-3}
-\frac12 v(\tn^+)^2 + v(\tn^+) \jump{v}_{n}  & =  \frac12 \jump{v}_{n}^2 - \frac12 v( \tn^-)^2,
\end{alignat}
\end{subequations}
where~$\eqref{eq:jumps-in-time-identity-3}$ is just a rewriting of~\eqref{eq:jumps-in-time-identity-2}.

We introduce the following seminorm in~$\oVdisc{q}$ associated with the time jumps:
\begin{equation*}
\Seminorm{\vht}{\sf J}^2 := \Norm{\vht}{L^2(\ST)}^2 + \sum_{n = 1}^{\Nt - 1} \Norm{\jump{\vht}_n}{L^2(\Omega)}^2 + \Norm{\vht}{L^2(\SO)}^2.
\end{equation*}
\subsection{A polynomial inverse estimate}
Next lemma is an inverse estimate for functions that are polynomials in time, and is one of the main ingredients used to obtain stability bounds in~$L^{\infty}(0, T; \mathcal{X})$ norms. 
This result in Bochner spaces is an extension of the analogous result in one dimension, and was proven in~\cite[Lemma 3.1]{Beirao_Gomez_Dassi:2025}.

\begin{lemma}[Inverse estimate]
\label{lemma:L2-Linfty}
For any~$q \in \N$ and any Banach space~$(Z, \Norm{\cdot}{Z})$, there exists a positive constant~$\Cinv$ depending only on~$q$ such that, for~$n = 1, \ldots, N$, it holds
\begin{equation*}
\Norm{w_{\tau}}{L^{\infty}(\In; Z)}^2 \le \Cinv \tau_n^{-1} \Norm{w_{\tau}}{L^2(\In; Z)}^2 \qquad \forall w_{\tau} \in \Pp{q}{\Tt} \otimes Z.
\end{equation*}
\end{lemma}

\begin{remark}[$q$-dependence of~$\Cinv$]
The proof in~\cite[Lemma 3.1]{Beirao_Gomez_Dassi:2025} yields~$\Cinv = (q+1)^3$ (the extra factor~$1/2$ therein is due to a small typo in the proof). On the other hand, the classical~$q$-explicit inverse estimate in one dimension from~\cite[eq.~(3.6.4)]{Schwab:1998} gives~$\Cinv = 32q^2$, which exhibits better asymptotic behavior but remains larger than~$(q+1)^3$ for~$q \le 28$. In any case, we focus on the~$(h,\tau)$-convergence of the method, as a~$q$-explicit analysis requires higher temporal regularity of the data to avoid further suboptimality in~$q$ (see~\cite[Rem.~1]{Dong-Mascotto-Wang:2025}).
\eremk
\end{remark}

\subsection{Auxiliary weight functions} 
\label{sec:auxiliary-weight}
The second main ingredient of the analysis is the use of some linear weight functions, which were implicitly introduced by~Walkington in~\cite[Cor.~4.4]{Walkington:2014}, and have been subsequently employed in~\cite{Dong-Mascotto-Wang:2025,Gomez-Nikolic:2025,Gomez:2025,Zhang_Dong_Yan:2025} to prove continuous dependence on the data in~$L^{\infty}(0, T; \mathcal{X})$ norms for linear and quasilinear wave problems. 

To the best of our knowledge, these auxiliary weight functions were used for the first time in the context of a parabolic-type problem in~\cite{Beirao_Gomez_Dassi:2025}, for the incompressible Navier--Stokes equations.

For~$n = 1, \ldots, N$, we define
\begin{equation}\label{lambdadef}
\varphi_n(t) := 1 - \lambda_n (t - \tnmo) \quad \text{ with } \lambda_n := \frac{1}{2\tau_n},
\end{equation}
which satisfies the following uniform bounds:
\begin{alignat*}{3}
\frac12 & \le \varphi_n(t)  \le 1 & & \qquad \forall t \in [\tnmo, \tn], \\
\varphi_n'(t) & = - \lambda_n & & \qquad \forall t \in [\tnmo, \tn].
\end{alignat*}
All the special test functions employed in the stability analysis in the subsequent sections involve these auxiliary weight functions; see Table~\ref{tab:special-test-functions} below.

The following identities will be extensively used.
\begin{lemma}[Identities involving~$\varphi_n$]
\label{lemma:identities-varphi-integration}
Let~$\HH$ be a Hilbert space with inner product~$(\cdot, \cdot)_{\HH}$. Then, for~$n = 1, \ldots, N$, the following identities hold:\ \footnote{\label{footnote:n=1}For~$n = 1$, identity~\eqref{eq:identities-varphi-integration-1} is replaced in the analysis by~$\int_{I_1} (\dpt v_{\tau}, \varphi_1 v_{\tau})_{\HH} \dt = \frac14 \Norm{v_{\tau}(t_1^-)}{\HH}^2 - \frac12 \Norm{v_{\tau}(0)}{\HH}^2 + \frac{\lambda_1}{2}\Norm{v_{\tau}}{L^2(I_1; \HH)}^2$. Only minor changes are needed in the arguments for this case, so we omit the details.}
\begin{subequations}
\begin{alignat}{3}
\nonumber
\int_{\In} (\dpt v_{\tau}, & \varphi_n v_{\tau})_{\HH} \dt + (\jump{v_{\tau}}_{n-1}, \varphi_n(\tnmo) v_{\tau}(\tnmo^+))_{\HH} \\
\nonumber
& = \frac14 \Norm{v_{\tau}(\tn^-)}{{\HH}}^2 + \frac12 \Norm{\jump{v_{\tau}}_{n-1}}{{\HH}}^2 - \frac12 \Norm{v_{\tau}(\tnmo^-)}{{\HH}}^2 \\
\label{eq:identities-varphi-integration-1}
& \quad + \frac{\lambda_n}{2} \Norm{v_{\tau}}{L^2(\In; {\HH})}^2 & & \quad \forall v_{\tau} \in \Pp{q}{\Tt} \otimes \HH, \\
\nonumber
\int_{\In} (\dpt v_{\tau}, & \varphi_n v_{\tau})_{{\HH}} \dt \\
\label{eq:identities-varphi-integration-2}
& = \frac14 \Norm{v_{\tau}(\tn)}{{\HH}}^2 - \frac12 \Norm{v_{\tau}(\tnmo)}{{\HH}}^2 + \frac{\lambda_n}{2} \Norm{v_{\tau}}{L^2(\In; {\HH})}^2 & & \quad \forall v_{\tau} \in \Pcont{q}{\Tt} \otimes \HH.
\end{alignat}
\end{subequations}
\begin{proof}
We only prove~\eqref{eq:identities-varphi-integration-1}, as~\eqref{eq:identities-varphi-integration-2} follows analogously. 
We recall that~$\varphi_n(\tnmo) = 1$, $\varphi_n(\tn) = 1/2$, and~$\varphi_n'(t) = -\lambda_n$. Therefore, integrating by parts in time and using identity~\eqref{eq:jumps-in-time-identity-3},
we obtain
\begin{alignat*}{3}
\nonumber
\int_{\In} \big(\dpt v_{\tau}, & \varphi_n v_{\tau} \big)_{{\HH}}\dt  + \big(\jump{v_{\tau}}_{n-1}, v_{\tau}(\tnmo^+) \big)_{{\HH}} \\
\nonumber
& = \frac12 \int_{\In} \frac{\dd}{\dt} (\varphi_n v_{\tau}, v_{\tau})_{{\HH}} \dt  - \frac12 \int_{\In} \varphi_n' \Norm{v_{\tau}}{{\HH}}^2 \dt + \big(\jump{v_{\tau}}_{n-1}, v_{\tau}(\tnmo^+) \big)_{\Omega} \\
\nonumber
& = \frac14 \Norm{v_{\tau}(\tn^-)}{{\HH}}^2 - \Big(\frac12 v_{\tau}(\tnmo^+) - \jump{v_{\tau}}_{n - 1}, v_{\tau}(\cdot, \tnmo^+) \Big)_{{\HH}} + \frac{\lambda_n}{2} \Norm{v_{\tau}}{L^2(\In; {\HH})}^2 \\
\label{eq:mht-ItR}
& = \frac14 \Norm{v_{\tau}(\tn^-)}{{\HH}}^2 + \frac12 \Norm{\jump{v_{\tau}}_{n - 1}}{{\HH}}^2 - \frac12 \Norm{v_{\tau}(\tnmo^-)}{{\HH}}^2  + \frac{\lambda_n}{2} \Norm{v_{\tau}}{L^2(\In; {\HH})}^2,
\end{alignat*}
as desired. 
\end{proof}
\end{lemma}

\begin{remark}[Alternative tools in the literature]
Another approach for the analysis of the DG time discretization of parabolic problems in~$L^{\infty}(0, T; \mathcal{X})$ norms is the use of the ``discrete characteristic functions" introduced in~\cite[\S2.3]{Chrysafinos_Walkington:2006}, 
whose definition can be written in compact form using the left-sided Thom\'ee projection operator~$\tPtTh$ in Definition~\ref{def:Pt} below as follows: for a fixed~$t \in \In$, the discrete characteristic function of~$\phi\in \Pp{q}{\In}$ is given by~$\widetilde{\phi} = \tPtTh (\chi_{[\tnmo, t)} \phi)$. 

Such an approach was ``complemented" by the introduction of an exponential interpolant in~\cite[\S3.2]{Chrysafinos_Walkington:2010} for the analysis of the DG time discretization of the incompressible Stokes and Navier--Stokes equations. The exponential interpolant of a function~$\phi \in \Pp{q}{\In}$ can be defined in compact form as~$\overline{\phi} := \tPtTh(\exp(\frac{\tnmo - t}{\tau_n}) \phi )$.  

The auxiliary weight function~$\varphi_n$ in~\eqref{lambdadef} corresponds to the Taylor polynomial of degree~$1$ around~$t = \tnmo$ of the exponential function
$$\exp\Big(\frac{\tnmo - t}{2\tau_n}\Big).$$
This evidences its close relation with the aforementioned exponential interpolant.
\eremk
\end{remark}

\subsection{The Lagrange interpolant~\texorpdfstring{$\ItR$}{ItR}}
\label{sec:ItR}
For~$n = 1, \ldots, N$, we denote by~$\{(\omega_i^{(n)}, s_i^{(n)})\}_{i = 1}^{q + 1}$ the left-sided Gauss--Radau quadrature rule in the interval~$\In$ with nodes
$$\tnmo =: s_1^{(n)} < s_2^{(n)} < \ldots < s_{q + 1}^{(n)} < \tn,$$ 
and positive weights~$\{\omega_i^{(n)}\}_{i = 1}^{q + 1}$.
This quadrature rule is exact for polynomials of degree less than or equal to~$2q$. 

Let~$\{\hLag_{i}\}_{i = 1}^{q+1}$ be the Lagrange polynomials associated with the Gauss--Radau nodes in the reference element~$[-1, 1]$. For~$n = 1, \ldots, N$, we define the Lagrange polynomials~$\{\Lag_i^{(n)}\}_{i = 1}^{q+1}$ in the interval~$\In$ as follows:
\begin{equation*}
    \Lag_i^{(n)}(t) = \hLag_i(\hat{t}), \quad t = \frac{(1 - \hat{t})}{2} \tnmo + \frac{(1+\hat{t})}{2} \tn, \quad i = 1, \ldots, q+1.
\end{equation*}

For any Banach space~$(Z, \Norm{\cdot}{Z})$, we denote by~$\ItR : H^1(\Tt; Z) \to \Pp{q}{\Tt} \otimes Z$ the broken Lagrange interpolant at the Gauss--Radau nodes, i.e., for any~$v \in H^1(\Tt; Z)$, the interpolant~$\ItR v \in \Pp{q}{\Tt} \otimes Z$ is given by%
\begin{equation*}
    \ItR v_{|_{\Qn}} (t) := \sum_{i = 1}^{q+1} v(s_i^{(n)}) \Lag_i^{(n)}(t), \quad \text{ for~$n = 1, \ldots, N.$}
\end{equation*}

Next lemma is an immediate consequence of the exactness of the Gauss--Radau quadrature rule. 
\begin{lemma}[Orthogonality of~$\ItR$]
Let~$u_{\tau} \in \Pp{q+1}{\Tt} \otimes L^2(\Omega)$ and~$w_{\tau} \in \Pp{q-1}{\Tt} \otimes L^2(\Omega)$. Then, the Lagrange interpolant~$\ItR$ satisfies the following identity:
\begin{equation}
\label{eq:ItR-exact}
\begin{split}
(\ItR u_{\tau}, w_{\tau})_{\Qn} & = (u_{\tau}, w_{\tau})_{\Qn}, \quad \text{ for~$n = 1, \ldots, N$}.
\end{split}
\end{equation}
\end{lemma}

Next lemma concerns some stability properties of~$\ItR$. 
\begin{lemma}[Stability of~$\ItR$]
\label{lemma:stab-ItR}
For any Banach space~$(Z, \Norm{\cdot}{Z})$, there exists a positive constant~$\CSI$ depending only on~$q$ such that, for~$n = 1, \ldots, N$, it holds
\begin{equation}
\label{eq:stab-ItR-Linfty}
\Norm{\ItR w%
}{L^{\infty}(\In; Z)} \le \CSI \Norm{w}{L^{\infty}(\In; Z)} \qquad \forall w \in %
H^1(\In; Z).
\end{equation}
Moreover, for all~$w_{\tau} \in \Pp{q+1}{\Tt} \otimes Z$,
\begin{equation}
\label{eq:stab-ItR-L2}
\Norm{\ItR w_{\tau}}{L^2(\In; Z)} \le \CSI \Cinv^{1/2} \Norm{w_{\tau}}{L^2(\In; Z)},
\end{equation}
where~$\Cinv = \Cinv(q+1)$ is the constant in Lemma~\ref{lemma:L2-Linfty}.
\end{lemma}
\begin{proof}
Bound~\eqref{eq:stab-ItR-Linfty} was proven in~\cite[Lemma 3.3]{Beirao_Gomez_Dassi:2025}. As for~\eqref{eq:stab-ItR-L2}, we use the H\"older inequality, bound~\eqref{eq:stab-ItR-Linfty}, and the inverse estimate from Lemma~\ref{lemma:L2-Linfty} to obtain
\begin{alignat*}{3}
\Norm{\ItR w_{\tau}}{L^2(\In; Z)} \le \tau_n^{1/2} \Norm{\ItR w_{\tau}}{L^{\infty}(\In; Z)} \le \CSI \tau_n^{1/2} \Norm{w_{\tau}}{L^{\infty}(\In; Z)} \le \CSI \Cinv^{1/2} \Norm{w_{\tau}}{L^2(\In; Z)}. 
\end{alignat*}
\end{proof}
The constant~$\CSI$ in~\eqref{eq:stab-ItR-Linfty} is the Lebesgue constant for the left-sided Radau nodes in~$[-1, 1]$, which grows as~$\mathcal{O}(\sqrt{q+1})$ according to~\cite[Thm.~5.1]{Hager_Hou_Rao:2017}.

We conclude this section with the following property from~\cite[Lemma 3.2]{Beirao_Gomez_Dassi:2025}, which involves the auxiliary weight functions in Section~\ref{sec:auxiliary-weight}.
\begin{lemma}
\label{lemma:bilinear-form-varphi}
Let~$a_h : \oVhp \times \oVhp \to \R$ be a coercive bilinear form on~$\oVhp$. Then, for~$n = 1, \ldots, N$, it holds
\begin{alignat*}{3}
\int_{\In} a_h\big(\vht, \ItR (\varphi_n \vht)\big) \dt  \geq \frac12\int_{\In} a_h(\vht, \vht) \dt \geq 0  \qquad \forall \vht \in \oVdisc{q}.
\end{alignat*}
\end{lemma}

\subsection{The time reconstruction operator~\texorpdfstring{$\Rt$}{Rt}}
\label{sec:time-reconstruction}
A common tool in the \emph{a posteriori} error analysis of DG time discretizations is the time reconstruction operator~$\Rt : \oVdisc{q} \to \oVcont{q+1}$ introduced in~\cite[\S2.1]{Makridakis_Nochetto:2006}. There are several ways to define this operator~(see, e.g., \cite[Def. 69.5 and Rem. 69.9]{Ern_Guermond:2021} and~\cite[\S2.3]{Smears:2017}), and all of them are equivalent for discrete functions. It turns out that such an operator is also useful to simplify the presentation of DG time discretizations. 
\begin{definition}[Time reconstruction~$\Rt$]
\label{def:time-reconstruction}
Let~$q \in \N$ and~${\HH}$ be a Hilbert space with inner product~$(\cdot, \cdot)_{{\HH}}$.
For any~$v \in H^1(\Tt; {\HH})$, the time reconstruction~$\Rt v \in \Pcont{q+1}{\Tt} \otimes {\HH}$ satisfies
\begin{alignat*}{3}
\Rt v(0) & = v(0) & & \quad \text{in~${\HH}$}, \\
\Rt v(\tn) & = v(\tn^-) & & \quad \text{in~${\HH}$, for~$n = 1, \ldots, N - 1$,} \\
\nonumber
\int_{0}^T  (\dpt \Rt v, w_{\tau})_{{\HH}} \dt  & = \sum_{n = 1}^{\Nt} \int_{\In} (\dpt v, w_{\tau})_{{\HH}} \dt  + \sum_{n = 1}^{\Nt -1} (\jump{v}_n, & & w_{\tau}(\tn^+))_{{\HH}}  \\
& \quad + (v(0), w_{\tau}(0))_{{\HH}} & & \quad \forall w_{\tau} \in \Pp{q}{\Tt} \otimes {\HH}.
\end{alignat*}
\end{definition}
This operator then represents the DG discretization of the first-order time derivative. 

\begin{lemma}[Properties of~$\Rt$]
\label{lemma:properties-Rt}
Let~${\HH}$ be a Hilbert space with inner product~$(\cdot,\cdot)_{{\HH}}$. 
The following identity holds for all~$v \in H^1(0, T; {\HH})$:
\begin{equation}
\label{eq:identity-Rt-v}
\int_0^T (\dpt \Rt v, w_{\tau})_{{\HH}} \dt = \int_0^T (\dpt v, w_{\tau})_{{\HH}} \dt + (v(0), w_{\tau}(0))_{{\HH}} \quad \forall w_{\tau} \in \Pp{q}{\Tt} \otimes {\HH}.
\end{equation}

Moreover, for~$n \in \{1, \ldots, N\}$, $v_{\tau} \in \Pp{q}{\Tt} \otimes {\HH}$, and~$w_{\tau} = \chi_{(0, \tn)} v_{\tau} \in \Pp{q}{\Tt} \otimes {\HH}$, it holds
\begin{alignat}{3}
\label{eq:Rt-vht-jump}
\int_0^T (\dpt \Rt v_{\tau}, w_{\tau})_{{\HH}} \dt = \frac12 \Big(\Norm{v_{\tau}(\tn^-)}{{\HH}}^2 + \sum_{m = 1}^{n-1} \Norm{\jump{v_{\tau}}_{m}}{{\HH}}^2 + \Norm{v_{\tau} (0)}{{\HH}}^2\Big).
\end{alignat}
\end{lemma}
\begin{proof}
Identity~\eqref{eq:identity-Rt-v} follows from the definition of~$\Rt$ and the continuous embedding $H^1(0, T; {\HH}) \hookrightarrow C^0([0, T]; {\HH})$, which makes all the terms involving the time jump vanish. 

Without loss of generality, we prove~\eqref{eq:Rt-vht-jump} only for~$n = N$. We use integration by parts in time and identity~\eqref{eq:jumps-in-time-identity-2} to obtain
\begin{alignat*}{3}
\int_{0}^{T} (\dpt \Rt v_{\tau}, w_{\tau})_{{\HH}} \dt  & = \int_0^T (\dpt \Rt v_{\tau}, v_{\tau})_{{\HH}} \dt  \\
& = \sum_{n = 1}^{\Nt} \int_{\In} (\dpt v_{\tau}, v_{\tau})_{{\HH}} \dt + \sum_{n = 1}^{\Nt - 1} (\jump{v_{\tau}}_n, v_{\tau}(\tn^+))_{{\HH}} + \Norm{v_{\tau}(0)}{{\HH}}^2 \\
& = \frac12\big( \Norm{v_{\tau}(T)}{{\HH}}^2 + \Norm{v_{\tau}(0)}{{\HH}}^2 \big) + \sum_{n = 1}^{N-1} \Big[ \big(\jump{v_{\tau}}_n, v_{\tau}(\tn^+)\big)_{{\HH}} - \frac12 \jump{\Norm{v_{\tau}}{{\HH}}^2}_n \Big] \\
& = \frac12 \Big(\Norm{v_{\tau}(T)}{{\HH}}^2 + \sum_{n = 1}^{\Nt - 1} \Norm{\jump{v_{\tau}}_n}{{\HH}}^2 + \Norm{v_{\tau}(0)}{{\HH}}^2\Big),
\end{alignat*}
which completes the proof. 
\end{proof}

\subsection{A roadmap of projections in time}
\label{sec:projections-time}
We now present the definition and properties of the projection operators in time involved in the analysis of Galerkin-type time discretizations. In particular, we show how different equivalent definitions of such operators lead to additional properties. 

Next lemma states that, if a projection operator onto a space~$\Pp{q}{\In} \otimes Z$ is stable in the~$L^{\infty}(\In; Z)$ norm for~$n = 1, \ldots, N$, it also has optimal local convergence properties in time. As a consequence, showing the stability in such norms of the projection operators in this section plays a crucial role in the convergence analysis. 

\begin{lemma}[Estimates for a generic projector~$\mathbb{P}_{\tau}$] 
\label{lemma:estimates-Ptau}
Let~$r \in [1, \infty]$, $(Z, \Norm{\cdot}{Z})$ be a Banach space, and~$\mathbb{P}_{\tau} : W^{1,1}(0, T; Z) \to \Pp{q}{\Tt} \otimes Z$ be a projection operator that is stable in the~$L^{\infty}(\In; Z)$ norm for~$n = 1, \ldots, N$. Then, the following estimates hold:
\begin{subequations}
\begin{alignat}{3}
\label{eq:estimate-ItR}
\Norm{(\Id - \mathbb{P}_{\tau}) v}{L^r(\In; Z)} & \lesssim \tau_n^{s} \Norm{\dpt^{(s)} v}{L^r(\In; Z)} & & \quad \forall v \in W^{s, r}(\In; Z),\ 1 \le s \le q + 1, \\
\label{eq:estimate-generic-Ptau-2}
\Norm{\dpt (\Id - \mathbb{P}_{\tau}) v}{L^r(\In; Z)} & \lesssim \tau_n^{s-1} \Norm{\dpt^{(s)} v}{L^r(\In; Z)} & & \quad \forall v \in W^{s, r}(\In; Z),\ 1 \le s \le q + 1.
\end{alignat}
\end{subequations}
\end{lemma}
\begin{proof}
We first focus on~\eqref{eq:estimate-ItR} and consider an integer~$s$ with~$1 \le s \le q+1$, as the case of fractional~$s$ follows by interpolation. Using the triangle and the H\"older inequalities, and the stability of~$\mathbb{P}_{\tau}$, for any~$v_{\tau} \in \Pp{q}{\In} \otimes Z$, we get
\begin{alignat}{3}
\nonumber
\Norm{(\Id - \mathbb{P}_{\tau}) v}{L^r(\In; Z)} 
& \le \Norm{v - v_{\tau}}{L^r(\In; Z)} + \Norm{\mathbb{P}_{\tau} (v - v_{\tau})}{L^r(\In; Z)} \\
\nonumber
& \lesssim \Norm{v - v_{\tau}}{L^r(\In; Z)} + \tau_n^{1/r} \Norm{v - v_{\tau}}{L^{\infty}(\In; Z)}.
\end{alignat}
Choosing~$v_{\tau}$ as the averaged Taylor polynomial of~$v$ of order~$s$ (and degree~$s - 1$) defined in~\cite[App.~A]{Diening_Storn_Tscherpel:2023}, we have
\begin{alignat*}{3}
\Norm{v - v_{\tau}}{L^{\infty}(\In; Z)}  \lesssim \tau_n^{s-1} \Norm{\dpt^{(s)} v}{L^1(\In; Z)}, 
\end{alignat*}
which, together with the approximation properties of~$\Pp{q}{\In} \otimes Z$ from~\cite[Lemma~A.1(e)]{Diening_Storn_Tscherpel:2023} and the H\"older inequality (with~$1/r + 1/r' = 1$), yields
\begin{alignat*}{3}
\Norm{(\Id - \Pt) v}{L^r(\In; Z)} 
& \lesssim \tau_n^{s} \Norm{\dpt^{(s)} v}{L^r(\In; Z)} + \tau_n^{1/r + s - 1} \Norm{\dpt^{(s)} v}{L^{1}(\In; Z)} \lesssim \tau_n^{s} \Norm{\dpt^{(s)} v}{L^r(\In; Z)}.
\end{alignat*}
This completes the proof of~\eqref{eq:estimate-ItR}. As for~\eqref{eq:estimate-generic-Ptau-2}, we can use a polynomial inverse estimate and proceed as above to obtain
\begin{alignat*}{3}
\Norm{\dpt (\Id - \Pt) v}{L^r(\In; Z)} & \lesssim \Norm{\dpt(v - v_{\tau})}{L^r(\In; Z)} + \tau_n^{1/r-1} \Norm{\Pt (v - v_{\tau})}{L^{\infty}(\In; Z)} \\
&\lesssim \tau_n^{s - 1} \Norm{\dpt^{(s)} v}{L^r(\In; Z)} + \tau_n^{1/r+s-2} \Norm{\dpt^{(s)} v}{L^1(\In; Z)} \lesssim \tau_n^{s-1} \Norm{\dpt^{(s)} v}{L^r(\In; Z)}. 
\end{alignat*}
\end{proof}

\subsubsection{The~\texorpdfstring{$L^2(0, T)$}{L2(0,T)}-orthogonal projection~\texorpdfstring{$\Pi_q^t$}{Piqt}}
Let~$\{\widehat{\Leg}_i\}_{i \in \N}$ be the Legendre polynomials defined on the reference element~$[-1, 1]$. For~$n = 1, \ldots, N$, we define the Legendre polynomials~$\{\Leg_i^{(n)}\}_{i \in \N}$ in the interval~$\In$ as follows:
\begin{equation*}
    \Leg_i^{(n)}(t) = \widehat{\Leg}_i (\hat{t}), \quad t = \frac{(1 - \hat{t})}{2} \tnmo + \frac{(1 + \hat{t})}{2} \tn, \quad i \in \N. 
\end{equation*}
The following orthogonality property is then obtained:
\begin{equation}
\label{eq:orthogonality-Legendre}
    \int_{\In} \Leg^{(n)}_i(t) \Leg_j^{(n)} (t) \dt  = \frac{\tau_n}{2i+1} \delta_{ij} \qquad \forall i,j \in \N,
\end{equation}
where~$\delta_{ij}$ is the Kronecker delta function. 
The superscript~$(n)$ will be omitted whenever the dependence on the interval~$n$ is clear from the context. 
\begin{definition}[$L^2(0, T)$-orthogonal projection]
\label{def:L2-orthogonal}
For~$q \in \N$, and a Hilbert space~${\HH}$ with inner product~$(\cdot, \cdot)_{{\HH}}$, we define the~$L^2(0, T)$-orthogonal projection~$\Pi_q^t : L^2(0, T; {\HH}) \to \Pp{q}{\Tt} \otimes {\HH}$ as follows: for~$v \in L^2(0, T; {\HH})$ and~$n = 1, \ldots, N$,
\begin{alignat*}{3}
\int_{\In} \big((\Id - \Pi_q^t) v, \phi_q\big)_{{\HH}} \dt = 0 \qquad \forall \phi_q \in \Pp{q}{\Tt} \otimes {\HH}.
\end{alignat*}
\end{definition}

From this definition, the commutativity of~$(\cdot, \cdot)_{\HH}$ and the integral in time (see~\cite[Cor.~64.14]{Ern_Guermond:2021}), and the orthogonality in~\eqref{eq:orthogonality-Legendre} of the Legendre polynomials, we can obtain
\begin{equation*}
\begin{split}
\Big(\int_{\In} v(t) \Leg_j(t) \dt, z\Big)_{{\HH}} & = \int_{\In} (v(t), z \Leg_j(t))_{{\HH}} \dt \\
& = \int_{\In} (\Pi_q^t v (t), z \Leg_j(t))_{{\HH}} \dt = \frac{\tau_n}{2j+1} (v_j, z)_{{\HH}}, \quad \text{ for~$j = 0, \ldots, q$,}
\end{split}
\end{equation*}
where~$v_j \in {\HH}$ denotes the~$j$th coefficient of~$\Pi_q^t v$ in the Legendre basis.
This induces the following equivalent expression using the Legendre expansion:
\begin{equation}
\label{eq:Legendre-expansion-Piqt}
\Pi_{q}^t v (t) = \sum_{i = 0}^q v_i \Leg_i(t), \ \text{with } v_i = \frac{2i+1}{\tau_n} \int_{\In} v(t) \Leg_i(t) \dt \qquad \forall t \in \In,
\end{equation}
which can be used to extend the definition of~$\Pi_{q}^t$ to functions in~$L^1(0, T; {\HH})$, as~$\Leg_i \in L^{\infty}(\In)$ for all~$i \in \N$.  

The natural stability of~$\Pi_q^t$ in the~$L^2(0, T; {\HH})$ norm follows directly from Definition~\ref{def:L2-orthogonal} and the Cauchy--Schwarz inequality. We now prove that the projection~$\Pi_q^t $ is stable in the~$L^r(0, T; {\HH})$ norm for~$r \in [1, \infty]$. 

\begin{lemma}[Stability of~$\Pi_q^t$]
\label{lemma:stab-pi-time}
Let~$q \in \N$, $r \in [1, \infty]$, and~${\HH}$ be a Hilbert space with inner product~$(\cdot, \cdot)_{\HH}$. There exists a positive constant~$\CS$ depending only on~$q$ such that
\begin{equation*}
\Norm{\Pi_q^t v}{L^r(\In; {\HH})} \le \CS \Norm{v}{L^r(\In; {\HH})} \qquad \forall v \in L^r(\In; {\HH}).
\end{equation*}
\end{lemma}
\begin{proof}
We consider the expression in~\eqref{eq:Legendre-expansion-Piqt}. Using the Jensen inequality and the fact that all norms in~${\HH}$ are convex functions, for all~$t \in \In$, we get
\begin{equation*}
    \Norm{\Pi_q^t v(t)}{{\HH}} \le \sum_{i = 0}^q |\Leg_i(t)| \Norm{v_i}{{\HH}} \le \sum_{i = 0}^q \frac{2i+1}{\tau_n}|\Leg_i(t)| \int_{\In} \Norm{v(t)}{{\HH}} |\Leg_i(t)| \dt,
\end{equation*}
which, using the triangle inequality, the H\"older inequality (with~$1/r+1/r' = 1$), and the uniform bound~$\Norm{\Leg_i}{L^{\infty}(\In)} = 1$ for all~$i \in \N$, leads to
\begin{equation*}
    \begin{split}
    \Norm{\Pi_q^t v(t)}{L^r(\In; {\HH})} & \le \Norm{v}{L^r(\In; {\HH})}\sum_{i = 0}^q \frac{2i+1}{\tau_n} \Norm{\Leg_i}{L^{r}(\In)}  \Norm{\Leg_i}{L^{r'}(\In)} \\
    & \le \Norm{v}{L^r(\In; {\HH})} \sum_{i = 0}^q (2i+1) \tau_n^{1/r + 1/r' - 1}  =  (q+1)^2 \Norm{v}{L^r(\In; {\HH})},
    \end{split}
    \end{equation*}
    which completes the proof. 
\end{proof}

The following lemma is key to avoiding artificial constraints between the time step~$\tau$ and the mesh size~$h$ for CG time discretizations when deriving stability bounds in~$L^{\infty}(0, T; \mathcal{X})$ norms.
\begin{lemma}[Bound involving~$\varphi_n$]
\label{lemma:bound-pi-time-varphi_n}
Let~$q \in \N$ with~$q \geq 1$, and~${\HH}$ be a Hilbert space with inner product~$(\cdot, \cdot)_{\HH}$. Then, the  following bound holds for all $u_{\tau} \in \Pp{q}{\Tt} \otimes {\HH}$ and~$n = 1, \ldots, N$:
\begin{equation*}
-\int_{\In} (\varphi_n \dpt u_{\tau}, (\Id - \Pi_{q-1}^t) u_{\tau})_{{\HH}} \dt  \geq 0.
\end{equation*}
\end{lemma}
\begin{proof}
Since~$u_{\tau}{}_{|_{\Qn}} \in \Pp{q}{\In} \otimes {\HH}$, there exists a function~$\alpha \in {\HH}$ such that
\begin{alignat*}{3}
(\Id - \Pi_{q-1}^t) u_{\tau}(t) & = \alpha \Leg_{q}(t), \\
\dpt u_{\tau} (t) & = \dpt \Pi_{q-1}^t u_{\tau}(t) + \alpha \Leg_{q}'(t).
\end{alignat*}
Moreover, since~$\dpt \Pi_{q-1}^t u_{\tau} \in \Pp{q-2}{\In} \otimes {\HH}$ (for~$q \geq 2$) and~$\dpt \Pi_{q-1}^t u_{\tau} = 0$  (for~$q = 1$), we can use the orthogonality properties of~$\Pi_{q-1}^t$ and the definition of~$\varphi_n$ in~\eqref{lambdadef} to obtain
\begin{alignat*}{3}
\nonumber
- \int_{\In} (\varphi_n \dpt u_{\tau}, (\Id - \Pi_{q-1}^t) u_{\tau})_{{\HH}}  \dt 
& = - \int_{\In} \big((1 - \lambda_n(t - \tnmo)) (\dpt \Pi_{q-1}^t u_{\tau} + \alpha \Leg_{q}'), \alpha \Leg_{q}\big)_{{\HH}} \dt \\
\nonumber
& = \lambda_n\Norm{\alpha}{{\HH}}^2 \int_{\In} (t - \tnmo) \Leg_{q}' \Leg_{q} \dt  \\
& = \frac{q}{2(2q + 1)} \Norm{\alpha}{{\HH}}^2 \geq 0. 
\end{alignat*}
where, in the last line, we have used the identity
\begin{equation*}
\begin{split}
\int_{I_n} (t - \tnmo) \Leg_{q}(t) \Leg_{q}'(t) \dt & = \frac12 \int_{\In} (t - \tnmo) (\Leg_{q}^2(t))' \dt \\
& = \frac12 \int_{\In} ((t - \tnmo) \Leg_{q}^2(t))' \dt - \frac12 \Norm{\Leg_{q}}{L^2(\In)}^2  \\
& = \frac{\tau_n}{2} - \frac{\tau_n}{2(2q+1)}
= \frac{\tau_n q}{2q + 1}.
\end{split}
\end{equation*}
\end{proof}

\begin{remark}[Relation between~$\Pi_{q-1}^t$ and a Lagrange interpolant]
\label{rem:equivalence-L2-orthogonal-Legendre-interpolant}
In~\cite[p. 1788]{Karakashian-Makridakis:1999} and~\cite[Prop.~2.2]{Ferrari_Perugia_Zampa:2026}, 
it was shown that
$$\Pi_{q-1}^t v_q = \mathcal{I}_{\tau}^{\mathsf{L}(q-1)} v_q \qquad \forall v_q \in \Pp{q}{\In} \otimes \HH,$$
where~$\mathcal{I}_{\tau}^{\mathsf{L}(q-1)}$ is the Lagrange interpolant in time of degree~$q-1$ associated with the~$q$-point Gauss--Legendre nodes. This equivalence follows from the exactness of the Gauss--Legendre quadrature rule with~$q$ points for polynomials of degree up to~$2q-1$, and can be used to deduce interpolatory properties of~$\Pi_{q-1}^t$ for specific functions that are polynomials in time. This double perspective can be useful in the treatment of the nonstandard test functions used in the proposed framework (see Table~\ref{tab:special-test-functions} below).
\eremk
\end{remark}
\subsubsection{The Thom\'ee projection operator~\texorpdfstring{$\PtTh$}{PtTh}}
We now recall the definition of the Thom\'ee projection operator in~\cite[eq.~(12.9) in Ch.~12]{Thomee_book:2006}, which is a classical tool for the convergence analysis of DG time discretizations.
\begin{definition}[Projection~$\PtTh$]
\label{def:Pt}
For~$q \in \N$ and a Hilbert space~${\HH}$ with inner product~$(\cdot, \cdot)_{{\HH}}$, the projection operator~$\PtTh: H^1(\Tt; {\HH}) \rightarrow \Pp{q}{\Tt} \otimes {\HH}$ is defined for any~$v \in H^1(\Tt; {\HH})$ as follows:
\begin{subequations}
\begin{alignat}{3}
\label{eq:Pt-1}
\PtTh v (\tn^-) & = v(\tn^-) & & \quad \text{in~${\HH}$}, \\
\label{eq:Pt-2}
\int_{\In} \big( (\Id - \PtTh) v, \phi_{q - 1} \big)_{{\HH}} \dt  & = 0 & & \quad \forall \phi_{q - 1} \in \Pp{q - 1}{\In} \otimes {\HH},
\end{alignat}
\end{subequations}
for~$n = 1, \ldots, N$. If~$q = 0$, condition~\eqref{eq:Pt-2} is omitted. We further denote by~$\PtThL$ the left-sided projection operator obtained by replacing condition~\eqref{eq:Pt-1} by~$\PtThL v(\tnmo^+) = v(\tnmo^+)$ in~${\HH}$. 
\end{definition}

From~\eqref{eq:Pt-2}, one can deduce that~$\Pi_{q-1}^t \PtTh v = \Pi_{q-1}^t v$, which leads to the identity
\begin{equation}
\label{eq:PtTh-alternative}
    \PtTh v(t) = \Pi_{q-1}^t v (t) + \alpha \Leg_q(t) \quad \forall t \in \In,
\end{equation}
where~$\alpha \in {\HH}$ is given by~$\alpha = (v - \Pi_{q-1}^t v)(\tn)$ due to~\eqref{eq:Pt-1} and the fact that~$\Leg_i(\tn) = 1$ for all~$i \in \N$. This corresponds to the explicit expression of~$\PtTh v$ in~\cite[Rem. 69.17]{Ern_Guermond:2021}, and allows us to extend the definition of~$\PtTh$ to functions in~$W^{1, 1}(\Tt; {\HH})$. 

Next lemma is then a consequence of~$\eqref{eq:PtTh-alternative}$,  the stability of~$\Pi_{q-1}^t$ in the~$L^{\infty}(\In; {\HH})$ norm, and the continuous embedding~$W^{1,1}(\In; {\HH}) \hookrightarrow C^0(\overline{\In}; {\HH})$.  
\begin{lemma}[Stability of~$\PtTh$]
\label{lemma:stab-Pt}
Let~${\HH}$ be a Hilbert space with inner product~$(\cdot, \cdot)_{{\HH}}$. For all~$v \in W^{1, 1}(\In; {\HH})$ and~$n \in \{1, \ldots, N\}$, it holds
\begin{equation*}
\Norm{\PtTh v}{L^{\infty}(\In; {\HH})} \lesssim \Norm{v}{L^{\infty}(\In; {\HH})}.
\end{equation*}
\end{lemma}

The use of this projection operator in the analysis of DG time discretizations is motivated by the following orthogonality condition, which allows us to avoid coupling conditions between the meshsize~$h$ and the time step~$\tau$. The absence of jump terms on the right-hand side of~\eqref{eq:orthogonality-PtTh} is crucial for this, highlighting the relevance of choosing ``the right'' time projection.
\begin{lemma}[Orthogonality of~$\PtTh$]
\label{lemma:orthogonality-DG-PtTh}
Let~${\HH}$ be a Hilbert space with inner product~$(\cdot, \cdot)_{{\HH}}$. For all~$v \in W^{1,1}(0, T; {\HH})$, it holds
\begin{equation}
\label{eq:orthogonality-PtTh}
\int_0^T (\dpt \Rt \PtTh v, w_{\tau})_{{\HH}} \dt  = \int_0^T (\dpt v, w_{\tau})_{{\HH}} \dt + (v(0), w_{\tau}(0))_{{\HH}}  \qquad \forall w_{\tau} \in \Pp{q}{\Tt} \otimes {\HH}.
\end{equation}
\end{lemma}
\begin{proof} From Definition~\ref{def:time-reconstruction} of the time reconstruction operator~$\Rt$, integration by parts in time, identity~\eqref{eq:jumps-in-time-identity-1}, and the properties of~$\PtTh$, we get
\begin{alignat*}{3}
\int_0^T (\dpt \Rt \PtTh v, & w_{\tau})_{{\HH}} \dt \\
& = \sum_{n = 1}^N \int_{\In} (\dpt \PtTh v, w_{\tau})_{{\HH}} \dt + \sum_{n = 1}^{N-1} (\jump{\PtTh v}_n, w_{\tau}(\tn^+))_{{\HH}}  + (\PtTh v(0), w_{\tau}(0))_{{\HH}} \\ 
& = - \sum_{n = 1}^N \int_{\In} (\PtTh v, \dpt w_{\tau})_{{\HH}} \dt + \sum_{n = 1}^{N-1} \Big[ (\jump{\PtTh v}_n, w_{\tau}(\tn^+))_{{\HH}} - \jump{(\PtTh v, w_{\tau})_{{\HH}}}_n\Big] \\
& \quad + (\PtTh v(T), w_{\tau}(T))_{{\HH}} \\
& = - \sum_{n = 1}^N \int_{\In} (\PtTh v, \dpt w_{\tau})_{{\HH}} \dt - \sum_{n = 1}^{N-1} (\PtTh v(\tn^-), \jump{w_{\tau}}_{n})_{{\HH}} + (\PtTh v(T), w_{\tau}(T))_{{\HH}} \\
& = - \sum_{n = 1}^N \int_{\In} (v, \dpt w_{\tau})_{{\HH}} \dt - \sum_{n = 1}^{N-1} (v(\tn), \jump{w_{\tau}}_{n})_{{\HH}} + (v(T), w_{\tau}(T))_{{\HH}} \\
& = \int_0^T (\dpt v, w_{\tau})_{{\HH}} \dt + (v(0), w_{\tau}(0))_{{\HH}},
\end{alignat*}
where, in the last step, we have integrated by parts back in time.
\end{proof}

\begin{lemma}[Bound involving~$\Ptdqmo$]
\label{lemma:bound-Ptdqmo}
Let~$q \in \N$ with~$q \geq 2$, and~${\HH}$ be a Hilbert space with inner product~$(\cdot, \cdot)_{\HH}$. For all~$\wht \in \oVdisc{q-1}$, it holds
\begin{equation*}
- \int_{\In} (\wht, (\Id - \Ptdqmo) (\varphi_n \wht))_{{\HH}} \dt  \geq 0. 
\end{equation*}
\end{lemma}
\begin{proof}
We can write
\begin{equation*}
    \wht = \Pi_{q-2}^t \wht + \alpha \Leg_{q-1} \quad \text{ and } \quad \varphi_n \wht = \varphi_n \Pi_{q-2}^t \wht + \alpha \varphi_n \Leg_{q-1},
\end{equation*}
for some~$\alpha \in {\HH}$. Therefore, using the orthogonality properties of~$\Ptdqmo$ and the definition in~\eqref{lambdadef} of~$\varphi_n$, we get
\begin{alignat*}{3}
- \int_{\In} (\wht, (\Id - \Ptdqmo)(\varphi_n \wht))_{{\HH}} \dt  & = -\int_{\tnmo}^{\tn} ((\Id - \Pi_{q-2}^t) \wht, (\Id - \Ptdqmo)(\varphi_n \wht))_{{\HH}} \dt \\
& = \lambda_n \int_{\tnmo}^{\tn} (\alpha \Leg_{q-1}, \alpha (\Id - \Ptdqmo)((t - \tnmo) \Leg_{q-1}))_{{\HH}} \dt \\
& = \lambda_n \Norm{\alpha}{{\HH}}^2 \int_{\tnmo}^{\tn} \Leg_{q-1} (\Id - \Ptdqmo)((t - \tnmo)\Leg_{q-1}) \dt.
\end{alignat*}

It only remains to compute the time integral. Using the Legendre expansion, we have
\begin{equation*}
(t - \tnmo) \Leg_{q-1}(t) = \sum_{i = 0}^{q} \gamma_i \Leg_i(t) \quad\text{ with } \gamma_i = \frac{2i+1}{\tau_n} \int_{\tnmo}^{\tn} (t - \tnmo) \Leg_{q-1} \Leg_i \ds.
\end{equation*}
Moreover, one can easily prove that~$\Ptdqmo \Leg_q = - \Leg_{q-1}$. Therefore, 
\begin{equation*}
(\Id - \Ptdqmo)((t - \tnmo) \Leg_{q-1}) = \gamma_q (\Leg_q + \Leg_{q-1}),
\end{equation*}
which, together with the properties of the Legendre polynomials, yields
\begin{equation*}
\int_{\tnmo}^{\tn} \Leg_{q-1} (\Id - \Ptdqmo)((t - \tnmo)\Leg_{q-1}) \dt =  \frac{\tau_n \gamma_q}{2q-1} = \frac{\tau_n^2 q}{2(2 q - 1)^2} \geq 0,
\end{equation*}
which completes the proof.
\end{proof}

\begin{remark}[Relation between~$\PtThL$ and~$\ItR$]
\label{rem:equivalence-ItR-PtThL}
Similarly to Remark~\ref{rem:equivalence-L2-orthogonal-Legendre-interpolant}, one can use the exactness of the $(q + 1)$-point left-sided  Gauss--Radau quadrature rule for polynomials up to degree~$2q$ to show that
$$\PtThL v_{q+1} = \ItR v_{q+1} \qquad \forall v_{q+1} \in \Pp{q+1}{\In} \otimes \HH,$$
where~$\ItR$ is the Lagrange interpolant from Section~\ref{sec:ItR}. This equivalence can be used to deduce additional interpolatory properties of~$\PtThL$ for specific functions that are polynomials in time. This double perspective can be useful in the treatment of the nonstandard test functions used in the proposed framework (see Table~\ref{tab:special-test-functions} below).
\eremk
\end{remark}
\subsubsection{The Aziz--Monk projection operator~\texorpdfstring{$\PtAM$}{PtAM}}
We now focus on the auxiliary projection in~\cite[eq.~(2.9)]{Aziz_Monk:1989} and its stability properties.
\begin{definition}[Projection~$\PtAM$]
\label{def:Pt-AM}
Let~$q \in \N$ with~$q \geq 1$, and~${\HH}$ be a Hilbert space with inner product~$(\cdot, \cdot)_{{\HH}}$. The projection operator~$\PtAM : H^1(0, T; {\HH}) \to \Pcont{q}{\Tt} \otimes {\HH}$ is defined for any~$v \in H^1(0, T; {\HH})$ as follows:
\begin{subequations}
\begin{alignat}{3}
\label{def:Pt-1}
\PtAM v(0) & = v(0) & & \quad \text{in~${\HH}$},\\
\label{def:Pt-2}
\int_{\In} (\dpt (\Id - \PtAM) v, \phi_{q-1})_{{\HH}} \dt  & = 0 & & \quad \forall \phi_{q-1} \in \Pp{q-1}{\In} \otimes {\HH}, \text{ for } n = 1, \ldots, N.
\end{alignat}
\end{subequations}
\end{definition}

We now derive several equivalent (local) definitions of~$\PtAM$, which serve different purposes.
\paragraph{Alternative definition~I.}  Taking~$\phi_{q-1} = \chi_{(0, t_1)} z$ with~$z \in {\HH}$ in~\eqref{def:Pt-2}, and using~\eqref{def:Pt-1}, it can be deduced that~$\PtAM v(t_1) = v(t_1)$ in~${\HH}$. Recursively, one can also show that
\begin{equation}
\label{eq:pointwise-identity-Pt}
\PtAM v (\tn) = v(\tn) \ \text{in~${\HH}$,} \quad \text{ for }  n = 0, \ldots, N.
\end{equation}
Moreover, using~\eqref{eq:pointwise-identity-Pt} and integrating by parts in time in~\eqref{def:Pt-2}, we have
\begin{equation*}
\begin{split}
0 & = \int_{\In} (\dpt(\Id - \PtAM) v, \phi_{q-1})_{{\HH}} \dt \\
& = ((\Id - \PtAM)v(\tn), \phi_{q-1}(\tn))_{{\HH}}  - ((\Id - \PtAM)v (\tnmo), \phi_{q-1}(\tnmo))_{{\HH}}\\
& \quad  - \int_{\In} ((\Id - \PtAM)v, \dpt \phi_{q-1})_{{\HH}} \dt \\
& = - \int_{\In} ((\Id - \PtAM) v, \dpt \phi_{q-1})_{{\HH}} \dt  \qquad\qquad \forall \phi_{q-1} \in \Pp{q-1}{\In} \otimes {\HH},
\end{split}
\end{equation*}
whence, the projection~$\PtAM v$ can be equivalently defined (locally) on each~$\In \in \Tt$ as follows:
\begin{subequations}
\label{def:local-Pt}
\begin{alignat}{3}
\label{def:local-Pt-1}
\PtAM v (\tnmo) & = v(\tnmo) & & \quad \text{in~${\HH}$}, \\
\label{def:local-Pt-2}
\PtAM v (\tn) & = v(\tn) & & \quad \text{in~${\HH}$},\\
\label{def:local-Pt-3}
\int_{\In} ((\Id - \PtAM) v, \phi_{q-2})_{{\HH}} \dt & = 0 & & \quad \forall \phi_{q-2} \in \Pp{q-2}{\In} \otimes {\HH},
\end{alignat}
\end{subequations}
where, for~$q = 1$, identity~\eqref{def:local-Pt-3} is omitted and~$\PtAM v$ reduces to the linear interpolant in time with nodes~$\tnmo$ and~$\tn$. 

\paragraph{Alternative definition~II.}
From identity~\eqref{def:local-Pt-3}, it is clear that, for~$q \geq 2$, ~$\Pi_{q-2}^t \PtAM v = \Pi_{q-2}^t v$. Therefore, 
\begin{equation*}
\PtAM v(t) = \Pi_{q-2} v(t) + \alpha \Leg_{q-1}(t) + \beta \Leg_q(t) \quad \forall t \in \In,
\end{equation*}
for some~$\alpha$ and~$\beta$ in~${\HH}$. Conditions~\eqref{def:local-Pt-1} and~\eqref{def:local-Pt-2}, together with the identities~$\Leg_i(\tn) = 1$ and~$\Leg_i(\tnmo) = (-1)^i$ for all~$i \in \N$, imply
\begin{equation*}
\alpha = \frac{1}{2} ( \mu_n - (-1)^q \mu_{n - 1}) \quad \text{ and } \quad \beta = \frac{1}{2 } (\mu_n + (-1)^q \mu_{n - 1}),
\end{equation*}
where~$\mu_{n-1} = (\Id - \Pi_{q-2}^t)v (\tnmo)$ and~$\mu_n = (\Id - \Pi_{q-2}^t) v(\tn)$. 

This expression can be used to extend the definition of~$\PtAM$ to functions in~$W^{1, 1}(\Tt; {\HH})$. Next lemma then follows by using the~$L^{\infty}(\In; {\HH})$ stability in Lemma~\ref{lemma:stab-pi-time} of~$\Pi_{q-2}^t$ (for~$q \geq 2$), and it holds trivially for~$q = 1$. 
\begin{lemma}[Stability of~$\PtAM$]
\label{lemma:stab-PtAM}
Let~${\HH}$ be a Hilbert space with inner product~$(\cdot, \cdot)_{{\HH}}$. For all~$v \in W^{1, 1}(\Tt; {\HH})$ and~$n \in \{1, \ldots, N\}$, it holds
\begin{equation*}
\Norm{\PtAM v}{L^{\infty}(\In; {\HH})} \lesssim \Norm{v}{L^{\infty}(\In; {\HH})}. %
\end{equation*}
\end{lemma}

\paragraph{Alternative definition~III.}
The following alternative definition can be obtained combining~\eqref{def:local-Pt-1} and the fact that identity~\eqref{def:Pt-2} implies~$\Pi_{q-1}^t \dpt \PtAM v = \Pi_{q-1}^t \dpt v$:
\begin{equation*}
    \PtAM v(t) = v(\tnmo) + \int_{\tnmo}^t  \Pi_{q-1}^t \dps v \ds \quad \forall t \in \In,
\end{equation*}
or, globally (see~\cite[Rem. 70.10]{Ern_Guermond:2021}), 
\begin{equation*}
    \PtAM v(t) = v(0) + \int_0^{t} \Pi_{q-1}^t \dps v \ds \quad \forall t \in [0, T].
\end{equation*}
\subsubsection{The Walkington projection operator~\texorpdfstring{$\PtW$}{PtWa}}
Next projection operator in time was introduced by Walkington in~\cite[Def.~5.1]{Walkington:2014} for the analysis of the DG--CG time discretization for the second-order formulation of the wave equation.
\begin{definition}[Projection~$\PtW$]
\label{def:PtW}
Let~$q \in \N$ with~$q \geq 2$, and~${\HH}$ be a Hilbert space with inner product~$(\cdot, \cdot)_{\HH}$. The projection operator~$\PtW : H^2(\Tt; {\HH}) \cap H^1(0, T; {\HH}) \to \Pcont{q}{\Tt} \otimes {\HH}$ is defined for all~$v \in H^2(\Tt; {\HH}) \cap H^1(0, T; \HH)$ as follows:
\begin{subequations}
\label{eq:proj-Pt-def}
\begin{alignat}{3} 
\label{eq:proj-Pt-def-1}
\PtW v(0)  & = v(0) & & \quad \text{in~${\HH}$},\\
\label{eq:proj-Pt-def-2}
\dpt \PtW v (\tn^-) & = \dpt v(\tn^-) & & \quad \text{in~${\HH}$, for~$n = 1, \ldots, N$},\\
\label{eq:proj-Pt-def-3}
\int_{\In} \big( \dpt (\Id - \PtW) v, \, \phi_{q-2}\big)_{{\HH}} \dt & = 0 & & \quad \forall \phi_{q - 2} \in \Pp{q-2}{\In} \otimes {\HH}, \ \text{for~$n = 1, \ldots, N$}.
\end{alignat}
\end{subequations}
\end{definition}

\paragraph{Alternative definition I.} The following (local) equivalent definition was derived in~\cite[\S5.1]{Walkington:2014}:
\begin{subequations}
\begin{alignat}{3} 
\PtW v(\tn)  & = v(\tn) & & \qquad \text{in~${\HH}$},\\
\label{eq:alternative-Walkington-2}
\PtW v(\tnmo)  & = v(\tnmo) & & \qquad \text{in~${\HH}$},\\
\dpt \PtW v (\tn^-) & = \dpt v(\tn^-) & & \qquad \text{in~${\HH}$},\\
\int_{\In} \big( (\Id - \PtW) v, \, \phi_{q-3}\big)_{{\HH}} \dt & = 0 & & \qquad \forall \phi_{q - 3} \in \Pp{q-3}{\In} \otimes {\HH},
\end{alignat}
\end{subequations}
where the last condition is omitted if~$q = 2$. 

\paragraph{Alternative definition II.} As shown in~\cite[\S2.2]{Dong-Mascotto-Wang:2025}, conditions~\eqref{eq:proj-Pt-def-2} and~\eqref{eq:proj-Pt-def-3} imply that~$\dpt \PtW v = \PtThqmo \dpt v$, which, combined with identity~\eqref{eq:alternative-Walkington-2}, gives
\begin{equation}
\label{eq:alternative-def-PtW-II}
    \PtW v(t) = v(\tnmo) + \int_{\tnmo}^t \PtThqmo \dps v \ds,
\end{equation}
which can be used to extend the definition of~$\PtW$ to functions in~$W^{2, 1}(\Tt; {\HH})$.

Next lemma is then a consequence of the stability properties in Lemma~\ref{lemma:stab-Pt} of~$\PtThqmo$.
\begin{lemma}[Stability of~$\PtW$]
\label{lemma:stab-PtW}
Let~$q \in \N$ with~$q \geq 2$, and~${\HH}$ be a Hilbert space with inner product~$(\cdot, \cdot)_{\HH}$. 
For all~$v \in W^{2, 1}(\Tt; {\HH})$ and~$n \in \{1, \ldots, N\}$, it holds
\begin{alignat*}{3}
\Norm{\dpt \PtW v}{L^{\infty}(\In; {\HH})} & \lesssim \Norm{\dpt v}{L^{\infty}(\In; {\HH})}, \\
\Norm{\PtW v}{L^{\infty}(\In; {\HH})} & \lesssim \Norm{v}{L^{\infty}(\In; {\HH})} + \tau_n \Norm{\dpt v}{L^{\infty}(\In; {\HH})}. 
\end{alignat*}
\end{lemma}
\noindent Stability and approximation properties for~$\PtW$ with~$q$-explicit constants were derived in~\cite[Lemma 2.4 and Prop. 2.5]{Dong-Mascotto-Wang:2025}.

We conclude this section with the main identity used in the convergence analysis of the DG--CG time discretization of the wave equation.
\begin{lemma}[Orthogonality of~$\PtW$]
\label{lemma:orthogonality-DG-PtW}
Let~${\HH}$ be a Hilbert space with inner product~$(\cdot, \cdot)_{{\HH}}$. For all~$v \in W^{2,1}(0, T; {\HH})$, it holds
\begin{equation*}
\int_{0}^T (\dpt \Rtq (\dpt \PtW v), \phi_{q-1})_{{\HH}} \dt  = \int_0^T (\dptt v, \phi_{q-1})_{{\HH}} \dt  + (\dpt v(0), \phi_{q-1}(0))_{{\HH}} \quad \forall \phi_{q-1} \in \Pp{q-1}{\Tt} \otimes {\HH},
\end{equation*}
where~$\Rtq : \Pp{q-1}{\Tt} \otimes {\HH} \to \Pcont{q}{\Tt} \otimes {\HH}$ is the time reconstruction operator from Section~\ref{sec:time-reconstruction}. 
\end{lemma}
\begin{proof}
    The result follows from the identity~$\dpt \PtW v = \PtThqmo \dpt v$ and Lemma~\ref{lemma:orthogonality-DG-PtTh} on the orthogonality properties of~$\PtThqmo$. 
\end{proof}

The projection operator~$\PtW$ is not stable in the~$L^{\infty}(\In; {\HH})$ norm, so we cannot apply Lemma~\ref{lemma:estimates-Ptau} to derive its approximation properties. Consequently, next lemma concerns the approximation properties of~$\PtW$. 
\begin{lemma}[Estimates for~$\PtW$]
\label{lemma:estimates-PtW}
Let~$r \in [1, \infty]$, and~${\HH}$ be a Hilbert space with inner product~$(\cdot, \cdot)_{\HH}$. Then, the following estimate holds for~$n = 1, \ldots, N$:
\begin{subequations}
\begin{alignat}{3}
\label{eq:estimate-Walkington-1}
    \Norm{(\Id - \PtW) v}{L^r(\In; {\HH})} & \lesssim \tau_n^s \Norm{\dpt^{(s)} v}{L^r(\In; {\HH})} & & \quad \forall v \in W^{s, r} (\In; {\HH}), \ 2 \le s \le q+1, \\
\label{eq:estimate-Walkington-2}
    \Norm{\dpt (\Id - \PtW) v}{L^r(\In; {\HH})} & \lesssim \tau_n^{s-1} \Norm{\dpt^{(s)} v}{L^r(\In; {\HH})} & & \quad \forall v \in W^{s, r} (\In; {\HH}), \ 2 \le s \le q+1,
\end{alignat}
\end{subequations}
\end{lemma}
\begin{proof}
From~\eqref{eq:alternative-def-PtW-II} and the Jensen inequality, we deduce the following bound:
\begin{alignat*}{3}
\Norm{(\Id - \PtW) v(t)}{{\HH}} & \le \int_{\tnmo}^t \Norm{(\Id - \PtThqmo) \dps v}{{\HH}} \ds, 
\end{alignat*}
which, combined with the H\"older inequality (with~$1/r + 1/r' = 1$), yields
\begin{equation*}
\Norm{(\Id - \PtW) v}{L^r(\In; {\HH})} \le \tau_n^{1/r'} \tau_n^{1/r} \Norm{(\Id - \PtThqmo) \dpt v}{L^r(\In; {\HH})} = \tau_n \Norm{(\Id - \PtThqmo) \dpt v}{L^r(\In; {\HH})}.
\end{equation*}
Estimate~\eqref{eq:estimate-Walkington-1} then follows from the approximation properties of~$\PtThqmo$. 

As for~\eqref{eq:estimate-Walkington-2}, it is enough to use the identity~$\dpt \PtW v = \PtThqmo \dpt v$ and the approximation properties of~$\PtThqmo$. 
\end{proof}
\subsection{The~\texorpdfstring{$L^2(\Omega)$}{L2(Omega)}-orthogonal projection~\texorpdfstring{$\Pi_h$}{Pih} and the Ritz projection operator~\texorpdfstring{$\Rh$}{Rh}}
We denote by~$\Pih : L^2(\Omega) \to \oVhp$ the~$L^2(\Omega)$-orthogonal projection in the space~$\oVhp$. 

Moreover, we denote by~$\Rh : H_0^1(\Omega) \to \oVhp$ the Ritz projection operator, defined for any~$u \in H_0^1(\Omega)$ as the solution to
\begin{equation*}
    (\nabla \Rh u, \nabla \vh)_{\Omega} = (\nabla u, \nabla \vh)_{\Omega} \qquad \forall \vh \in \oVhp,
\end{equation*}
and recall the following approximation properties of~$\Rh$ from~\cite[Thms.~5.4.4 and 5.4.8]{Brenner_Scott:2008}.
\begin{lemma}[Estimates for~$\Rh$]
\label{lemma:Estimates-Rh}
Let~$p \in \N$ with~$p \geq 1$. Then,
\begin{equation*}
    \Norm{\nabla (\Id - \Rh) u}{L^2(\Omega)^d} \lesssim h^r\Seminorm{u}{H^{r+1}(\Omega)} \qquad \forall u \in H^{r+1}(\Omega) \cap H_0^1(\Omega), \ 0 \le r \le p.
\end{equation*}
Moreover, if the spatial domain~$\Omega$ satisfies the following elliptic regularity condition:
\begin{equation}
\label{eq:elliptic-regularity-Omega}
\text{
$v \in H_0^1(\Omega)$ and~$\Delta v \in L^2(\Omega)$ $\Longrightarrow$~$v \in H^2(\Omega)$, 
}
\end{equation}
it also holds
\begin{equation}
\label{eq:err-Rh-L2}
\Norm{(\Id - \Rh) u}{L^2(\Omega)} \lesssim h^{r+1} \Seminorm{u}{H^{r+1}(\Omega)} \qquad \forall u \in H^{r+1}(\Omega) \cap H_0^1(\Omega), \ 0 \le r \le p. 
\end{equation}
\end{lemma}

\begin{remark}[Reduced elliptic regularity]
A priori error estimates for spatial domains that do not satisfy full elliptic regularity follow analogously, and the loss of regularity is reflected in~\eqref{eq:err-Rh-L2} through the elliptic regularity parameter~$\sigma \in (0, 1)$. 
\eremk
\end{remark}

\begin{remark}[Discrete initial conditions in CG time discretizations]
In CG time discretizations, discrete approximations of the initial conditions must be imposed strongly. We use~$\Pih$ (resp.~$\Rh$) to approximate initial data in~$L^2(\Omega)$ (resp. in~$H_0^1(\Omega)$). 
In particular, the use of~$\Pih$ for initial data in~$L^2(\Omega)$ allows us to derive stability bounds with the minimal regularity assumed in Section~\ref{sec:model-problems}, whereas using~$\Rh$ for functions in~$H_0^1(\Omega)$ avoids additional assumptions on the spatial mesh~$\Th$ that would otherwise be required to ensure the stability of~$\Pih$ in the~$H^1(\Omega)$ norm (cf. \cite{Crouzeix_Thomee:1987,Bramble_Pasciak_Steinbach:2002,Castensen:2002,Diening_Storn_Tscherpel:2021}). 
\eremk
\end{remark}

\subsection{Structure of the stability and convergence analysis}
\label{sec:structure-theory}
All the results in the next sections concerning the stability and convergence analysis of the corresponding schemes follow the same structure, which we summarize below. 

\paragraph{Stability analysis.} We derive continuous dependence of any discrete solution on the data of the problem in~$L^{\infty}(0, T; \mathcal{X})$ norms, which also implies the existence and uniqueness of a discrete solution. Such stability bounds are derived in two steps:
\begin{enumerate}[label = \emph{\roman*)}, ref = \emph{\roman*)}]
    \item\label{step-i} We first employ standard energy arguments (with natural test functions) to obtain a \emph{weak partial bound} on the discrete solution. We call these bounds \emph{weak} because the left-hand side only controls the solution at discrete time instances, and \emph{partial} because the right-hand side still involves~$L^{\infty}(0, T; \mathcal{X})$ norms of the discrete solution.

    \item Then, we derive a local stability bound using test functions with support in a single time slab~$\Qn$, which involve the auxiliary weight functions~$\{\varphi_n\}_{n = 1}^N$ defined in Section~\ref{sec:auxiliary-weight}. We summarize the choice of such test functions in the fifth column of Table~\ref{tab:special-test-functions}.

    The resulting bounds provide control on~$L^2(\Qn)$ norms of the discrete solution weighted with~$\tau_n^{-1}$, which is then used in combination with the polynomial inverse estimate in Lemma~\ref{lemma:L2-Linfty} and step~\ref{step-i} to obtain control in~$L^{\infty}(0, T; \mathcal{X})$ norms. This also guarantees the existence of a unique solution to the scheme. 
\end{enumerate}

\begin{table}[!ht]
    \centering
    \caption{Choice of the special test functions used to derive stability bounds in~$L^{\infty}(0, T; \mathcal{X})$ norms, and the time projection~$\Pt$ used in the convergence analysis.}
    \label{tab:special-test-functions}
    \begin{tabular}{cccccc}
    \hline 
    Section & Model & Discrete solution & Test space & Test function & Time projection\\
    \hline\\[-0.1em]
    \S\ref{sec:Jamet-heat} & Heat equation & $\uht \in \oVdisc{q}$ &  $\oVdisc{q}$ & $\ItR (\varphi_n \uht)$ & $\PtTh$ \\[0.7em]
    \S\ref{sec:Aziz-Monk-heat} & Heat equation &  $\uht \in \oVcont{q}$ & $\oVdisc{q-1}$ & $\Pi_{q-1}^t(\varphi_n \Pi_{q-1}^t \uht)$ & $\PtAM$ \\[0.7em]
    \S\ref{sec:plain-vanilla-wave} & 
    Wave equation & 
    $\uht \in \oVdisc{q}$ & $\oVdisc{q}$ 
    & eq.~\eqref{eq:test-function-plain-DG} & $\PtW$ \\[0.7em]
    \multirow{2}{*}{\S\ref{sec:French-Peterson-wave}} & \multirow{2}{*}{Wave equation} & $\uht \in \oVcont{q} $ & $\oVdisc{q-1} $ & $\Pi_{q-1}^t (\varphi_n \Pi_{q-1}^t \uht)$ & $\PtAM$ \\
    & & $\vht \in \oVcont{q} $ & $\oVdisc{q-1} $ & $\Pi_{q-1}^t (\varphi_n \Pi_{q-1}^t \vht)$ & $\PtAM$ \\[0.7em]
    \multirow{2}{*}{\S\ref{sec:Johnson-wave}} & \multirow{2}{*}{Wave equation} & $\uht \in \oVdisc{q} $ & $\oVdisc{q} $ & $\ItR(\varphi_n \uht)$ & $\PtTh $\\
    & & $\vht \in \oVdisc{q} $ & $\oVdisc{q} $ & $\ItR(\varphi_n \vht)$ & $\PtTh$ \\[0.7em]
    \S\ref{sec:Walkington-wave} & Wave equation & $\uht \in \oVcont{q} $ & $\oVdisc{q-1} $ & $\Pi_{q-1}^t(\varphi_n \dpt \uht)$ & $\PtW$ \\[0.7em]
    \hline
    \hline
    \end{tabular}
    \end{table}

\begin{remark}[Equivalent choices of the test functions]
By virtue of the equivalences in Remarks~\ref{rem:equivalence-L2-orthogonal-Legendre-interpolant} and~\ref{rem:equivalence-ItR-PtThL}, the test functions in the fifth column of Table~\ref{tab:special-test-functions} can be equivalently expressed in terms of the alternative operators~$\PtThL$ and~$\mathcal{I}_{\tau}^{\mathsf{L}(q-1)}$.
\eremk
\end{remark}

\paragraph{Convergence analysis.}
We derive \emph{a priori} error estimates in a standard way.
\begin{itemize}
\item We first define error functions of the form
\begin{equation*}
    e_u := u - \uht = (u - \Piht u) + \Piht e_u,
\end{equation*}
where~$\Piht$ is a space--time projection obtained by combining a suitable projection~$\Pt$ in time from Section~\ref{sec:projections-time} with the Ritz projection~$\Rh$ in space. The specific choice of~$\Pt$ used in each section is summarized in the last column of Table~\ref{tab:special-test-functions}.

\item Then, we use the consistency of the scheme and the orthogonality properties of~$\Piht$ to derive an equation satisfied by the discrete error function~$\Piht e_u$. 

\item \emph{A priori} error estimates are then obtained by simply applying the continuous dependence on the data derived in the stability analysis, and using the approximation properties of~$\Piht$. 
\end{itemize}

\begin{remark}[Discrete test spaces]
Since no continuity is imposed on the test spaces, the global linear system arising from these methods can be decomposed into a sequence of smaller problems on each time slab~$\Qn$. 
This structure is relevant not only from a computational point of view but also from a theoretical perspective.
Indeed, to derive stability bounds in~$L^{\infty}(0, T; \mathcal{X})$ norms, it is essential that the specific test functions chosen have local support on a single time slab~$\Qn$.
The absence of this property complicates the analysis of methods based on spline functions with higher temporal regularity, which require alternative approaches; see~\cite{Ferrari_Fraschini:2026,Ferrari_Fraschini_Loli_Perugia:2025,Ferrari_Gomez:2025,Ferrari_Perugia:2025,Ferrari_Perugia_Zampa:2025}.
\eremk
\end{remark}

\begin{remark}[Discrete approximation of Dirichlet data]
\label{rem:nonhomogeneous-Dirichlet-boundary}
Nonhomogeneous smooth Dirichlet data must be suitably approximated according to the time discretization. Previous works~\cite{Walkington:2014,Voulis:2019,Gomez:2025} suggest that the right choice is to approximate the Dirichlet data so that the discrete error~$\Piht e_u$ satisfies a problem with homogeneous boundary conditions, i.e., such that~$\uht = \Piht u$ on~$\partial \Omega \times \In$, for~$n = 1, \ldots, N$. This, in addition, allows one to apply the stability bounds presented here for the homogeneous Dirichlet case.
\eremk
\end{remark}

\section{DG time discretization for the heat equation (Jamet, 1978)}
\label{sec:Jamet-heat}
Building upon the DG time discretization of ordinary differential equations from~\cite{Delfour_Hager_Trochu:1981}, Jamet~\cite{Jamet:1978} proposed the following space--time finite element method (see also the extension to more general parabolic problems in~\cite{Eriksson_Johnson_Thomee:1985}): find~$\uht \in \oVdisc{q}$ such that
\begin{equation}
\label{eq:variational-DG-heat}
\begin{split}
\Bht(\uht, \vht) & = \ell(\vht) \qquad \forall \vht \in \oVdisc{q},
\end{split}
\end{equation}
where the bilinear form~$\Bht : \oVdisc{q} \times \oVdisc{q} \to \R$ and the linear functional~$\ell : \oVdisc{q} \to \R$ are given by
\begin{alignat*}{3}
\nonumber
\Bht(\uht, \vht) & := \sum_{n = 1}^{\Nt} (\dpt \uht, \vht)_{\Qn} + \sum_{n = 1}^{\Nt - 1} (\jump{\uht}_n, \vht(\cdot, \tn^+))_{\Omega} + (\uht, \vht)_{\SO} \\
& \quad + \nu (\nabla \uht, \nabla \vht)_{\QT}, \\
\ell(\vht) & := (f, \vht)_{\QT} + (u_0, \vht(\cdot, 0))_{\Omega}.
\end{alignat*}

Using the time reconstruction operator~$\Rt$ in Section~\ref{sec:time-reconstruction}, the bilinear form~$\Bht(\cdot, \cdot)$ can be written in compact form
\begin{equation*}
\Bht(\uht, \vht) = (\dpt \Rt \uht, \vht)_{\QT} + \nu (\nabla \uht, \nabla \vht)_{\QT} \qquad \forall \uht, \vht \in \oVdisc{q}. 
\end{equation*}

For the convergence analysis, it is useful to study the stability properties of the following perturbed space--time formulation: find~$\uht \in \oVdisc{q}$ such that
\begin{equation}
\label{eq:compact-variational-DG-heat}
\Bht(\uht, \vht) = (f, \vht)_{\QT} + \nu (\nabla \Upsilon, \nabla \vht)_{\QT} + (u_0, \vht(\cdot, 0))_{\Omega} \quad \forall \vht \in \oVdisc{q},
\end{equation}
where~$\Upsilon \in L^2(0, T; H_0^1(\Omega))$ is a perturbation function that will represent a projection error in the convergence analysis.

The forthcoming stability and convergence analysis follows the structure described in Section~\ref{sec:structure-theory}. 

\subsection{Stability analysis}
\begin{lemma}[Weak partial bound on the discrete solution\label{lemma:weak-bound}]
Any solution to the discrete space--time formulation~\eqref{eq:compact-variational-DG-heat} satisfies the following bound for~$n = 1, \ldots, N$: 
\begin{alignat}{3}
\nonumber
\frac12 & \Norm{\uht(\cdot, \tn^{-})}{L^2(\Omega)}^2 + \frac{\nu}{2} \Norm{\nabla \uht}{L^2(0, \tn; L^2(\Omega)^d)}^2  + \frac12 \sum_{m = 1}^{n - 1} \Norm{\jump{\uht}_{m}}{L^2(\Omega)}^2 + \frac14 \Norm{\uht(\cdot, 0)}{L^2(\Omega)}^2  \\
\label{eq:weak-partial-bound-heat}
& \quad \le \Norm{f}{L^1(0, \tn; L^2(\Omega))} \Norm{\uht}{L^{\infty}(0, \tn; L^2(\Omega))} + \frac{\nu}{2} \Norm{\nabla \Upsilon}{L^2(0, \tn; L^2(\Omega)^d)}^2 + \Norm{u_0}{L^2(\Omega)}^2.
\end{alignat}
\end{lemma}
\begin{proof}
Choosing~$\vht = \chi_{(0, \tn)} \uht$ in~\eqref{eq:compact-variational-DG-heat}, and using identity~\eqref{eq:Rt-vht-jump} and the H\"older inequality, we get
\begin{alignat*}{3}
\frac12 & \Big(\Norm{\uht(\cdot, \tn^-)}{L^2(\Omega)}^2  + \sum_{m = 1}^{n - 1} \Norm{\jump{\uht}_{m}}{L^2(\Omega)}^2 + \Norm{\uht}{L^2(\SO)}^2 \Big) + \nu \Norm{\nabla \uht}{L^2(0, \tn; L^2(\Omega)^d)}^2 \\
&\quad  \le \Norm{f}{L^1(0, \tn; L^2(\Omega))} \Norm{\uht}{L^{\infty}(0, \tn; L^2(\Omega))} + \nu \Norm{\nabla \Upsilon}{L^2(0, \tn; L^2(\Omega)^d)} \Norm{\nabla \uht}{L^2(0, \tn; L^2(\Omega)^d)} \\
& \qquad + \Norm{u_0}{L^2(\Omega)} \Norm{\uht}{L^2(\SO)}.
\end{alignat*}
The result then follows by using the Young inequality for the last two terms on the right-hand side.
\end{proof}

\begin{proposition}[Continuous dependence on the data]
\label{prop:continuous-dependence-DG-heat}
Any solution to the discrete space--time formulation~\eqref{eq:compact-variational-DG-heat} satisfies
\begin{alignat*}{3}
\Norm{\uht}{L^{\infty}(0, T; L^2(\Omega))}^2 + \nu \Norm{\nabla \uht}{L^2(\QT)^d}^2 + \Seminorm{\uht}{\sf J}^2 \lesssim \Norm{f}{L^1(0, T; L^2(\Omega))}^2  + \nu \Norm{\nabla \Upsilon}{L^2(\QT)^d}^2 + \Norm{u_0}{L^2(\Omega)}^2,
\end{alignat*}
where the hidden constant depends only on~$q$.
\end{proposition}
\begin{proof}
For~$n = 1, \ldots, N$, we choose the test function~$\vht \in \oVdisc{q}$ in~\eqref{eq:compact-variational-DG-heat} as
\begin{equation*}
\vht{}_{|_{\Qm}} := \begin{cases}
\ItR(\varphi_n \uht) & \text{ if } m = n, \\
0 & \text{ otherwise}.
\end{cases}
\end{equation*}
We focus on the case~$n > 1$, as the argument for~$n = 1$ requires only minor modifications. 

The following identity is obtained:
\begin{alignat}{3}
\nonumber
(\dpt \uht, \ItR(\varphi_n \uht))_{\Qn} + (\jump{\uht}_{n-1}, \ItR(\varphi_n \uht)(\cdot, \tnmo^+))_{\Omega} + \nu (\nabla \uht, \nabla \ItR(\varphi_n \uht))_{\Qn}  \\
\label{eq:aux-local-identity-varphi}
 = (f, \ItR(\varphi_n \uht))_{\Qn} + \nu (\nabla \Upsilon, \nabla \ItR (\varphi_n \uht))_{\Qn}.
\end{alignat}

Using property~\eqref{eq:ItR-exact} of the interpolant~$\ItR$, the identities~$\ItR w(\cdot, \tnmo^+) = w(\cdot, \tnmo^+)$ and $\varphi_n(\tnmo) = 1$, and identity~\eqref{eq:identities-varphi-integration-1} from Lemma~\ref{lemma:identities-varphi-integration}, we get
\begin{alignat}{3}
\nonumber
\big(\dpt \uht, & \, \ItR(\varphi_n \uht) \big)_{\Qn} + \big(\jump{\uht}_{n-1}, \ItR(\varphi_n \uht)(\cdot, \tnmo^+) \big)_{\Omega} \\
\nonumber
& = \big(\dpt \uht, \varphi_n \uht \big)_{\Qn} + \big(\jump{\uht}_{n-1}, \uht(\cdot, \tnmo^+) \big)_{\Omega} \\
\nonumber
& = \frac14 \Norm{\uht(\cdot, \tn^-)}{L^2(\Omega)}^2 + \frac12 \Norm{\jump{\uht}_{n - 1}}{L^2(\Omega)}^2 - \frac12 \Norm{\uht(\cdot, \tnmo^-)}{L^2(\Omega)}^2  + \frac{\lambda_n}{2} \Norm{\uht}{L^2(\Qn)}^2.
\end{alignat}

The application of Lemma~\ref{lemma:bilinear-form-varphi} immediately gives
$$\nu \big(\nabla \uht, \nabla \ItR(\varphi_n \uht) \big)_{\Qn} \geq \frac{\nu}{2} \Norm{\nabla \uht}{L^2(\Qn)^d}^2.$$

As for the terms on the right-hand side of~\eqref{eq:aux-local-identity-varphi}, we use the H\"older inequality, the stability bound in~\eqref{eq:stab-ItR-L2} for~$\ItR$, and the Young inequality to get
\begin{alignat*}{3}
(f, \ItR( &\varphi_n \uht))_{\Qn}  + \nu (\nabla \Upsilon, \nabla \ItR (\varphi_n \uht))_{\Qn} \\
& \le \Norm{f}{L^1(\In; L^2(\Omega))} \Norm{\ItR(\varphi_n \uht)}{L^{\infty}(\In; L^2(\Omega))} + \nu \Norm{\nabla \Upsilon}{L^2(\Qn)^d} \Norm{\nabla \ItR (\varphi_n \uht)}{L^2(\Qn)^d} \\
& \le \CSI \Norm{f}{L^1(\In; L^2(\Omega))} \Norm{\uht}{L^{\infty}(\In; L^2(\Omega))} + \nu (\CSI)^2 \Cinv \Norm{\nabla \Upsilon}{L^2(\Qn)^d}^2 + \frac{\nu}{4} \Norm{\nabla \uht}{L^2(\Qn)^d}^2.
\end{alignat*}

Combining these three bounds with identity~\eqref{eq:aux-local-identity-varphi} and the weak partial bound in Lemma~\ref{lemma:weak-bound} for~$n - 1$ yields
\begin{alignat}{3}
\nonumber
\frac14 & \Norm{\uht(\cdot, \tn^-)}{L^2(\Omega)}^2  + \frac12 \Norm{\jump{\uht}_{n - 1}}{L^2(\Omega)}^2 + \frac{\nu}{4} \Norm{\nabla \uht}{L^2(\Qn)^d}^2 + \frac{\lambda_n}{2} \Norm{\uht}{L^2(\Qn)}^2 \\
\nonumber
& \quad \lesssim 
\frac12 \Norm{\uht(\cdot, \tnmo^-)}{L^2(\Omega)}^2 + \Norm{f}{L^1(\In; L^2(\Omega))} \Norm{\uht}{L^{\infty}(\In; L^2(\Omega))}  + \nu \Norm{\nabla \Upsilon}{L^2(\Qn)^d}^2 \\
\nonumber
& \quad \lesssim \Norm{u_0}{L^2(\Omega)}^2 + \Norm{f}{L^1(0, \tnmo; L^2(\Omega))} \Norm{\uht}{L^{\infty}(0, \tnmo; L^2(\Omega))} + \nu \Norm{\nabla \Upsilon}{L^2(0, \tnmo; L^2(\Omega)^d)}^2 \\
\nonumber
& \qquad + \Norm{f}{L^1(\In; L^2(\Omega))} \Norm{\uht}{L^{\infty}(\In; L^2(\Omega))} + \nu \Norm{\nabla \Upsilon}{L^2(\In; L^2(\Omega)^d)}^2\\
\label{eq:continuous-dependence-bound-tn}
& \quad \lesssim \Norm{u_0}{L^2(\Omega)}^2 + \Norm{f}{L^1(0, \tn; L^2(\Omega))} \Norm{\uht}{L^{\infty}(0, \tn; L^2(\Omega))} + \nu \Norm{\nabla \Upsilon}{L^2(0, \tn; L^2(\Omega)^d)}^2,
\end{alignat}
where it is crucial that the hidden constant is independent of the time interval~$\In$ and depends only on the degree of approximation~$q$.

From the inverse estimate in Lemma~\ref{lemma:L2-Linfty}, we deduce that
\begin{equation*}
\frac{1}{4 \Cinv} \Norm{\uht}{L^{\infty}(\In; L^2(\Omega))}^2 \le \frac{\lambda_n}{2} \Norm{\uht}{L^2(\Qn)}^2,
\end{equation*}
which, together with~\eqref{eq:continuous-dependence-bound-tn}, allows us to show that
\begin{alignat*}{3}
\Norm{\uht}{L^{\infty}(\In; L^2(\Omega))}^2 \lesssim \Norm{u_0}{L^2(\Omega)}^2 + \Norm{f}{L^1(0, T; L^2(\Omega))} \Norm{\uht}{L^{\infty}(0, T; L^2(\Omega))} + \nu \Norm{\nabla \Upsilon}{L^2(\QT)^d}^2,
\end{alignat*}
and a uniform bound on~$\Norm{\uht}{L^{\infty}(0, T; L^2(\Omega))}$ can be obtained by choosing the index~$n$ where the left-hand side achieves its maximum value. This step, together with Lemma~\ref{lemma:weak-bound}, completes the proof. 
\end{proof}
The existence of a unique solution to~\eqref{eq:variational-DG-heat} is an immediate consequence of Proposition~\ref{prop:continuous-dependence-DG-heat}, as it corresponds to a linear system with a nonsingular matrix.

\begin{remark}[Limitations of the energy argument]
A stability bound can be obtained from Lemma~\ref{lemma:weak-bound} only for the lowest order cases. This explains why many analyses restrict to that situation (see, e.g., \cite{Hansbo_Szepessy:1990,Beirao_Dassi_Vacca:2024,Vexler_Wagner:2024} for the incompressible Navier--Stokes equations).
The main issue is that the terms on the right-hand side of~\eqref{eq:weak-partial-bound-heat} cannot control the~$L^2(\Omega)$ norm of~$\uht$ at all times, which highlights the limitations of the standard energy argument. 

A continuous dependence on the data in~$L^{\infty}(0, T; \mathcal{X})$ norms for the DG time discretization of linear parabolic problems with~$q$-variable approximations was established in~\cite{Schmutz_Wihler:2019}.
\eremk
\end{remark}

\subsection{Convergence analysis}
Recalling Definition~\ref{def:Pt} of the Thom\'ee projection operator~$\PtTh$, we define the space--time projection~$\Piht = \PtTh \circ \Rh$, and the error functions
\begin{equation*}
\eu := u - \uht = (\Id - \Piht)u + \Piht \eu. 
\end{equation*}
\begin{theorem}[Estimates for the discrete error]
\label{thm:a-priori-Thomee}
Let~$u \in X$ be the solution to the continuous weak formulation~\eqref{eq:weak-formulation-heat}, and let~$\uht \in \oVdisc{q}$ be the solution to the discrete space--time formulation~\eqref{eq:variational-DG-heat}. Assume that~$\Omega$ satisfies the elliptic regularity assumption~\eqref{eq:elliptic-regularity-Omega}, and that the continuous solution~$u$ satisfies
\begin{equation*}
u  \in X \cap W^{1, 1}(0, T; H^{r+1}(\Omega)) \cap H^{s+1}(0, T; H^1(\Omega)),
\end{equation*}
for some~$0 \le r \le p$ and~$0 \le s \le q$. 
Then, the following estimate holds:
\begin{alignat}{3}
\nonumber
& \Norm{\Piht \eu}{L^{\infty}(0, T; L^2(\Omega))}   + \sqrt{\nu} \Norm{\nabla \Piht \eu}{L^2(\QT)^d} + \Seminorm{\Piht \eu}{\sf J} \\
\label{eq:discrete-error-DG-heat}
& \quad \lesssim h^{r+1} \Seminorm{\dpt u}{L^1(0, T; H^{r+1}(\Omega))} + \sqrt{\nu} \tau^{s+1} \Norm{\nabla \dpt^{(s+1)}u}{L^2(\QT)^d}+ h^{r+1} \Seminorm{u_0}{H^{r+1}(\Omega)}.
\end{alignat}
\end{theorem}
\begin{proof}
The space--time formulation~\eqref{eq:variational-DG-heat} is consistent due to the continuous embedding~$u \in X \hookrightarrow C^0([0, T]; L^2(\Omega))$, as it makes the jump terms vanish.
Therefore, we have the error equation
\begin{equation}
\label{eq:error-equation-Thomee}
\Bht(\Piht \eu, \vht) = -\Bht((\Id - \Piht) u, \vht) \qquad \forall \vht \in \oVdisc{q}.
\end{equation}

The right-hand side of~\eqref{eq:error-equation-Thomee} can be simplified by using the orthogonality properties of~$\PtTh$ from Lemma~\ref{lemma:orthogonality-DG-PtTh} and of~$\Rh$ as follows:
\begin{subequations}
\begin{alignat}{3}
(\dpt \Rt (\Id - \Piht u), \vht)_{\QT} 
& = (\dpt (\Id - \Rh) u, \vht)_{\QT} + ((\Id - \Rh) u, \vht)_{\SO}, \\
\label{eq:identity-nabla-nabla-heat}
\nu (\nabla (\Id - \Piht) u, \nabla \vht)_{\QT} & = \nu (\nabla (\Id - \PtTh) u, \nabla \vht)_{\QT}.
\end{alignat}
\end{subequations}
Therefore, the discrete error~$\Piht \eu = \Piht u - \uht$ satisfies a problem of the form of~\eqref{eq:variational-DG-heat} with data
\begin{equation*}
\widetilde{f} = -(\Id - \Rh)\dpt u, \quad \widetilde{\Upsilon} = -(\Id - \PtTh) u, \quad \text{ and } \quad \widetilde{u}_0  = -(\Id - \Rh) u_0. 
\end{equation*}

The continuous dependence on the data from Proposition~\ref{prop:continuous-dependence-DG-heat} then gives the \emph{a priori} error bound
\begin{alignat*}{3}
&\Norm{\Piht \eu}{L^{\infty}(0, T; L^2(\Omega))}^2   + \nu \Norm{\nabla \Piht\eu}{L^2(\QT)^d}^2 + \Seminorm{\Piht \eu}{\sf J}^2 \\
& \quad \lesssim \Norm{(\Id - \Rh) \dpt u}{L^1(0, T; L^2(\Omega))}^2 + \nu \Norm{(\Id - \PtTh) \nabla u}{L^2(\QT)^d}^2  + \Norm{(\Id - \Rh)u_0}{L^2(\Omega)}^2.
\end{alignat*}
Estimate~\eqref{eq:discrete-error-DG-heat} then follows by the approximation properties of~$\PtTh$ and~$\Rh$ from Lemmas~\ref{lemma:estimates-Ptau} (relying on~Lemma~\ref{lemma:stab-Pt}) and~\ref{lemma:Estimates-Rh}, respectively. 
\end{proof}

\begin{corollary}[\emph{A priori} error estimates]
Under the assumptions of Theorem~\ref{thm:a-priori-Thomee}, and assuming that the continuous solution~$u$ to~\eqref{eq:weak-formulation-heat} is sufficiently regular, there hold
\begin{alignat*}{3}
\Norm{\eu}{L^{\infty}(0, T; L^2(\Omega))} & \lesssim \tau^{q+1} \big(\Norm{\dpt^{(q+1)} u}{L^{\infty}(0, T; L^2(\Omega))} + \sqrt{\nu} \Norm{\nabla \dpt^{(q+1)} u}{L^2(\QT)^d}\big)\\
& \quad + h^{p+1} \big(\Norm{u}{L^{\infty}(0, T; H^{p+1}(\Omega))} + \Seminorm{\dpt u}{L^1(0, T; H^{p+1}(\Omega))} + \Seminorm{u_0}{H^{p+1}(\Omega)} \big), \\
\sqrt{\nu} \Norm{\nabla \eu}{L^2(\QT)^d} & \lesssim \sqrt{\nu} \tau^{q+1} \Norm{\nabla \dpt^{(q+1)} u}{L^2(\QT)^d} \\
& \quad + h^{p} \big(\sqrt{\nu} \Seminorm{u}{L^2(0, T; H^{p+1}(\Omega))} + \Seminorm{\dpt u}{L^1(0, T; H^p(\Omega))} + \Seminorm{u_0}{H^p(\Omega)} \big).
\end{alignat*}
\end{corollary}

We conclude this section with some remarks.
\begin{remark}[Equivalence with Runge--Kutta schemes]
The usual claim that DG time discretizations are equivalent to some Runge--Kutta time-stepping schemes can be misleading. Such an equivalence is only obtained if one interpolates the source term as~$\ItR f$, which requires that~$f \in C^0([0, T]; L^2(\Omega))$. 
As it is evident from our convergence analysis, the regularity required on~$f$ in the stability analysis has a direct impact on the regularity required on the continuous solution~$u$ in the error estimate. 
\eremk
\end{remark}

\begin{remark}[An alternative estimate]
Integrating by parts in space in~\eqref{eq:identity-nabla-nabla-heat}, one can obtain an alternative estimate of the form
    \begin{equation*}
    \begin{split}
        \nu (\nabla (\Id - \PtTh) u, \nabla \wht)_{\QT} & = - \nu ((\Id - \PtTh) \Delta u, \wht)_{\QT} \\
        & \lesssim \nu \tau^{s+1} \Seminorm{\Delta \dpt^{(s+1)}u}{L^1(0, T; L^2(\Omega))} \Norm{\wht}{L^{\infty}(0, T; L^2(\Omega))}.
    \end{split}
    \end{equation*}
    Thus requiring~$\Delta u \in W^{s+1,1}(0, T; L^2(\Omega))$ instead of~$u \in H^{s+1}(0, T; H^1(\Omega))$.
\eremk
\end{remark}

\begin{remark}[Inf--sup approach]
The classical inf--sup analysis for parabolic problems has been extended to the fully discrete setting with DG time discretizations in~\cite[\S71.3]{Ern_Guermond:2021} and~\cite{Smears:2017,Neumuller_Smears:2019,Saito:2021}.
\eremk
\end{remark}

\begin{remark}[$q$-explicit estimates]
Estimates with~$q$-explicit constants for the DG time discretization of parabolic problems were derived in~\cite{Schotzau_Schwab:2000}.
\eremk
\end{remark}
\section{CG time discretization for the heat equation (Aziz--Monk, 1989)}
\label{sec:Aziz-Monk-heat}
Building upon the CG time discretization of ordinary differential equations in~\cite{Hulme:1972,Hulme:1972b}, Aziz and Monk proposed the following continuous space--time finite element method in~\cite{Aziz_Monk:1989}: find~$\uht \in \oVcont{q}$ ($q \geq 1$) with~$\uht(\cdot, 0) = \Pih u_0$ such that
\begin{alignat}{3}
\label{eq:Aziz-Monk}
\Bht(\uht, \vht) = \ell(\vht) \qquad \forall \vht \in \oVdisc{q-1},
\end{alignat}
where the bilinear form~$\Bht : \oVcont{q} \times \oVdisc{q-1} \to \R$ and the linear functional~$\ell : \oVdisc{q-1} \to \R$ are given by
\begin{alignat*}{3}
\Bht(\uht, \vht) & = (\dpt \uht, \vht)_{\QT} + \nu (\nabla \uht, \nabla \vht)_{\QT}, \\
\ell(\vht) & = (f, \vht)_{\QT}.
\end{alignat*}

As in the previous section, it is convenient to study the stability of the following perturbed space--time formulation: find~$\uht \in \oVcont{q}$ (with~$q \geq 1$) such that~$\uht(\cdot, 0) = \Pih u_0$ and
\begin{equation}
\label{eq:CG-heat-perturbed}
\Bht(\uht, \vht) = (f, \vht)_{\QT} + \nu (\nabla \Upsilon, \nabla \vht)_{\QT} \quad \forall \vht \in \oVdisc{q-1},
\end{equation}
where~$\Upsilon \in L^2(0, T; H^1(\Omega))$ is a perturbation function that will represent a projection error.

The forthcoming stability and convergence analysis follows the structure described in Section~\ref{sec:structure-theory}. 
\subsection{Stability analysis}

\begin{lemma}[Weak partial bound on the discrete solution]
\label{lemma:weak-partial-bound-Aziz-Monk}
Any solution~$\uht \in \oVcont{q}$ to the discrete space--time formulation~\eqref{eq:CG-heat-perturbed} satisfies the following bound for~$n = 1, \ldots, N$:
\begin{alignat*}{3}
\frac12 \Norm{\uht}{L^2(\Sn)}^2 + \frac{\nu}{4} \Norm{\nabla \Pi_{q-1}^t \uht}{L^2(0, \tn; L^2(\Omega)^d)}^2 & \le \frac12 \Norm{u_0}{L^2(\Omega}^2 + \CS \Norm{f}{L^1(0, \tn; L^2(\Omega))} \Norm{\uht}{L^{\infty}(0, \tn; L^2(\Omega))} \\
& \quad + \nu \Norm{\nabla \Upsilon}{L^2(0, \tn; L^2(\Omega)^d)}^2,
\end{alignat*}
where~$\CS$ is the constant in~Lemma~\ref{lemma:stab-pi-time}, which depends only on~$q$.
\end{lemma}
\begin{proof}
Without loss of generality, we prove the result only for~$n = N$. 
Choosing~$\vht = \Pi_{q-1}^t \uht \in \oVdisc{q-1}$ in~\eqref{eq:CG-heat-perturbed}, we have
\begin{equation}
\label{eq:aux-identity-Aziz-Monk}
\Bht(\uht, \Pi_{q-1}^t \uht) = (f, \Pi_{q-1}^t \uht)_{\QT} + \nu (\nabla \Upsilon, \nabla \Pi_{q-1}^t \uht)_{\QT}.
\end{equation}

The orthogonality properties of~$\Pi_{q-1}^t$ and integration by parts in time yield
\begin{alignat*}{3}
\nonumber
\Bht(\uht, \Pi_{q-1}^t \uht) & = (\dpt \uht, \Pi_{q-1}^t \uht)_{\QT} + \nu (\nabla \uht, \nabla \Pi_{q-1}^t \uht)_{\QT} \\
\nonumber
& = (\dpt \uht, \uht)_{\QT} + \nu (\nabla \Pi_{q-1}^t \uht, \nabla \Pi_{q-1}^t \uht)_{\QT} \\
& = \frac12 \Norm{\uht}{L^2(\ST)}^2  - \frac12 \Norm{\Pih u_0}{L^2(\Omega)}^2 + \nu \Norm{\nabla \Pi_{q-1}^t \uht}{L^2(\QT)^d}^2.
\end{alignat*}

As for the right-hand side of~\eqref{eq:aux-identity-Aziz-Monk}, we use the H\"older inequality, Lemma~\ref{lemma:stab-pi-time} on the stability properties of~$\Pi_{q-1}^t$, and the Young inequality to obtain
\begin{alignat*}{3}
(f, \Pi_{q-1}^t \uht)_{\QT} + \nu (\nabla \Upsilon, \nabla \Pi_{q-1}^t \uht)_{\QT} & \le \Norm{f}{L^1(0, T; L^2(\Omega))} \Norm{\Pi_{q-1}^t \uht}{L^{\infty}(0, T; L^2(\Omega))} \\
& \quad + \nu \Norm{\nabla \Upsilon}{L^2(\QT)^d} \Norm{\nabla \Pi_{q-1}^t \uht}{L^2(\QT)^d} \\
& \le \CS \Norm{f}{L^1(0, T; L^2(\Omega))} \Norm{ \uht}{L^{\infty}(0, T; L^2(\Omega))} \\
& \quad + \nu \Norm{\nabla \Upsilon}{L^2(\QT)^d}^2 + \frac{\nu}{4}\Norm{\nabla \Pi_{q-1}^t \uht}{L^2(\QT)^d}^2.
\end{alignat*}
Therefore, the following bound can be obtained by using the stability properties of~$\Pih$:
\begin{alignat*}{3}
\frac12 \Norm{\uht}{L^2(\ST)}^2 + \frac{\nu}{4} \Norm{\nabla \Pi_{q-1}^t \uht}{L^{2}(\QT)^d}^2 & \le  \frac12 \Norm{u_0}{L^2(\Omega)}^2 + \Norm{f}{L^1(0, T; L^2(\Omega))} \Norm{\uht}{L^{\infty}(0, T; L^2(\Omega))} \\
& \quad + \nu \Norm{\nabla \Upsilon}{L^2(\QT)^d}^2,
\end{alignat*}
which completes the proof. 
\end{proof}

\begin{proposition}[Continuous dependence on the data]
\label{prop:continuous-dependence-Aziz-Monk}
Any solution~$\uht \in \oVcont{q}$ to the discrete space--time formulation~\eqref{eq:CG-heat-perturbed} satisfies
\begin{alignat*}{3}
\Norm{\uht}{C^0([0, T]; L^2(\Omega))}^2 + \nu \Norm{\nabla \Pi_{q-1}^t \uht}{L^2(\QT)^d}^2 \lesssim \Norm{u_0}{L^2(\Omega)}^2 + \Norm{f}{L^1(0, T; L^2(\Omega))}^2 + \nu \Norm{\nabla \Upsilon}{L^2(\QT)^d}^2,
\end{alignat*}
where the hidden constant depends only on~$q$.
\end{proposition}
\begin{proof}
For~$n = 1, \ldots, N$, we choose the following test function in~\eqref{eq:CG-heat-perturbed}:
\begin{alignat*}{3}
\vht{}_{|_{Q_m}}:= \begin{cases}
\Pi_{q-1}^t (\varphi_n \Pi_{q-1}^t \uht) & \text{if } m = n,\\
0 & \text{otherwise},
\end{cases}
\end{alignat*}
and obtain the following identity:
\begin{equation}
\label{eq:identity-CG-heat-phin}
\Bht(\uht, \vht) = \big(f, \Pi_{q-1}^t (\varphi_n \Pi_{q-1}^t \uht)\big)_{\Qn} + \nu \big(\nabla \Upsilon, \nabla \Pi_{q-1}^t (\varphi_n \Pi_{q-1}^t \uht)\big)_{\Qn}.
\end{equation}

Using the orthogonality properties of~$\Pi_{q-1}^t$, we get
\begin{alignat*}{3}
\Bht(\uht, \vht) & = (\dpt \uht, \Pi_{q-1}^t (\varphi_n \Pi_{q-1}^t \uht))_{\Qn} + \nu (\nabla \uht, \nabla \Pi_{q-1}^t (\varphi_n \Pi_{q-1}^t \uht))_{\Qn}  \\
& = (\varphi_n \dpt \uht, \Pi_{q-1}^t \uht)_{\Qn}  + \nu (\varphi_n \nabla \Pi_{q-1}^t \uht, \nabla \Pi_{q-1}^t \uht)_{\Qn} \\
& = (\varphi_n \dpt \uht, \uht)_{\Qn} - (\varphi_n \dpt \uht, (\Id - \Pi_{q-1}^t) \uht)_{\Qn} + \nu (\varphi_n \nabla \Pi_{q-1}^t \uht, \nabla \Pi_{q-1}^t \uht)_{\Qn} \\
& =: J_1 + J_2 + J_3.
\end{alignat*}

Identity~\eqref{eq:identities-varphi-integration-2} from Lemma~\ref{lemma:identities-varphi-integration} gives
\begin{alignat}{3}
J_1 = (\varphi_n \dpt \uht, \uht)_{\Qn} 
\label{eq:J1-Aziz-Monk}
& = \frac14 \Norm{\uht}{L^2(\Sn)}^2 - \frac12 \Norm{\uht}{L^2(\Snmo)}^2 + \frac{\lambda_n}{2} \Norm{\uht}{L^2(\Qn)}^2. 
\end{alignat}

As for~$J_2$, it follows from Lemma~\ref{lemma:bound-pi-time-varphi_n} that
\begin{equation}
\label{eq:J2-Aziz-Monk}
J_2 = -(\varphi_n \dpt \uht, (\Id - \Pi_{q-1}^t) \uht)_{\Qn} \geq 0. 
\end{equation}

Moreover, since~$\varphi_n \geq 1/2$ in~$\In$, we have
\begin{equation}
\label{eq:J3-Aziz-Monk}
J_3 = \nu (\varphi_n \nabla \Pi_{q-1}^t \uht, \nabla \Pi_{q-1}^t \uht)_{\Qn} \geq \frac{\nu}{2} \Norm{\nabla \Pi_{q-1}^t \uht}{L^2(\Qn)^d}^2. 
\end{equation}

The terms on the right-hand side of~\eqref{eq:identity-CG-heat-phin} can be bounded using the H\"older inequality, the stability properties in Lemma~\ref{lemma:stab-pi-time} of~$\Pi_{q-1}^t$, and the Young inequality as follows:
\begin{alignat}{3}
\nonumber
\big(f, \Pi_{q-1}^t (\varphi_n \Pi_{q-1}^t \uht)\big)_{\Qn} & + \nu \big(\nabla \Upsilon, \nabla \Pi_{q-1}^t (\varphi_n \Pi_{q-1}^t \uht)\big)_{\Qn} \\
\nonumber
& \le \Norm{f}{L^1(\In; L^2(\Omega))} \Norm{\Pi_{q-1}^t (\varphi_n \Pi_{q-1}^t \uht)}{L^{\infty}(\In; L^2(\Omega))}  \\
\nonumber
& \quad + \nu \Norm{\nabla \Upsilon}{L^2(\Qn)^d}^2 + \frac{\nu}{4}\Norm{\Pi_{q-1}^t (\varphi_n \nabla \Pi_{q-1}^t \uht)}{L^2(\Qn)^d}^2 \\
\nonumber
& \le \CS^2 \Norm{f}{L^1(\In; L^2(\Omega))} \Norm{\uht}{L^{\infty}(\In; L^2(\Omega))}  \\
\label{eq:RHS-Aziz-Monk}
& \quad +  \nu \Norm{\nabla \Upsilon}{L^2(\Qn)^d}^2 + \frac{\nu}{4}\Norm{\nabla \Pi_{q-1}^t \uht}{L^2(\Qn)^d}^2.
\end{alignat}

Combining~\eqref{eq:J1-Aziz-Monk}, \eqref{eq:J2-Aziz-Monk}, \eqref{eq:J3-Aziz-Monk}, and~\eqref{eq:RHS-Aziz-Monk}, and using the weak partial bound from Lemma~\ref{lemma:weak-partial-bound-Aziz-Monk}, we obtain
\begin{alignat}{3}
\nonumber
\frac14 \Norm{\uht}{L^2(\Sn)}^2 & + \frac{\lambda_n}{2} \Norm{\uht}{L^2(\Qn)}^2 + \frac{\nu}{4} \Norm{\nabla \Pi_{q-1}^t \uht}{L^2(\Qn)^d}^2 \\
\nonumber
& \le \frac12 \Norm{\uht}{L^2(\Snmo)}^2 + \Norm{f}{L^1(\In; L^2(\Omega))} \Norm{\Pi_{q-1}^t (\varphi_n \Pi_{q-1}^t \uht)}{L^{\infty}(\In; L^2(\Omega))}  + \nu \Norm{\nabla \Upsilon}{L^2(\Qn)^d}^2 \\
\nonumber
& \le \frac12 \Norm{u_0}{L^2(\Omega)}^2 + \CS \Norm{f}{L^1(0, \tnmo; L^2(\Omega))} \Norm{\uht}{L^{\infty}(0, \tnmo; L^2(\Omega))} \\
\nonumber
& \quad + \CS^2  \Norm{f}{L^1(\In; L^2(\Omega))} \Norm{\uht}{L^{\infty}(\In; L^2(\Omega))} + \nu \Norm{\Upsilon}{L^2(0, \tn; L^2(\Omega)^d)}^2 \\
\label{eq:aux-bound-Aziz-Monk}
& \lesssim \Norm{u_0}{L^2(\Omega)}^2 + \Norm{f}{L^1(0, T; L^2(\Omega))} \Norm{\uht}{L^{\infty}(0, T; L^2(\Omega))}.
\end{alignat}

The inverse estimate in Lemma~\ref{lemma:L2-Linfty} gives
\begin{equation*}
\frac{1}{4 \Cinv} \Norm{\uht}{L^{\infty}(\In; L^2(\Omega))}^2 \le \frac{\lambda_n}{2} \Norm{\uht}{L^2(\Qn)}^2,
\end{equation*}
which, combined with~\eqref{eq:aux-bound-Aziz-Monk}, implies
\begin{alignat*}{3}
\Norm{\uht}{L^{\infty}(\In; L^2(\Omega))}^2 \lesssim \Norm{u_0}{L^2(\Omega)}^2 + \Norm{f}{L^1(0, T; L^2(\Omega))}^2 + \nu \Norm{\Upsilon}{L^2(\QT)^d}^2,
\end{alignat*}
where the hidden constant is independent of the time interval~$n$ and depends only on~$q$. A uniform bound on~$\Norm{\uht}{L^{\infty}(0, T; L^2(\Omega))}$ can then be obtained by taking~$n$ as the index where the left-hand side achieves its maximum value. This step, together with Lemma~\ref{lemma:weak-partial-bound-Aziz-Monk}, gives the desired result.
\end{proof}
The existence of a unique solution to~\eqref{eq:Aziz-Monk} is an immediate consequence of Proposition~\ref{prop:continuous-dependence-Aziz-Monk}, as it corresponds to a linear system with a nonsingular matrix.

\subsection{Convergence analysis}
Recalling Definition~\ref{def:Pt-AM} of the Aziz--Monk projection operator~$\PtAM$, we define the space--time projection~$\Piht = \PtAM \circ \Rh$, and the error functions
\begin{equation*}
\eu := u - \uht  = (\Id - \Piht) u + \Piht \eu. 
\end{equation*}

\begin{theorem}[Estimates for the discrete error]
\label{thm:a-priori-Aziz-Monk}
Let~$u \in X$ be the solution to the continuous weak formulation~\eqref{eq:weak-formulation-heat}, and let~$\uht \in \oVcont{q}$ be the solution to~\eqref{eq:Aziz-Monk}. Assume that~$\Omega$ satisfies the elliptic regularity assumption~\eqref{eq:elliptic-regularity-Omega}, and that the continuous solution~$u$ satisfies
\begin{equation*}
u \in X \cap W^{1, 1}(0, T; H^{r+1}(\Omega)) \cap H^{s+1}(0, T; H^1(\Omega)),
\end{equation*}
for some~$0 \le r \le p$ and~$0 \le s \le q$. 
Then, the following estimate holds:
\begin{equation*}
\begin{split}
&\Norm{\Piht \eu}{C^0([0, T]; L^2(\Omega))}  + \sqrt{\nu} \Norm{\nabla \Pi_{q-1}^t \Piht \eu}{L^2(\QT)} \\
& \qquad \lesssim h^{r+1} \Seminorm{\dpt u}{L^1(0, T; H^{r+1}(\Omega))} + \sqrt{\nu} \tau^{s+1} \Norm{\nabla \dpt^{(s+1)} u}{L^2(\QT)^d} + h^{r+1} \Seminorm{u_0}{H^{r+1}}.
\end{split}
\end{equation*}
\end{theorem}
\begin{proof}
Since the discrete space--time formulation is consistent and~$\oVdisc{q-1} \subset Y$, we have
\begin{equation*}
    \Bht(\Piht \eu, \vht) = -\Bht((\Id - \Piht) u, \vht) \quad \forall \vht \in \oVdisc{q-1}.
\end{equation*}
The orthogonality properties of~$\Rh$ and~$\PtAM$ lead to
\begin{alignat*}{3}
\Bht(\Piht \eu, \vht)
& = -(\dpt (\Id - \Piht)u, \vht)_{\QT} - \nu (\nabla (\Id - \Piht) u, \nabla \vht)_{\QT} \\
& = - ((\Id - \Rh) \dpt u, \vht)_{\QT} - \nu ((\Id - \PtAM) \nabla u, \nabla \vht)_{\QT}.
\end{alignat*}
Therefore, the discrete error function~$\Piht \eu$ solves a problem of the form~\eqref{eq:CG-heat-perturbed} with data
\begin{alignat*}{3}
\widetilde{f} = - (\Id - \Rh) \dpt u, \  \Upsilon = -(\Id - \PtAM) u, \ \text{ and } \ \Piht \eu(\cdot, 0) = \Rh u_0 - \Pih u_0 = - \Pih (\Id - \Rh u_0),
\end{alignat*}
and the continuous dependence on the data in Proposition~\ref{prop:continuous-dependence-Aziz-Monk} then gives
\begin{alignat*}{3}
& \Norm{\Piht \eu}{C^0([0, T]; L^2(\Omega))} + \sqrt{\nu} \Norm{\nabla \Pi_{q-1}^t \Piht \eu}{L^2(\QT)^d} \\
& \quad \lesssim \Norm{(\Id - \Rh) \dpt u}{L^1(0, T; L^2(\Omega))} + \sqrt{\nu} \Norm{(\Id - \PtAM) \nabla u}{L^2(\QT)^d} + \Norm{(\Id - \Rh) u_0}{L^2(\Omega)}.
\end{alignat*}
The desired estimate follows from the approximation properties in Lemmas~\ref{lemma:Estimates-Rh} (relying on Lemma~\ref{lemma:stab-PtAM}) and~\ref{lemma:estimates-Ptau} of~$\Rh$ and~$\PtAM$, respectively. 
\end{proof}

\begin{corollary}[\emph{A priori} error estimates]
Under the assumptions of Theorem~\ref{thm:a-priori-Aziz-Monk}, and assuming that the continuous solution~$u$ to~\eqref{eq:weak-formulation-heat} is sufficiently regular, there hold
\begin{alignat*}{3}
\Norm{\eu}{C^{0}([0, T]; L^2(\Omega))} & \lesssim \tau^{q+1} \big(\Norm{\dpt^{(q+1)} u}{C^0([0, T]; L^2(\Omega))} + \sqrt{\nu} \Norm{\nabla \dpt^{(q+1)} u}{L^2(\QT)^d}\big)\\
& \quad + h^{p+1} \big(\Norm{u}{C^0([0, T]; H^{p+1}(\Omega))} + \Seminorm{\dpt u}{L^1(0, T; H^{p+1}(\Omega))} + \Seminorm{u_0}{H^{p+1}(\Omega)} \big), \\
\sqrt{\nu} \Norm{\nabla \Pi_{q-1}^t \eu}{L^2(\QT)^d} & \lesssim \sqrt{\nu} \tau^{q+1} \Norm{\nabla \dpt^{(q+1)} u}{L^2(\QT)^d} \\
& \quad + h^{p} \big(\sqrt{\nu} \Seminorm{u}{L^2(0, T; H^{p+1}(\Omega))} + \Seminorm{\dpt u}{L^1(0, T; H^p(\Omega))} + \Seminorm{u_0}{H^p(\Omega)} \big).
\end{alignat*}
\end{corollary}

\begin{remark}[Inf--sup approach]
The classical inf--sup analysis for parabolic problems has been
extended to the fully discrete setting with CG time discretizations in~\cite[\S71.4]{Ern_Guermond:2021}, with a complete a priori error analysis  in~\cite{Loscher_Reichelt_Steinbach:2026}.
\eremk
\end{remark}
\begin{remark}[Relation between the DG and CG time discretizations]
In the recent work~\cite{Cockburn:2025}, it was shown that the DG and the CG discretizations of the time derivative are ``exactly the same".
There is a very subtle but important difference with respect to saying that the DG and CG time discretizations of the whole problem are the same, as that only happens if a suitable nontrivial interpolant operator is used (see~\cite[Cor. 3.1]{Cockburn:2025}).
\eremk
\end{remark}

\section{The ``plain vanilla" DG time discretization for the second-order formulation of the wave equation (Hughes--Hulbert, 1988)}
\label{sec:plain-vanilla-wave}
The most natural DG time discretization for the second-order formulation of the wave equation~\eqref{eq:second-order-wave} would be: 
find~$\uht \in \oVdisc{q}$ (with~$q \geq 2$) such that
\begin{equation}
\label{eq:plain-DG-wave}
    \Bht(\uht, \wht) = \ell(\wht) \qquad \forall \wht \in \oVdisc{q},
\end{equation}
where the bilinear form~$\Bht : \oVdisc{q} \times \oVdisc{q} \to \R$ and the linear functional~$\ell : \oVdisc{q} \to \R$ are given by
\begin{alignat*}{3}
\nonumber
\Bht(\uht, \wht)& :=  \sum_{n = 1}^{N} (\dptt \uht, \dpt \wht)_{\Qn} + \sum_{n = 1}^{N - 1} (\jump{\dpt \uht}_{n}, \dpt \wht(\cdot, \tn^+))_{\Omega}
 + (\dpt \uht, \dpt\wht)_{\SO} \\
 \nonumber
& \quad + c^2 \sum_{n = 1}^{N}  (\nabla \uht, \nabla \dpt \wht)_{\Qn} + c^2  \sum_{n = 1}^{N-1} (\jump{\nabla \uht}_n, \nabla \wht(\cdot, \tn^+))_{\Omega} \\
& \quad + c^2 (\nabla \uht, \nabla \wht)_{\SO}, \\
\ell(\wht) &  := \sum_{n = 1}^N (f, \dpt \wht)_{\Qn} + (v_0, \dpt \wht(\cdot, 0))_{\Omega} + c^2 (\nabla u_0, \nabla \wht (\cdot, 0))_{\Omega}.
\end{alignat*}
This method was presented for the first time in~\cite[\S5]{Hughes_Hulbert:1988} as the \emph{single-field Galerkin formulation}, in the context of elastodynamics. It can be derived by multiplying the wave equation by a test function~$\dpt \wht$, and using standard upwind DG numerical fluxes in time. 

The definition of the time reconstruction operator~$\Rt$ in Section~\ref{sec:time-reconstruction} allows us to rewrite the bilinear form~$\Bht(\cdot, \cdot)$ in compact form as follows:
\begin{equation*}
\begin{split}
\Bht(\uht, \wht) & = (\dpt \Rtq \dptau \uht, \dptau \wht)_{\QT} + c^2 (\nabla \uht, \nabla \dptau \wht)_{\QT} \\
& \quad + c^2(\nabla \dptau(\Rt \uht - \uht), \nabla \wht)_{\QT},
\end{split}
\end{equation*}
where~$\dptau$ denotes the broken first-order time derivative in~$\Tt$.

The forthcoming stability and convergence analysis follows the structure described in Section~\ref{sec:structure-theory}.
\subsection{Stability analysis}
\label{sec:stab-plain-DG-wave}
\begin{lemma}[Weak partial bound on the discrete solution]
\label{lemma:weak-bound-plain-vanilla-DG}
Any solution to~\eqref{eq:plain-DG-wave} satisfies the following bound for~$n = 1, \ldots, N$:
\begin{alignat*}{3}
\frac12  \Norm{\dptau \uht(\cdot, \tn^-)}{L^2(\Omega)}^2  +\frac{c^2}{2} \Norm{\nabla \uht(\cdot, \tn^-)}{L^2(\Omega)^d}^2 & + \frac12 \sum_{m = 1}^{n - 1} \Big(\Norm{\jump{\dptau \uht}_m}{L^2(\Omega)}^2 + c^2 \Norm{\jump{\nabla \uht}_m}{L^2(\Omega)^d}^2 \Big) \\
+ \frac14 \Norm{\dptau \uht}{L^2(\SO)}^2 + \frac{c^2}{4} \Norm{\nabla \uht}{L^2(\SO)^d}^2 & \le \Norm{f}{L^1(0, T; L^2(\Omega))} \Norm{\dptau \uht}{L^{\infty}(0, T; L^2(\Omega))} \\
& \quad + \Norm{v_0}{L^2(\Omega)}^2 + c^2 \Norm{\nabla u_0}{L^2(\Omega)^d}^2. 
\end{alignat*}
\end{lemma}
\begin{proof}
Without loss of generality, we prove the result for~$n = N$. Taking~$\wht = \uht$ in~\eqref{eq:plain-DG-wave} gives
\begin{equation}
\label{eq:identity-weak-partial-bound-plain-vanilla-DG}
\begin{split}
(\dpt \Rtq \dptau \uht, \dptau \uht)_{\QT} & + c^2 (\nabla \uht, \nabla \dptau \uht)_{\QT} + c^2(\nabla \dptau (\Rt \uht - \uht), \nabla \uht)_{\QT} \\
& = (f, \dptau \uht)_{\QT} + (v_0, \dptau \uht(\cdot, 0))_{\Omega} + c^2 (\nabla u_0, \nabla \uht(\cdot, 0))_{\Omega}.
\end{split}
\end{equation}
Using the symmetry of the~$L^2(\QT)^d$-inner product and identity~\eqref{eq:Rt-vht-jump} on the left-hand side of~\eqref{eq:identity-weak-partial-bound-plain-vanilla-DG} yields
\begin{alignat*}{3}
    (\dpt \Rtq \dptau \uht, \dptau \uht)_{\QT} & + c^2 (\nabla \uht, \nabla \dptau \uht)_{\QT} + c^2(\nabla \dptau (\Rt \uht - \uht), \nabla \uht)_{\QT} \\
    & = (\dpt \Rtq \dptau \uht, \dptau \uht)_{\QT} + c^2 (\nabla \dpt \Rt \uht, \nabla \uht)_{\QT} \\
    & = \frac12 \Seminorm{\dptau \uht}{\sf J}^2 + \frac12 \Seminorm{\nabla \uht}{\sf J}^2,
\end{alignat*}
which, combined with~\eqref{eq:identity-weak-partial-bound-plain-vanilla-DG} and the H\"older and the Young inequalities, completes the proof. 
\end{proof}
The ``unconditional stability" shown in~\cite{Costanzo_Huang:2005} simply refers to the fact that, for~$f = 0$, Lemma~\ref{lemma:weak-bound-plain-vanilla-DG} gives 
\begin{equation*}
\frac12 \Norm{\dpt \uht}{L^2(\ST)}^2 + \frac{c^2}{2} \Norm{\nabla \uht}{L^2(\ST)^d}^2 \le \frac12 \Norm{v_0}{L^2(\Omega)}^2 + \frac{c^2}{2} \Norm{\nabla u_0}{L^2(\Omega)^d}^2,
\end{equation*}
but this is not enough to prove the existence of discrete solutions. 

Proposition~\ref{prop:continuous-dependence-DG-wave} states that, assuming a Courant--Friedrichs--Lewy (CFL) condition of the form
\begin{equation}
\label{eq:CFL-DG-wave}
    \tau \le \frac{\CCFL h_{\min}}{c},
\end{equation}
where~$h_{\min}$ denotes the minimum element diameter of~$\Th$, the space--time formulation~\eqref{eq:plain-DG-wave} is stable, which also guarantees the existence of a unique solution. To the best of our knowledge, such a result is not available in the literature. 

The proof of Proposition~\ref{prop:continuous-dependence-DG-wave} involves the following nonstandard test function:
\begin{equation}
\label{eq:test-function-plain-DG}
\wht(\cdot, t)  = \uht(\cdot, \tnmo^+) + \int_{\tnmo}^t \Ptdqmo(\varphi_n \dps \uht) \ds \quad \text{ in~$\Omega$,}
\end{equation}
where~$\Ptdqmo$ is the left-sided Thom\'ee projection of degree~$q - 1$, which interpolates at~$t = \tnmo^+$ instead of~$t = \tn^-$ (see Definition~\ref{def:Pt}). This test function satisfies the following three important identities (for~$q \geq 2)$:
\begin{subequations}
\label{eq:identities-test-function-DG-wave}
\begin{alignat}{3}
\wht(\cdot, \tnmo^+) & = \uht(\cdot, \tnmo^+), \\
\dpt \wht(\cdot, \tnmo^+) & = \Ptdqmo (\varphi_n \dpt \uht)(\cdot, \tnmo^+) = \dpt \uht(\cdot, \tnmo^+), \\
(\dptt \uht, \dpt \wht)_{\Qn} & = (\dptt \uht, \Ptdqmo (\varphi_n \dpt \uht))_{\Qn} = (\dptt \uht, \varphi_n \dpt\uht)_{\Qn}.
\end{alignat}
\end{subequations}

\begin{proposition}[Continuous dependence on the data]
\label{prop:continuous-dependence-DG-wave}
There exists a positive constant~$\CCFL$ depending only on~$p$ and~$q$ such that, if~$\tau \le \frac{\CCFL h_{\min}}{c}$, then any solution~$\uht \in \oVdisc{q}$ to the discrete space--time formulation~\eqref{eq:plain-DG-wave} satisfies
\begin{alignat*}{3}
\Norm{\dptau \uht}{L^{\infty}(0, T; L^2(\Omega))}^2 + c^2 \Norm{\nabla \uht}{L^{\infty}(0, T; L^2(\Omega)^d)}^2 + \Seminorm{\dptau \uht}{\sf J}^2  + c^2 \Seminorm{\nabla \uht}{\sf J}^2 \\
\lesssim \Norm{f}{L^1(0, T; L^2(\Omega))}^2 + \Norm{v_0}{L^2(\Omega)}^2 + c^2 \Norm{\nabla u_0}{L^2(\Omega)^d}^2,
\end{alignat*}
where the hidden constant depends only on~$q$.
\end{proposition}
\begin{proof}
For~$n = 1, \ldots, N$, we choose~$\wht$ in~\eqref{eq:plain-DG-wave} as
\begin{equation*}
\wht{}_{|_{\Qm}} := \begin{cases}
\uht(\cdot, \tnmo^+) + \int_{\tnmo}^t \Ptdqmo (\varphi_n \dps \uht) \ds & \text{if~$m = n$},\\
0 & \text{otherwise}.
\end{cases}
\end{equation*}
As usual, we focus on the case~$n > 1$, as the argument for~$n = 1$ requires only minor modifications. Using the identities in~\eqref{eq:identities-test-function-DG-wave}, we get
\begin{equation}
\label{eq:aux-id-1}
\begin{split}
(\dptt \uht, \varphi_n \dpt \uht)_{\Qn}  + (\jump{\dptau \uht}_{n-1}, &  \dptau \uht(\cdot, \tnmo^+))_{\Omega}  + c^2(\nabla \uht, \nabla \Ptdqmo (\varphi_n \dpt \uht))_{\Qn} \\
+ c^2 (\jump{\nabla \uht}_{n-1}, \nabla \uht(\cdot, \tnmo^+))_{\Omega} & = (f, \Ptdqmo(\varphi_n \dpt \uht))_{\Qn}.
\end{split}
\end{equation}

Identity~\eqref{eq:identities-varphi-integration-1} from Lemma~\ref{lemma:identities-varphi-integration} gives
\begin{equation*}
\begin{split}
(\dptt \uht, \varphi_n \dpt \uht)_{\Qn} + (\jump{\dptau \uht}_{n-1}, \dptau \uht(\cdot, \tnmo^+))_{\Omega} & = \frac14 \Norm{\dptau \uht(\cdot, \tn^-)}{L^2(\Omega)}^2 + \frac12 \Norm{\jump{\dptau \uht}_{n-1}}{L^2(\Omega)}^2 \\
& \quad - \frac12 \Norm{\dptau \uht(\cdot, \tnmo^-)}{L^2(\Omega)}^2 + \frac{\lambda_n}{2} \Norm{\dpt \uht}{L^2(\Qn)}^2. 
\end{split}
\end{equation*}

As for the third term on the left-hand side of~\eqref{eq:aux-id-1}, we have
\begin{alignat*}{3}
c^2 (\nabla \uht, \nabla \Ptdqmo (\varphi_n \dpt \uht))_{\Qn} & = c^2 (\nabla \uht, \varphi_n \nabla \dpt \uht)_{\Qn} \\
& \quad - c^2(\nabla \uht, \nabla (\Id - \Ptdqmo) (\varphi_n \dpt \uht))_{\Qn} \\
& =: J_1 + J_2.
\end{alignat*}

As before, identity~\eqref{eq:identities-varphi-integration-1} leads to
\begin{equation*}
J_1 = \frac{c^2}{4} \Norm{\nabla \uht(\cdot, \tn^-)}{L^2(\Omega)^d}^2 + \frac{c^2}{2} \Norm{\jump{\nabla \uht}_{n-1}}{L^2(\Omega)^d}^2 - \frac{c^2}{2} \Norm{\nabla \uht(\cdot, \tnmo^-)}{L^2(\Omega)^d}^2 + \frac{c^2 \lambda_n}{2} \Norm{\nabla \uht}{L^2(\Qn)^d}^2. 
\end{equation*}

Using the Young inequality, an inverse estimate in space with constant~$\Cinv(p)$, the H\"older inequality, and the stability of~$\Ptdqmo$ in the~$L^{\infty}(\In; L^2(\Omega))$ norm with constant~$C_{\mathcal{P}_{\tau}}(q)$, we get
\begin{alignat}{3}
\nonumber
J_2 & \geq -\frac{c^2}{8\tau_n} \Norm{\nabla \uht}{L^2(\Qn)^d}^2 - 2 c^2 \tau_n  \Norm{\nabla (\Id - \Ptdqmo) (\varphi_n \dpt \uht)}{L^2(\Qn)^d}^2 \\
\label{eq:J2-plain-DG-wave}
& \geq - \frac{c^2}{8\tau_n} \Norm{\nabla \uht}{L^2(\Qn)^d}^2 - \frac{4 c^2 \tau_n^2 (1+ C_{\mathcal{P}_{\tau}}^2(q)) \Cinv^2(p)}{h_{\min}^2} \Norm{\dpt \uht}{L^{\infty}(\In; L^2(\Omega))}^2. 
\end{alignat}
The CFL condition~\eqref{eq:CFL-DG-wave} is then requested to get
\begin{alignat*}{3}
\frac{4 c^2 \tau_n^2 (1+ C_{\mathcal{P}_{\tau}}^2(q)) \Cinv^2(p)}{h_{\min}^2} \le \frac{1}{8 \Cinv(q)},
\end{alignat*}
where~$\Cinv(q)$ is the constant in Lemma~\ref{lemma:L2-Linfty}.

The following bound is then obtained from the inverse estimate in Lemma~\ref{lemma:L2-Linfty}, the weak partial bound in Lemma~\ref{lemma:weak-bound-plain-vanilla-DG}, and the stability properties of~$\Ptdqmo$:
\begin{alignat*}{3}
\frac14 & \Norm{\dpt \uht(\cdot, \tn^-)}{L^2(\Omega)}^2  + \frac{c^2}{4} \Norm{\nabla \uht(\cdot, \tn^-)}{L^2(\Omega)^d}^2 + \frac12 \Norm{\jump{\dptau \uht}_{n-1}}{L^2(\Omega)}^2 + \frac{c^2}{2} \Norm{\jump{\nabla \uht}_{n-1}}{L^2(\Omega)^d}^2 \\
& \quad + \frac{1}{8 \Cinv(q)} \Norm{\dpt \uht}{L^{\infty}(\In; L^2(\Omega))}^2 + \frac{c^2}{8\Cinv(q)} \Norm{\nabla \uht}{L^{\infty}(\In; L^2(\Omega)^d)}^2 \\
& \le \frac12 \Norm{\dptau \uht(\cdot, \tnmo^-)}{L^2(\Omega)}^2 + \frac{c^2}{2} \Norm{\nabla \uht(\cdot, \tnmo^-)}{L^2(\Omega)^d}^2 \\
& \quad + \Norm{f}{L^1(\In; L^2(\Omega))} \Norm{\Ptdqmo(\varphi_n \dpt \uht)}{L^{\infty}(\In; L^2(\Omega))} \\
& \lesssim \Norm{v_0}{L^2(\Omega)}^2 + c^2 \Norm{\nabla u_0}{L^2(\Omega)^d}^2 + \Norm{f}{L^1(0, \tnmo; L^2(\Omega))} \Norm{\dptau \uht}{L^{\infty}(0, \tnmo; L^2(\Omega))}  \\
& \quad + \Norm{f}{L^1(\In; L^2(\Omega))} \Norm{\dpt \uht}{L^{\infty}(\In; L^2(\Omega))} \\
& \lesssim \Norm{v_0}{L^2(\Omega)}^2 + c^2 \Norm{\nabla u_0}{L^2(\Omega)^d}^2 + \Norm{f}{L^1(0, T; L^2(\Omega))} \Norm{\dptau \uht}{L^{\infty}(0, T; L^2(\Omega))},
\end{alignat*}
which implies
\begin{alignat*}{3}
\Norm{\dpt \uht}{L^{\infty}(\In; L^2(\Omega))}^2 + c^2 \Norm{\nabla \uht}{L^{\infty}(\In; L^2(\Omega)^d)}^2 & \lesssim \Norm{v_0}{L^2(\Omega)}^2 + c^2 \Norm{\nabla u_0}{L^2(\Omega)^d}^2 \\
& \quad + \Norm{f}{L^1(0, T; L^2(\Omega))} \Norm{\dptau \uht}{L^{\infty}(0, T; L^2(\Omega))}.
\end{alignat*}
For each of the terms on the left-hand side, one can choose an index~$n$ where it achieves its maximum value. This, combined with  Lemma~\ref{lemma:weak-bound-plain-vanilla-DG}, gives the desired result. 
\end{proof}
The existence of a unique solution to~\eqref{eq:plain-DG-wave} is an immediate consequence of Proposition~\ref{prop:continuous-dependence-DG-wave}, as it corresponds to a linear system with a nonsingular matrix provided that the CFL condition~\eqref{eq:CFL-DG-wave} holds.

\subsection{Convergence analysis}
\label{sec:convergence-plain-DG-wave}
Recalling Definition~\ref{def:PtW} of the projection operator~$\PtW$, we define the space--time projection~$\Piht = \PtW \circ \Rh$ and the error functions
\begin{equation*}
    \eu = u - \uht = (\Id - \Piht) u + \Piht \eu.
\end{equation*}
In the proof of the following theorem, we rely on the fact that~$\Piht v$ is continuous in time provided that~$v$ is continuous in time.
\begin{theorem}[Estimates for the discrete error]
\label{thm:a-priori-plain-DG-wave}
Let the assumptions of Proposition~\ref{prop:continuous-dependence-DG-wave} hold. Let~$u$ be the continuous weak solution to~\eqref{eq:second-order-weak-formulation-wave}, and let~$\uht \in \oVdisc{q}$ be the solution to~\eqref{eq:plain-DG-wave}. Assume that~$\Omega$ satisfies the elliptic regularity assumption~\eqref{eq:elliptic-regularity-Omega}, and that~$u$ satisfies
\begin{equation*}
u \in H^1(0, T; H_0^1(\Omega)) \cap W^{2, 1}(0, T; H^{r+1}(\Omega)) \quad \text{ with } \quad \Delta u \in W^{s+1, 1}(0, T; L^2(\Omega)), 
\end{equation*}
for some~$0 \le r \le p$ and~$1 \le s \le q$. 
Then, the following estimate holds:
\begin{alignat*}{3}
&\Norm{\dptau \Piht \eu}{L^{\infty}(0, T; L^2(\Omega))}  + c\Norm{\nabla \Piht \eu}{L^{\infty}(0, T; L^2(\Omega)^d)} + \Seminorm{\dptau \Piht \eu}{\sf J} + c \Seminorm{\nabla \Piht \eu}{\sf J} \\
& \qquad \lesssim h^{r+1} \big(\Seminorm{\dptt u}{L^1(0, T; H^{r+1}(\Omega))} + \Seminorm{v_0}{H^{r+1}(\Omega)}\big) + c^2 \tau^{s+1} \Norm{\Delta \dpt^{(s+1)} u}{L^1(0, T; L^2(\Omega))}. 
\end{alignat*}
\end{theorem}
\begin{proof}
The space--time formulation~\eqref{eq:plain-DG-wave} is consistent for the regularity assumptions made on~$u$. Therefore, we have the following error equation:
\begin{equation}
\label{eq:error-identity-plain-DG}
\Bht(\Piht \eu, \wht) = - \Bht((\Id - \Piht) u, \wht) \qquad \forall \wht \in \oVdisc{q}. 
\end{equation}

Using the orthogonality property of~$\PtW$ from Lemma~\ref{lemma:orthogonality-DG-PtW}, the definition of the Ritz projection~$\Rh$, the continuity in time of~$\Piht u$, and the fact that~$\PtW v(0) = v(0)$, we can simplify the terms on the right-hand side of~\eqref{eq:error-identity-plain-DG} as follows:
\begin{alignat*}{3}
- \big(\dpt \Rtq \dptau (\Id - \Piht ) u, \dptau \wht)_{\QT} & = - ((\Id - \Rh) \dptt u, \dptau \wht)_{\QT} \\
& \quad - ((\Id - \Rh) v_0, \dpt \wht(\cdot, 0))_{\Omega}, \\
- c^2 (\nabla (\Id - \Piht) u, \nabla \dptau \wht)_{\QT} & = -c^2 ((\Id - \PtW) \nabla u, \nabla \dptau \wht)_{\QT} \\
&  = c^2 ((\Id - \PtW) \Delta u, \dptau \wht)_{\QT}, \\
- c^2 \sum_{n = 1}^{N - 1} (\jump{\nabla (\Id - \Piht) u}_n, \nabla \wht(\cdot, \tn^+))_{\Omega} 
& = 0, \\
- c^2 (\nabla (\Id - \Piht) u, \nabla \wht)_{\SO} & = 0.
\end{alignat*}
Therefore, the discrete error~$\Piht \eu$ satisfies a problem of the form~\eqref{eq:plain-DG-wave} with data
\begin{equation*}
\widetilde{f} = - (\Id - \Rh) \dptt u + c^2 (\Id - \PtW) \Delta u, \quad \widetilde{v}_0 = - (\Id - \Rh) v_0, \quad \widetilde{u}_0 = 0. 
\end{equation*}
The estimate in the statement of this theorem then follows by Proposition~\ref{prop:continuous-dependence-DG-wave}, and the approximation properties of~$\PtW$ and~$\Rh$ in Lemma~\ref{lemma:estimates-PtW} and~\ref{lemma:Estimates-Rh}, respectively. 
\end{proof}

\begin{corollary}[\emph{A priori} error estimates]
Under the assumptions of Theorem~\ref{thm:a-priori-plain-DG-wave}, and assuming that the continuous solution~$u$ to~\eqref{eq:second-order-weak-formulation-wave} is sufficiently regular, there hold
\begin{alignat*}{3}
\Norm{\dptau \eu }{L^{\infty}(0, T; L^2(\Omega))} & \lesssim h^{p+1} \Big(\Seminorm{\dptt u}{L^1(0, T; H^{p+1}(\Omega))} + \Seminorm{v_0}{H^{p+1}(\Omega)} + \Seminorm{\dpt u}{L^{\infty}(0, T; H^{p+1}(\Omega))} \Big) \\
& \quad + \tau^{q} \big( c^2 \Norm{\Delta \dpt^{(q)} u}{L^1(0, T; L^2(\Omega))} + \Norm{\dpt^{(q + 1)} u}{L^{\infty}(0, T; L^2(\Omega))}\big), \\
c \Norm{\nabla \eu}{L^{\infty}(0, T; L^2(\Omega)^d)} & \lesssim  h^{p} \big(\Seminorm{\dptt u}{L^1(0, T; H^p(\Omega))} + \Seminorm{v_0}{H^p(\Omega)} + c \Seminorm{u}{L^{\infty}(0, T; H^{p+1}(\Omega))} \big) \\
& \quad + \tau^{q+1} \big(c^2 \Norm{\Delta \dpt^{(q+1)} u}{L^1(0, T; L^2(\Omega))} + c \Norm{\nabla \dpt^{(q+1)} u}{L^{\infty}(0, T; L^2(\Omega)^d)} \big).
\end{alignat*}
\end{corollary}

\subsection{A least-squares stabilization (Hulbert--Hughes, 1988)}
Some oscillations were observed in the numerical experiments in~\cite{Hulbert_Hughes:1990} for the space--time method~\eqref{eq:plain-DG-wave} applied to the elastodynamics equation. 
In~\cite{Hughes_Hulbert:1988}, the authors introduced a space--time method for the Hamiltonian formulation~\eqref{eq:first-order-wave} with additional least-squares stabilization terms. For a particular combination of the degrees of approximation of the two fields, such a method reduces to a single-field formulation with a least-squares stabilization term. For the wave equation~\eqref{eq:second-order-wave}, denoting by~$\mathcal{F}_h^{\mathcal{I}}$ the set of the interior facets of~$\Th$, this method corresponds to adding the following term to the left-hand side of~\eqref{eq:plain-DG-wave}:
$$ \sum_{n = 1}^{N} \int_{\In}  \Big(\sum_{K \in \Th} \int_{K} \Upsilon(h) (\mathcal{L} (\uht) - f) \mathcal{L}(\wht) \dx  +\sum_{F \in \mathcal{F}_h^{\mathcal{I}}} \zeta \int_F \jump{\nabla \uht}_{\sf N} \jump{\nabla \wht}_{\sf N} \dS \Big)\dt, $$
where~$\mathcal{L}(\varphi) := \dptt \varphi - c^2 \Delta \varphi$, $\Upsilon(h) := h/(2c)$, $\zeta := c/2$, and~$\jump{\cdot}_{\sf N}$ denotes the spatial DG normal jump operator of~$d$-vector-valued functions; see the remark at the end of~\cite[\S4]{Hughes_Hulbert:1988}.

In this setting, the resulting bilinear form becomes coercive with respect to the norm
\begin{equation*}
\frac12 \Seminorm{\dptau \uht}{\sf J}^2  + \frac{c^2}{2} \Seminorm{\nabla \uht}{\sf J}^2 + \Upsilon(h) \sum_{n = 1}^{N} \sum_{K \in \Th} \Norm{\mathcal{L}(\uht)}{L^2(\In; L^2(K))}^2 + \zeta \Norm{\jump{\nabla \uht}_{\sf N}}{L^2(\mathcal{F}_h^{\mathcal{I}})}^2. 
\end{equation*}
The estimate in~\cite[Thm.~4.7]{Hughes_Hulbert:1988} predicts suboptimal convergence rates of order~$\mathcal{O}(\widetilde{h}^{2p-1})$ (where~$\widetilde{h} = \max\{h, c\tau\}$ and~$p = q$) for the discrete error in this energy norm.

\subsection{Inner products with exponential weights (French, 1993)}
The use of test functions involving (approximated) exponential weights is  useful to derive robust stability and \emph{a priori} error estimates. This is related to the fact that the change of variables~$u = e^{\delta t} \rho$ in the second-order formulation of the wave equation~\eqref{eq:second-order-wave} leads to the following problem:
\begin{alignat*}{3}
\dptt \rho + 2 \delta \dpt \rho + \delta^2 \rho - c^2 \Delta \rho & = f e^{-\delta t} & & \quad \text{ in } \QT, \\
\rho & = 0 &  & \quad \text{ on } \partial \Omega \times (0, T), \\
\rho = u_0  \quad \text{ and } \quad \dpt \rho & = v_0 - \delta u_0  & &\quad \text{ on } \SO.
\end{alignat*}
This transformation introduces ``nice" damping terms for~$\delta > 0$, which simplify the analysis (see Section~\ref{sec:damped-wave} below). This approach has been used to motivate the space--time method~\eqref{eq:plain-DG-wave} for the damped wave equation (see, e.g., \cite[Rem.~2.3]{Shao:2022}).
However, returning to the original unknowns reveals that the errors suffer from exponential amplification in time. More precisely, setting~$\widetilde{u}_{h\tau} = e^{\delta t} \rho_{h\tau}$, it holds
\begin{alignat*}{3}
\Norm{u - \widetilde{u}_{h\tau}}{} = \Norm{e^{\delta t} (\rho - \rho_{h\tau})}{},
\end{alignat*}
for any norm~$\Norm{\cdot}{}$.

This makes it clear that it is better to avoid such a change of variables and rather exploit the ``intrinsic" structure of the method. An alternative approach was proposed in~\cite{French:1993}, which consists of introducing exponential weight functions explicitly. 
More precisely, defining 
$$\phi_n = e^{-\delta(t - \tnmo)},$$ 
for~$n = 1, \ldots, N$, the method reads: find~$\uht \in \oVdisc{q}$ such that
\begin{equation*}
    \Bht(\uht, \wht) = \ell(\wht) \quad \forall \wht \in \oVdisc{q},
\end{equation*}
where
\begin{alignat*}{3}
\Bht(\uht, \wht) & = \sum_{n = 1}^N (\dptt \uht, \phi_n \dpt \wht)_{\Qn} + \sum_{n = 1}^{N - 1} (\jump{\dpt \uht}_n, \dpt \wht(\cdot, \tn))_{\Omega} + (\dpt \uht, \dpt \wht)_{\SO} \\
& \quad + c^2 \sum_{n = 1}^N (\nabla \uht, \phi_n \nabla \dpt \wht)_{\Qn} + \sum_{n = 1}^{N - 1} c^2 (\jump{\nabla \uht}_n, \nabla \wht(\cdot, \tn^+))_{\Omega} \\
& \quad + c^2 (\nabla \uht, \nabla \wht)_{\SO}, \\
\ell(\wht) &  = \sum_{n = 1}^N (f, \phi_n \dpt \wht)_{\Qn} + (v_0, \dpt \wht(\cdot, 0))_{\Omega} + c^2 (\nabla u_0, \nabla \wht(\cdot, 0))_{\Omega}. 
\end{alignat*}

Choosing~$\wht = \chi_{(\tnmo, \tn)} \uht$ and proceeding as usual, it holds
\begin{equation*}
\begin{split}
& \frac{e^{-\delta \tau_n}}{2} \Norm{\dptau \uht(\cdot, \tn^-)}{L^2(\Omega)}^2 + \frac12 \Norm{\jump{\dptau \uht}_{n-1}}{L^2(\Omega)}^2 + \frac{\delta}{2} \Norm{\sqrt{\varphi_n} \dpt \uht}{L^2(\Qn)}^2 \\
& + \frac{c^2 e^{-\delta \tau_n}}{2} \Norm{\nabla \uht(\cdot, \tn^-)}{L^2(\Omega)^d}^2  + \frac{c^2}{2} \Norm{\jump{\nabla \uht}_{n-1}}{L^2(\Omega)^d}^2 + \frac{\delta  c^2}{2} \Norm{\sqrt{\varphi_n} \nabla \uht}{L^2(\Qn)^d}^2 \\
& \qquad \le \Norm{f}{L^2(\Qn)} \Norm{\sqrt{\varphi_n} \dpt \uht}{L^2(\Qn)} + \frac12 \Norm{\dptau \uht(\cdot, \tnmo^-)}{L^2(\Omega)}^2 + \frac{c^2}{2} \Norm{\nabla \uht(\cdot, \tnmo^-)}{L^2(\Omega)^d}^2,
\end{split}
\end{equation*}
which can be used to show the existence of a unique solution for the homogeneous problem~$(f = 0,\ u_0 = 0, \ v_0 = 0)$ using a recursive argument. 

The parameter~$\delta$ is taken fixed independent of~$\tau$ and~$h$. The following continuous dependence on the data can be obtained using a Gr\"onwall-like argument:
\begin{alignat*}{3}
& \Norm{\dpt \uht(\cdot, \tn^-)}{L^2(\Omega)}^2 + {e^{\delta \tau_n}} \Norm{\jump{\dpt \uht}_{n-1}}{L^2(\Omega)}^2 + \frac{\delta e^{\delta \tau_n}}{2} \Norm{\sqrt{\varphi_n} \dpt \uht}{L^2(\Qn)}^2 \\
& + \Norm{\nabla \uht(\cdot, \tn^-)}{L^2(\Omega)^d}^2 + c^2 e^{\delta \tau_n} \Norm{\jump{\nabla \uht}_{n-1}}{L^2(\Omega)^d}^2 + \frac{\delta e^{\delta \tau_n} c^2}{2} \Norm{\sqrt{\varphi_n} \nabla \uht}{L^2(\Qn)^d}^2 \\
& \qquad \le e^{\delta \tn}  \Big(\Norm{v_0}{L^2(\Omega)}^2 + c^2 \Norm{\nabla u_0}{L^2(\Omega)^d}^2 + \delta^{-1} \Norm{f}{L^2(0, \tn; L^2(\Omega))}^2 \Big),
\end{alignat*}
which, for~$\delta = 1/T$, implies
\begin{alignat*}{3}
\Norm{\dpt \uht(\cdot, \tn^-)}{L^2(\Omega)}^2 + \Norm{\nabla \uht(\cdot, \tn^-)}{L^2(\Omega)^d}^2 \lesssim \Norm{v_0}{L^2(\Omega)}^2 + c^2 \Norm{\nabla u_0}{L^2(\Omega)^d}^2 + T \Norm{f}{L^2(0, \tn; L^2(\Omega))}^2.
\end{alignat*}
However, this does not seem to lead to stability in stronger norms. 

\begin{remark}[Suboptimal convergence]
The a priori error estimates in~\cite[Thm.~2]{French:1993} assume that~$\tau \simeq h^{\gamma}$ and have a term~$\mathcal{O}(\tau^{-1} h^{p})$, which leads to a degradation of the rates for~$\gamma > 0$. 
\eremk
\end{remark}

\begin{remark}[$B$-splines version]
In the B-splines version in~\cite{Ferrari_Perugia:2025}, the following phenomenon was observed: suboptimal convergence rates are obtained for odd degrees when using~$C^1$-regular splines, whereas optimal convergence rates are numerically obtained for any degree when using splines of maximal regularity. 
\eremk
\end{remark}

\subsection{The damped wave equation (Antonietti--Mazzieri--Migliorini, 2020)} 
\label{sec:damped-wave}
The presence of a positive damping term yields coercivity of the corresponding ``plain vanilla" DG formulation without a CFL condition. However, the stronger stability and \emph{a priori} error estimates in Proposition~\ref{prop:new-continuous-dependence-DG-wave} and Theorem~\ref{thm:new-a-priori-plain-DG-wave} in~$L^{\infty}(0, T; \mathcal{X})$ norms still require the CFL condition~\eqref{eq:CFL-DG-wave}. 
More precisely, we consider the following damped wave equation:
\begin{subequations}
\label{eq:damped-wave-equation}
\begin{alignat}{3}
\dptt u  + \delta \dpt u - c^2 \Delta u  & = f & & \quad \text{ in } \QT,\\
u & = 0 & &\quad \text{ on }~\partial \Omega \times (0, T),\\
u = u_0 \quad \text{and} \quad \dpt u & = v_0 & & \quad \text{ on } \SO,
\end{alignat}
\end{subequations}
for some~$\delta > 0$. 

In this setting, the corresponding space--time method reads: find~$\uht \in \oVdisc{q}$ (with~$q \geq 2)$ such that
\begin{equation}
\label{eq:damped-plain-DG-wave}
\Bht(\uht, \wht) = \ell(\wht) \quad \forall \wht \in \oVdisc{q},
\end{equation}
where~$\ell(\cdot)$ is as in~\eqref{eq:plain-DG-wave}, whereas the 
bilinear form~$\Bht(\cdot, \cdot)$ is given by
\begin{alignat*}{3}
\Bht(\uht, \wht)& :=  \sum_{n = 1}^{N} (\dptt \uht, \dpt \wht)_{\Qn} + \sum_{n = 1}^{N - 1} (\jump{\dpt \uht}_{n}, \dpt \wht(\cdot, \tn^+))_{\Omega}
 + (\dpt \uht, \dpt\wht)_{\SO} \\
 \nonumber
 & \quad + \delta \sum_{n = 1}^N (\dpt \uht, \dpt \wht)_{\Qn} + c^2 (\nabla \uht, \nabla \dpt \wht)_{\QT}
 \\
 \nonumber
& \quad  + c^2 \sum_{n = 1}^{N-1} (\jump{\nabla \uht}_n, \nabla \wht(\cdot, \tn^+))_{\Omega}  + c^2 (\nabla \uht, \nabla \wht)_{\SO},
\end{alignat*}
which is coercive with respect to the norm
\begin{equation*}
\frac12 \Seminorm{\dpt \uht}{\sf J}^2 + \frac{c^2}{2} \Seminorm{\nabla \uht}{\sf J}^2 + \delta \Norm{\dptau \uht}{L^2(\QT)}^2,
\end{equation*}
but coercivity is numerically lost when~$\delta \ll 1$. The analysis in this setting was carried out in~\cite{Antonietti_Mazzieri_Migliorini:2020} (see also~\cite{Antonietti_Mazzieri_DalSanto_Quarteroni:2018}).

\begin{remark}[Suboptimal convergence]
The estimate in~\cite[Thm.~7]{Antonietti_Mazzieri_Migliorini:2020} for the discrete error~$\Piht \eu$ has a suboptimal term~$\mathcal{O}(\delta \tau^q)$; cf. Theorem~\ref{thm:a-priori-plain-DG-wave} above. This issue is related to the suboptimal approximation of~$\dpt u$ in second-order formulations (cf. \cite[Rem.~4.2]{Gomez-Nikolic:2025}).
\eremk
\end{remark}

One can extend the results in Sections~\ref{sec:stab-plain-DG-wave} and~\ref{sec:convergence-plain-DG-wave} to show stability bounds and \emph{a priori} error estimates in~$L^{\infty}(0, T; \mathcal{X})$ norms. However, despite the presence of the damping term, we still require the CFL condition~\eqref{eq:CFL-DG-wave} that results from the treatment of~$J_2$ in~\eqref{eq:J2-plain-DG-wave}.

\begin{proposition}[Continuous dependence on the data]
\label{prop:new-continuous-dependence-DG-wave}
There exists a positive constant~$\CCFL$ depending only on~$p$ and~$q$ such that, if~$\tau \le \frac{\CCFL h_{\min}}{c}$, then any solution~$\uht \in \oVdisc{q}$ to the discrete space--time formulation~\eqref{eq:damped-plain-DG-wave} satisfies
\begin{alignat*}{3}
\Norm{\dptau \uht}{L^{\infty}(0, T; L^2(\Omega))}^2 + c^2 \Norm{\nabla \uht}{L^{\infty}(0, T; L^2(\Omega)^d)}^2 + \Seminorm{\dptau \uht}{\sf J}^2  + c^2 \Seminorm{\nabla \uht}{\sf J}^2 + \delta \Norm{\dptau \uht}{L^2(\QT)}^2  \\
\lesssim \Norm{f}{L^1(0, T; L^2(\Omega))}^2 + \Norm{v_0}{L^2(\Omega)}^2 + c^2 \Norm{\nabla u_0}{L^2(\Omega)^d}^2,
\end{alignat*}
where the hidden constant depends only on~$q$.
\end{proposition}
\begin{proof}
The proof follows exactly as in Proposition~\ref{prop:continuous-dependence-DG-wave}, except for the bound of the additional term
\begin{equation*}
\begin{split}
\delta (\dpt \uht, \Ptdqmo (\varphi_n \dpt \uht))_{\Qn} & = \delta (\dpt \uht, \varphi_n \dpt \uht)_{\Qn} - \delta (\dpt \uht, (\Id - \Ptdqmo) (\varphi_n \dpt \uht))_{\Qn} \\
& =: L_1 + L_2. 
\end{split}
\end{equation*}
The uniform bound~$\varphi_n \geq 1/2$ immediately gives
\begin{equation*}
    L_1 \geq \frac{\delta}{2} \Norm{\dpt \uht}{L^2(\Qn)}^2. 
\end{equation*}
As for~$L_2$, it is enough to apply Lemma~\ref{lemma:bound-Ptdqmo} to get~$L_2 \geq 0$.
\end{proof}

\begin{theorem}[Estimates for the discrete error]
\label{thm:new-a-priori-plain-DG-wave}
Let the assumptions of Proposition~\ref{prop:new-continuous-dependence-DG-wave} hold. Let~$u$ be the continuous weak solution to~\eqref{eq:damped-wave-equation}, and let~$\uht \in \oVdisc{q}$ be the solution to~\eqref{eq:damped-plain-DG-wave}. Assume that~$\Omega$ satisfies the elliptic regularity assumption~\eqref{eq:elliptic-regularity-Omega}, and that~$u$ satisfies
\begin{equation*}
u \in H^1(0, T; H_0^1(\Omega)) \cap W^{2, 1}(0, T; H^{r+1}(\Omega)) \quad \text{ with } \quad \Delta u \in W^{s+1, 1}(0, T; L^2(\Omega)), 
\end{equation*}
for some~$0 \le r \le p$ and~$1 \le s \le q$. 
Then, the following estimate holds:
\begin{alignat*}{3}
&\Norm{\dptau \Piht \eu}{L^{\infty}(0, T; L^2(\Omega))}  + c\Norm{\nabla \Piht \eu}{L^{\infty}(0, T; L^2(\Omega)^d)} \\
& \quad + \Seminorm{\dptau \Piht \eu}{\sf J} + c \Seminorm{\nabla \Piht \eu}{\sf J} + \sqrt{\delta} \Norm{\dptau \Piht \eu}{L^2(\QT)} \\
& \qquad \lesssim h^{r+1} \big(\Seminorm{\dptt u}{L^1(0, T; H^{r+1}(\Omega))} + \Seminorm{v_0}{H^{r+1}(\Omega)} + \delta T \Seminorm{\dpt u}{L^{\infty}(0, T; H^{r+1}(\Omega))} \big) \\
& \quad \qquad + \tau^{s} \big( c^2 \tau \Norm{\Delta \dpt^{(s+1)} u}{L^1(0, T; L^2(\Omega))} + \delta \Norm{\dpt^{(s+1)} u}{L^1(0, T; L^2(\Omega))}\big). 
\end{alignat*}
\end{theorem}
\begin{proof}
The result follows as in Theorem~\ref{thm:a-priori-plain-DG-wave}, and the additional terms involving~$\delta$ arise from the estimate
\begin{alignat*}{3}
\delta \Norm{\dptau (\Id - \Piht) u}{L^1(0, T; L^2(\Omega))} & \le \delta \Norm{(\Id - \PtThqmo) \dpt u}{L^1(0, T; L^2(\Omega))} \\
& \quad + \delta T \Norm{\Ptdqmo (\Id - \Rh) \dpt u}{L^{\infty}(0, T; L^2(\Omega))},
\end{alignat*}
and applying the approximation properties of~$\PtThqmo$ and~$\Rh$. 
\end{proof}

\subsection*{An ultraweak formulation (Adjerid--Temimi, 2011)}
Another DG time discretization for the second-order formulation~\eqref{eq:second-order-wave} was studied in~\cite{Adjerid_Temimi:2011}. The method reads: find~$\uht \in \oVdisc{q}$ such that
\begin{alignat*}{3}
\Bht(\uht, \wht) = \ell(\wht) \qquad \forall \wht \in \oVdisc{q},
\end{alignat*}
where
\begin{alignat*}{3}
\Bht(\uht, \wht) &:= \sum_{n = 1}^{N - 1} (\uht, \dptt \wht)_{\Qn}  - \sum_{n = 1}^{N - 1} (\dptau \uht (\cdot, \tn^-), \jump{\wht}_{n})_{\Omega} \\
& \quad + \sum_{n = 1}^{N - 1} (\uht(\cdot, \tn^-), \jump{\dptau \wht}_n)_{\Omega} + (\dpt \uht, \wht)_{\ST} - (\uht, \dpt\wht)_{\ST}, \\
& \quad + c^2 (\nabla \uht, \nabla \wht)_{\QT}\\
\ell(\wht) &= (f, \wht)_{\QT} + (v_0, \wht(\cdot, 0))_{\Omega} - (u_0, \dpt \wht(\cdot, 0))_{\Omega} .
\end{alignat*}

Integrating by parts twice in time and using identity~\eqref{eq:jumps-in-time-identity-1}, we can rewrite the bilinear form~$\Bht(\cdot, \cdot)$ as
\begin{alignat*}{3}
\Bht(\uht, \wht) & = \sum_{n = 1}^{N} (\dptt \uht, \wht)_{\Qn} + \sum_{n = 1}^{N - 1} \Big[ (\jump{\dptau \uht}_n,  \wht(\cdot, \tn^+))_{\Omega} 
-  (\jump{\uht}_n, \dpt \wht(\cdot, \tn^+))_{\Omega}\Big] \\
& \quad + (\dpt \uht, \wht)_{\SO} - (\uht, \dpt \wht)_{\SO} + c^2(\nabla \uht, \nabla \wht)_{\QT}.
\end{alignat*}

We limit ourselves to highlighting that:
\begin{itemize}[noitemsep]
\item the analysis in~\cite{Adjerid_Temimi:2011} does not address the existence of discrete solutions nor the stability of the scheme;
\item a matrix-based argument is used to get error estimates;
\item no obvious coercivity is obtained for this formulation. 
\end{itemize}

The next three sections address the stability and \emph{a priori} error analysis for some unconditionally stable space--time formulations.

\section{CG time discretization for the Hamiltonian formulation of the wave equation (French--Peterson, 1991)}
\label{sec:French-Peterson-wave}
The following space--time CG method for the Hamiltonian formulation~\eqref{eq:first-order-wave} of the wave equation was proposed and analyzed for the first time in~\cite{French_Peterson:1991}. The method reads as follows: find~$(\uht, \vht) \in \oVcont{q} \times \oVcont{q}$ (with~$q \geq 1$) such that~$\uht(\cdot, 0) = \Rh u_0$, $\vht(\cdot, 0) = \Pih v_0$, and 
\begin{subequations}
\label{eq:French-Peterson}
\begin{alignat}{3}
\label{eq:French-Peterson-1}
c^2 (\nabla \vht, \nabla \zht)_{\QT} & = c^2 (\nabla \dpt \uht, \nabla \zht)_{\QT} & & \quad \forall \zht \in \oVdisc{q-1}, \\
(\dpt \vht, \wht)_{\QT} + c^2 (\nabla \uht, \nabla \wht)_{\QT} & = (f, \wht)_{\QT} & & \quad \forall \wht \in \oVdisc{q-1}. 
\end{alignat}
\end{subequations}

The space--time formulation~\eqref{eq:French-Peterson} belongs to the class of Petrov--Galerkin methods (i.e., with different test and trial spaces). Using the fact that~$\dpt \oVcont{q}= \oVdisc{q-1}$, it is possible to rewrite~\eqref{eq:French-Peterson} as follows: find~$(\uht, \vht) \in \oVcont{q} \times \oVcont{q}$ (with~$q \geq 1$) such that
\begin{alignat*}{3}
(c^2 \nabla \vht, \nabla \dpt \zht)_{\QT} - (c^2 \nabla \dpt \uht, \nabla \dpt  \zht)_{\QT} & = 0  & & \qquad \forall \zht \in \oVcont{q},\\
(\dpt \vht, \dpt \wht)_{\QT} + (c^2 \nabla \uht, \nabla \dpt \wht)_{\QT} & = (f, \dpt \wht)_{\QT}
& & \qquad \forall \wht \in \oVcont{q},
\end{alignat*}
thus avoiding the use of different test and trial spaces. 
However, this equivalent formulation is misleading for the analysis, as the nonstandard test functions from the space~$\oVcont{q-1}$ that we use in the stability analysis cannot be easily written as the time derivative of some functions from the space~$\oVcont{q}$.

For homogeneous Dirichlet boundary conditions, equation~\eqref{eq:French-Peterson-1} reduces to
\begin{equation}
\label{eq:French-Peterson-identity-vht-dpt-uht}
    \Pi_{q-1}^t \vht = \dpt \uht.
\end{equation}
This is no longer true for nonhomogeneous Dirichlet boundary conditions (see~\cite{Gomez:2025} for more details).

Different analyses of the method in~\cite{French_Peterson:1991} can be found in the literature. In Table~\ref{tab:comparison-CG}, we compare such results in terms of \emph{(i)} the absence of a CFL condition, \emph{(ii)} the treatment of nonhomogeneous Dirichlet boundary conditions, \emph{(iii)} changing spatial meshes, \emph{(iv)} variable time steps, \emph{(v)} the dependence of the stability (and error) constant on the final time~$T$, \emph{(vi)} and the regularity required in time on~$f$ (which influences the regularity required on~$u$ in the error estimate). The rest of this section is based on~\cite{Gomez:2025}, but considers the simpler situation of homogeneous Dirichlet boundary conditions.

\begin{table}[ht]
\centering
\caption{Comparison of the theoretical analyses in the literature for the space--time continuous finite element method in~\cite{French_Peterson:1991}.}
\label{tab:comparison-CG}
\begin{tabular}{ccccccccc}
\hline 
\multirow{2}{*}{Ref.} & Uncond. & Variable time & Changing & Nonhomogeneous & Time  & Reg. $f$\\
& stab. & steps & meshes & Dirichlet & dependence & in time \\
\hline\\[-0.4em]
\cite{French_Peterson:1991} & \xmark & \xmark & \xmark & -- & -- & -- \\[0.5em]
\cite{Bales_Lasiecka:1994} & \cmark & \xmark & \xmark & \xmark & $e^T$ & $L^2$ \\[0.5em]
\cite{Bales_Lasiecka:1995} & \cmark & \xmark & \xmark & \cmark & $e^T$ & $L^2$ \\[0.5em]
\cite{French_Peterson:1996} & \cmark & \cmark & \xmark &  \xmark & $T$ & $L^{\infty}$ \\[0.5em]
\cite{Karakashian_Makridakis:2005} & \cmark & \cmark & \cmark & \xmark & $e^T$ & $L^2$ \\[0.5em]
\cite{Zhao_Li:2016} & \cmark & \cmark & \xmark & \xmark & $e^T$ & $L^2$ \\[0.5em]
\cite{Gomez:2025} & \cmark & \cmark & \xmark & \cmark & $\mathcal{O}(1)$ & $L^1$ \\[0.5em]
\hline
\hline 
\end{tabular}
\end{table}

For the convergence analysis, it is useful to analyze the stability of the following perturbed space--time variational problem: find~$(\uht, \vht) \in \oVcont{q} \times \oVcont{q}$ (with~$q \geq 1)$ such that~$\uht(\cdot, 0) = \Rh u_0$, $\vht(\cdot, 0) = \Pih v_0$, and
\begin{subequations}
\label{eq:French-Peterson-perturbed}
\begin{alignat}{3}
\label{eq:French-Peterson-perturbed-1}
c^2 (\nabla \vht, \nabla \zht)_{\QT} - c^2 (\nabla \dpt \uht, \nabla \zht)_{\QT} & = c^2(\nabla \Upsilon , \nabla \zht)_{\QT} & & \quad \forall \zht \in \oVdisc{q-1}, \\
\label{eq:French-Peterson-perturbed-2}
(\dpt \vht, \wht)_{\QT} + c^2 (\nabla \uht, \nabla \wht)_{\QT} & = (f, \wht)_{\QT} & & \quad \forall \wht \in \oVdisc{q-1},
\end{alignat}
\end{subequations}
where the perturbation function~$\Upsilon \in L^2(0, T; H_0^1(\Omega))$ will represent a projection error in the convergence analysis.

The forthcoming stability and convergence analysis follows the structure described in Section~\ref{sec:structure-theory}.
\subsection{Stability analysis}

\begin{lemma}[Weak partial bound on the discrete solution]
\label{lemma:weak-continuous-CG-wave}
Any solution~$(\uht, \vht) \in \oVcont{q} \times \oVcont{q}$ to the discrete space--time formulation~\eqref{eq:French-Peterson-perturbed} satisfies the following bound for~$n = 1, \ldots, N$:
\begin{equation}
\label{eq:weak-partial-bound}
\begin{split}
\frac12 \big( \Norm{\vht}{L^2(\Sn)}^2 + c^2 \Norm{\nabla \uht}{L^2(\Sn)^d}^2\big) & \le \frac12 \big(\Norm{v_0}{L^2(\Omega)}^2 + c^2 \Norm{ \nabla u_0}{L^2(\Omega)^d}^2 \big) \\
& \quad + \CS \Norm{f}{L^1(0, \tn; L^2(\Omega))} \Norm{\vht}{L^{\infty}(0, \tn; L^2(\Omega))} \\
& \quad + c^2 \CS \Norm{\nabla \Upsilon}{L^1(0, \tn; L^2(\Omega)^d)} \Norm{\nabla \uht}{L^{\infty}(0, \tn; L^2(\Omega)^d)},
\end{split}
\end{equation}
where~$\CS$ is the constant in Lemma~\ref{lemma:stab-pi-time}, which depends only on~$q$.
\end{lemma}
\begin{proof}
Without loss of generality, we only prove the result for~$n = N$.

We choose~$\wht = \Pi_{q-1}^t \vht \in \oVdisc{q-1}$ in~\eqref{eq:French-Peterson-perturbed-2} and obtain the following identity:
\begin{alignat}{3}
\label{eq:aux-identity-weak-bound}
(\dpt \vht, \Pi_{q-1}^t \vht)_{\QT} + (c^2 \nabla \uht, \nabla \Pi_{q-1}^t \vht)_{\QT} & = (f, \Pi_{q-1}^t \vht)_{\QT}. %
\end{alignat}

Using the orthogonality properties of~$\Pi_{q-1}^t$, the continuity in time of~$\vht$, and the choice of the discrete initial conditions, we have
\begin{equation}
\label{eq:identity-dpt}
(\dpt \vht, \Pi_{q-1}^t \vht)_{\QT} = (\dpt \vht, \vht)_{\QT} = \frac12 \big(\Norm{\vht}{L^2(\ST)}^2 - \Norm{\Pih v_0}{L^2(\Omega)}^2 \big).
\end{equation}

As for the second term on the left-hand side of~\eqref{eq:aux-identity-weak-bound}, we use~\eqref{eq:French-Peterson-perturbed-2}, the continuity in time of~$\uht$, and the choice of the discrete initial conditions to obtain
\begin{alignat}{3}
\nonumber
c^2 (\nabla \uht, \nabla \Pi_{q-1}^t \vht)_{\QT} 
& = c^2 (\nabla \Pi_{q-1}^t \uht,  \nabla \vht)_{\QT} \\
\nonumber
& = c^2 (\nabla \Pi_{q-1}^t \uht,  \nabla \dpt \uht)_{\QT} + c^2(\nabla \Pi_{q-1}^t \uht,  \nabla \Upsilon)_{\QT} \\
\nonumber
& = c^2(\nabla \uht, \nabla \dpt \uht)_{\QT} + c^2(\nabla \Pi_{q-1}^t  \uht, \nabla \Upsilon)_{\QT} \\
\label{eq:identity-grad-grad}
& = \frac{c^2}{2} \big(\Norm{ \nabla \uht}{L^2(\ST)^d}^2 - \Norm{ \nabla \Rh u_0}{L^2(\Omega)^d}^2 \big) + c^2 (\nabla \Upsilon,  \nabla \Pi_{q-1}^t \uht)_{\QT} .
\end{alignat}

Inserting identities~\eqref{eq:identity-dpt} and~\eqref{eq:identity-grad-grad} in~\eqref{eq:aux-identity-weak-bound}, and using the stability properties of~$\Pih$ and~$\Rh$, the H\"older inequality, and Lemma~\ref{lemma:stab-pi-time}, we get
\begin{alignat*}{3}
\frac12 \big( & \Norm{\vht}{L^2(\ST)}^2  + c^2\Norm{ \nabla \uht}{L^2(\ST)^d}^2\big)\\
& = \frac12 \big(\Norm{\Pih v_0}{L^2(\Omega)}^2 + c^2 \Norm{\nabla \Rh u_0}{L^2(\Omega)^d}^2 \big)  + (f, \Pi_{q-1}^t \vht)_{\QT} - c^2 (\nabla \Upsilon, \Pi_{q-1}^t \nabla \uht)_{\QT} \\
& \le \frac12 \big(\Norm{v_0}{L^2(\Omega)}^2 + c^2 \Norm{\nabla u_0}{L^2(\Omega)^d}^2 \big) + \Norm{f}{L^1(0, T; L^2(\Omega))} \Norm{\Pi_{q-1}^t \vht}{L^{\infty}(0, T; L^2(\Omega))} \\
& \quad + c^2 \Norm{ \nabla \Upsilon }{L^1(0, T; L^2(\Omega)^d)} \Norm{ \Pi_{q-1}^t \nabla \uht}{L^{\infty}(0, T; L^2(\Omega)^d)} \\
& \le \frac12 \big(\Norm{v_0}{L^2(\Omega)}^2 + c^2\Norm{ \nabla u_0}{L^2(\Omega)^d}^2 \big) + \CS \Norm{f}{L^1(0, T: L^2(\Omega))} \Norm{\vht}{L^{\infty}(0, T; L^2(\Omega))} \\
& \quad + c^2 \CS  \Norm{ \nabla \Upsilon}{L^1(0, T; L^2(\Omega)^d)} \Norm{\nabla \uht}{L^{\infty}(0, T; L^2(\Omega)^d)},
\end{alignat*}
where~$\CS$ is the constant in Lemma~\ref{lemma:stab-pi-time}. This completes the proof of~\eqref{eq:weak-partial-bound}.
\end{proof}

\begin{proposition}[Continuous dependence on the data]
\label{prop:continuous-dependence-CG-wave}
Any solution~$(\uht, \vht) \in \oVcont{q} \times \oVcont{q}$ to the discrete space--time formulation~\eqref{eq:French-Peterson-perturbed} satisfies
\begin{equation*}
\begin{split}
& \frac12 \Big(\Norm{\vht}{C^0([0, T]; L^2(\Omega))}^2 + c^2 \Norm{ \nabla \uht}{C^0([0, T]; L^2(\Omega)^d)}^2 \Big) \\
& \qquad \qquad \lesssim  \frac12 \big(\Norm{v_0}{L^2(\Omega)}^2 + c^2 \Norm{ \nabla u_0}{L^2(\Omega)^d}^2 \big) + \Norm{f}{L^1(0, T; L^2(\Omega))}^2 + c^2 \Norm{\nabla \Upsilon}{L^1(0, T; L^2(\Omega)^d)}^2,
\end{split}
\end{equation*}
where the hidden constant depends only on~$q$.
\end{proposition}
\begin{proof}

Let~$n \in \{1, \ldots, N\}$. We choose the following test function in~\eqref{eq:French-Peterson-perturbed-2}:
\begin{equation*}
\wht {}_{|_{Q_m}} := \begin{cases}
\Pi_{q-1}^t (\varphi_n \Pi_{q-1}^t \vht) & \text{ if } m = n, \\
0 & \text{ otherwise}.
\end{cases} 
\end{equation*}
We focus on the case~$n > 1$, as the argument for~$n = 1$ follows with minor modifications.

The following identity is obtained:
\begin{alignat}{3}
\label{eq:aux-identity-linear-wave}
(\dpt \vht, \Pi_{q-1}^t (\varphi_n \Pi_{q-1}^t \vht) )_{\Qn} + c^2 (\nabla \uht, \nabla \Pi_{q-1}^t (\varphi_n \Pi_{q-1}^t \vht))_{\Qn} = (f, \Pi_{q-1}^t (\varphi_n \Pi_{q-1}^t \vht))_{\Qn}.
\end{alignat}

Using the orthogonality properties of~$\Pi_{q-1}^t$, the first term on the left-hand side of~\eqref{eq:aux-identity-linear-wave} can be split as follows:
\begin{alignat*}{3}
(\dpt \vht, \Pi_{q-1}^t (\varphi_n \Pi_{q-1}^t \vht))_{\Qn} 
&  = (\dpt \vht, \varphi_n \Pi_{q-1}^t \vht)_{\Qn}  \\
& = (\varphi_n \dpt \vht,  \vht)_{\Qn} - (\varphi_n \dpt \vht, (\Id - \Pi_{q-1}^t ) \vht)_{\Qn} \\
& =: J_1 + J_2.
\end{alignat*}
Identity~\eqref{eq:identities-varphi-integration-2} from Lemma~\ref{lemma:identities-varphi-integration} gives
\begin{alignat*}{3}
J_1 = (\varphi_n \dpt \vht, \vht)_{\Qn} 
\label{eq:J1}
& = \frac{1}{4} \Norm{\vht}{L^2(\Sn)}^2 - \frac12 \Norm{\vht}{L^2(\Sigma_{n - 1})}^2 + \frac{\lambda_n}{2} \Norm{\vht}{L^2(\Qn)}^2.
\end{alignat*}
Moreover, Lemma~\ref{lemma:bound-pi-time-varphi_n} implies that~$J_2 \geq 0$.

Similarly, using~\eqref{eq:French-Peterson-perturbed-1} in the definition of the perturbed problem and the orthogonality properties of~$\Pi_{q-1}^t$, the second term on the left-hand side of~\eqref{eq:aux-identity-linear-wave} can be split as follows:
\begin{alignat*}{3}
c^2 ( \nabla \uht, \nabla \Pi_{q-1}^t(\varphi_n \Pi_{q-1}^t \vht))_{\Qn} 
& = c^2  (\Pi_{q-1}^t(\varphi_n \nabla \Pi_{q-1}^t \uht), \nabla \vht)_{\Qn} \\
& = c^2 (\varphi_n \nabla \Pi_{q-1}^t \uht,  \nabla \dpt \uht)_{\Qn} + (\Pi_{q-1}^t(\varphi_n \nabla \Pi_{q-1}^t \uht),  \nabla \Upsilon)_{\Qn} \\
& = c^2 ( \nabla \uht, \nabla \Pi_{q-1}^t (\varphi_n \dpt \uht) )_{\Qn} + (\Pi_{q-1}^t(\varphi_n \nabla \Pi_{q-1}^t \uht),  \nabla \Upsilon)_{\Qn}  \\
& = c^2 ( \varphi_n \nabla \uht, \nabla \dpt \uht)_{\Qn} - c^2  (\nabla \uht, (\Id - \Pi_{q-1}^t) (\varphi_n \nabla \dpt \uht))_{\Qn}  \\
& \quad + c^2 (\Pi_{q-1}^t(\varphi_n \nabla \Pi_{q-1}^t \uht),  \nabla \Upsilon)_{\Qn} \\
& =: L_1 + L_2 + L_3.
\end{alignat*}

Using again identity~\eqref{eq:identities-varphi-integration-2} and Lemma~\ref{lemma:bound-pi-time-varphi_n}, we get~$L_2 \geq 0$ and
\begin{alignat*}{3}
L_1 = c^2 ( \varphi_n \nabla \uht, \nabla \dpt \uht)_{\Qn} = \frac{c^2}{4} \Norm{ \nabla \uht}{L^2(\Sn)^d}^2 - \frac{c^2}{2} \Norm{\nabla \uht}{L^2(\Sigma_{n - 1})^d}^2 + \frac{\lambda_n c^2}{2} \Norm{ \nabla \uht}{L^2(\Qn)^d}^2.
\end{alignat*}

Using the H\"older inequality, the stability bound in Lemma~\ref{lemma:stab-pi-time} for~$\Pi_{q-1}^t$, and the uniform bound for~$\varphi_n$, we get
\begin{alignat*}{3}
L_3 = c^2 (\Pi_{q-1}^t(\varphi_n \nabla \Pi_{q-1}^t \uht),  \nabla \Upsilon)_{\Qn}
& \le c^2 \CS^2 \Norm{\nabla \Upsilon}{L^1(\In; L^2(\Omega)^d)} \Norm{\nabla \uht}{L^{\infty}(\In; L^2(\Omega)^d)}, \\
(f, \Pi_{q-1}^t(\varphi_n \Pi_{q-1}^t \vht))_{\Qn} & \le \CS^2 \Norm{f}{L^1(\In; L^2(\Omega))} \Norm{\vht}{L^{\infty}(\In; L^2(\Omega))}.
\end{alignat*}

Using the inverse estimate in Lemma~\ref{lemma:L2-Linfty}, we have
\begin{alignat*}{3}
\frac{1}{4\Cinv} \Norm{\vht}{L^{\infty}(\In; L^2(\Omega))}^2 & \le \frac{\lambda_n}{2}\Norm{\vht}{L^2(\Qn)}^2, \\
\frac{c^2}{4\Cinv} \Norm{ \nabla \uht}{L^{\infty}(\In; L^2(\Omega)^d)}^2 & \le \frac{\lambda_n c^2}{2} \Norm{\nabla \uht}{L^2(\Qn)^d}^2.
\end{alignat*}
Therefore, combining the above estimates, and using Lemma~\ref{lemma:weak-continuous-CG-wave} to bound the energy terms at~$\tnmo$, we obtain
\begin{alignat*}{3}
\nonumber
\frac{1}{4} & \big(\Norm{\vht}{L^2(\Sn)}^2 + c^2 \Norm{ \nabla \uht}{L^2(\Sn)^d}^2 \big) + \frac{1}{4 \Cinv^2} \big(\Norm{\vht}{L^{\infty}(\In; L^2(\Omega))}^2 + c^2 \Norm{ \nabla \uht}{L^{\infty}(\In; L^2(\Omega)^d)}^2 \big) \\
\nonumber
& \le \frac12 \big(\Norm{\vht}{L^2(\Sigma_{n - 1})}^2 + c^2 \Norm{\nabla \uht}{L^2(\Sigma_{n - 1})^d}^2 \big) +  \CS^2 \Norm{f}{L^1(\In; L^2(\Omega))} \Norm{\vht}{L^{\infty}(\In; L^2(\Omega))} \\
\nonumber 
& \quad + c^2 \CS^2 \Norm{ \nabla \Upsilon}{L^1(\In; L^2(\Omega)^d)} \Norm{ \nabla \uht}{L^{\infty}(\In; L^2(\Omega)^d)} \\
\nonumber
& \le \frac12\big(\Norm{v_0}{L^2(\Omega)}^2 + c^2\Norm{\nabla u_0}{L^2(\Omega)^d}^2 \big) \\
\nonumber
& \quad + \CS \Norm{f}{L^1(0, \tnmo; L^2(\Omega))} \Norm{\vht}{L^{\infty}(0, \tnmo; L^2(\Omega))} 
+ \CS^2 \Norm{f}{L^1(\In; L^2(\Omega))} \Norm{\vht}{L^{\infty}(\In; L^2(\Omega))}\\
\nonumber
& \quad + c^2 \CS \Norm{\nabla \Upsilon}{L^1(0, \tnmo; L^2(\Omega)^d))} \Norm{\nabla \uht }{L^{\infty}(0, \tnmo; L^2(\Omega)^d)} \\
\nonumber 
& \quad + c^2 \CS^2 \Norm{\nabla \Upsilon}{L^1(\In; L^2(\Omega)^d)} \Norm{\nabla \uht }{L^{\infty}(\In; L^2(\Omega)^d)}\\ 
\nonumber
& \lesssim \frac12 \big(\Norm{v_0}{L^2(\Omega)}^2 + c^2 \Norm{\nabla u_0}{L^2(\Omega)^d}^2  \big) + \Norm{f}{L^1(0, T; L^2(\Omega))} \Norm{\vht}{L^{\infty}(0, T; L^2(\Omega))} \\
& \quad + c^2 \Norm{\nabla \Upsilon}{L^1(0, T; L^2(\Omega)^d)} \Norm{\nabla \uht}{L^{\infty}(0, T; L^2(\Omega)^d)},
\end{alignat*}
where the hidden constant depends only on~$q$.

For each of the terms on the left-hand side, one can choose an index~$n$ where it achieves its
maximum value. This, combined with Lemma~\ref{lemma:weak-continuous-CG-wave}, gives the desired result.
\end{proof}
The existence of a unique solution to~\eqref{eq:French-Peterson} is an immediate consequence of Proposition~\ref{prop:continuous-dependence-CG-wave}, as it corresponds to a linear system with a nonsingular matrix.

\subsection{Convergence analysis}
Recalling Definition~\ref{def:Pt-AM} of the Aziz--Monk projection operator~$\PtAM$, we define the space--time projection~$\Piht = \PtAM \circ \Rh$ and the error functions
\begin{equation*}
\eu = u - \uht = u - \Piht u + \Piht \eu \qquad \text{ and } \qquad \ev = v - \vht = v - \Piht v + \Piht \ev. 
\end{equation*}

\begin{theorem}[Estimates for the discrete error]
\label{thm:a-priori-CG-wave}
Let~$(u, v)$ be the continuous weak solution to~\eqref{eq:first-order-wave}, and~$(\uht, \vht) \in \oVcont{q} \times \oVcont{q}$ be the solution to the discrete space--time formulation~\eqref{eq:French-Peterson}.
Let also~$\Omega$ satisfy the elliptic regularity condition~\eqref{eq:elliptic-regularity-Omega}, and~$u$ satisfies
\begin{equation*}
\begin{split}
u \in H^1(0, T; & H_0^1(\Omega)) \cap W^{2, 1}(0, T; H^{r+1}(\Omega)) \\
& \text{ with } \Delta u \in W^{s+1, 1}(0, T; L^2(\Omega)) \text{ and } \nabla u \in W^{s+2, 1}(0, T; L^2(\Omega)^d),
\end{split}
\end{equation*}
for some~$0 \le r \le p$ and~$1 \le s \le q$. 
Then, the following estimate holds:
\begin{equation*}
\begin{split}
 \Norm{\Piht \ev}{C^0([0, T]; L^2(\Omega))} 
&  + c \Norm{ \nabla \Piht \eu}{C^0([0, T]; L^2(\Omega)^d)} \\
&  \lesssim  h^{r + 1} \big(\Seminorm{\dptt u}{L^1(0, T; H^{r + 1}(\Omega))} + \Seminorm{v_0}{H^{r+1}(\Omega)} \big)\\
& \quad + \tau^{s + 1} \big(c^2 \Norm{\Delta \dpt^{(s + 1)} u}{L^1(0, T; L^2(\Omega))} + c \Norm{\nabla \dpt^{(s + 2)} u}{L^1(0, T; L^2(\Omega)^d)}\big).
\end{split}
\end{equation*}
\end{theorem}
\begin{proof}
Since the space--time formulation~\eqref{eq:French-Peterson} is consistent  for the regularity assumptions made on~$u$, the following error equations hold:
\begin{subequations}
\label{eq:error-equations-original}
\begin{alignat}{3}
\nonumber
(c^2 \nabla \Piht \ev, \nabla \zht)_{\QT} & - (c^2 \nabla \dpt \Piht \eu, \nabla \zht) \\
\label{eq:error-equations-original-1}
& = (c^2 \nabla \Piht v, \nabla \zht)_{\QT} - (c^2 \nabla \dpt \Piht u, \nabla \zht)_{\QT}, \\
\nonumber
(\dpt \Piht \ev, \wht)_{\QT} & + (c^2 \nabla \Piht \eu, \wht)_{\QT} \\
\label{eq:error-equations-original-2}
& = -(\dpt (\Id - \Piht)v , \wht)_{\QT} - (c^2 \nabla (\Id - \Piht)u , \nabla \wht)_{\QT}
\end{alignat}
\end{subequations}
for all~$(\zht, \wht) \in \oVdisc{q-1} \times \oVdisc{q-1}$.

We now simplify the terms on the right-hand side of~\eqref{eq:error-equations-original} using the orthogonality properties of~$\PtAM$ and~$\Rh$ as follows:
\begin{subequations}
\begin{alignat}{3}
\nonumber
c^2( \nabla \Piht v, \nabla \zht)_{\QT} & - c^2( \nabla \dpt \Piht u, \nabla \zht)_{\QT}  \\
\nonumber
& = c^2 (\nabla \PtAM v, \nabla \zht)_{\QT} - c^2( \nabla \dpt  u, \nabla \zht)_{\QT} \\
\label{eq:projection-error-term-1}
& = - c^2( \nabla (\Id - \PtAM) v, \nabla \zht)_{\QT}, \\
\nonumber
-(\dpt (\Id - \Piht) v, \wht)_{\QT} & - c^2( \nabla (\Id - \Piht) u, \nabla \wht)_{\QT} \\
\label{eq:projection-error-term-2}
& = -( (\Id - \Rh) \dpt v, \wht)_{\QT} + c^2 ((\Id - \PtAM)\Delta u, \wht)_{\QT},
\end{alignat}
\end{subequations}
where, in the last step of~\eqref{eq:projection-error-term-2}, we have also integrated by parts in space.

We deduce that the discrete error pair~$(\Piht \eu, \Piht \ev)$ solves a perturbed variational problem of the form~\eqref{eq:French-Peterson-perturbed} with data
\begin{equation*}
\begin{split}
& \widetilde{f} = - (\Id - \Rh) \dpt v + c^2 (\Id - \PtAM) \Delta u, \quad \Upsilon  = - (\Id - \PtAM) v,\quad  
 \\
& \Piht \eu(\cdot, 0)  = 0, \quad 
\Piht \ev(\cdot, 0) = \Rh(v_0 - \Pih v_0) = - \Pih(\Id - \Rh) v_0,
\end{split}
\end{equation*}
as~$\PtAM v(\cdot, 0) = v(\cdot, 0)$.

The estimate in the statement of this theorem then follows by the continuous dependence on the data in Proposition~\ref{prop:continuous-dependence-CG-wave}, and the approximation properties
of~$\PtAM$ and~$\Rh$ in Lemmas~\ref{lemma:estimates-Ptau} (relying on Lemma~\ref{lemma:stab-PtAM}) and~\ref{lemma:Estimates-Rh}, respectively.
\end{proof}

\begin{corollary}[\emph{A priori} error estimates]
Under the assumptions of Theorem~\ref{thm:a-priori-CG-wave}, and assuming that the continuous solution~$u$ to~\eqref{eq:second-order-weak-formulation-wave} is sufficiently regular, there hold
\begin{alignat*}{3}
\Norm{\ev}{C^{0}([0, T]; L^2(\Omega))} & \lesssim h^{p+1} \Big(\Seminorm{\dptt u}{L^1(0, T; H^{p+1}(\Omega))} + \Seminorm{v_0}{H^{p+1}(\Omega)} + \Seminorm{\dpt u}{L^{\infty}(0, T; H^{p+1}(\Omega))} \Big) \\
& \quad + \tau^{q+1} \big( c^2 \Norm{\Delta \dpt^{(q+1)} u}{L^1(0, T; L^2(\Omega))} + c \Norm{\nabla \dpt^{(q+2)} u}{L^1(0, T; L^2(\Omega)^d)} \\
& \qquad \qquad + \Norm{\dpt^{(q + 2)} u}{L^{\infty}(0, T; L^2(\Omega))}\big), \\
c \Norm{\nabla \eu}{C^0([0, T]; L^2(\Omega)^d)} & \lesssim  h^{p} \big(\Seminorm{\dptt u}{L^1(0, T; H^p(\Omega))} + \Seminorm{v_0}{H^p(\Omega)} + c \Seminorm{u}{L^{\infty}(0, T; H^{p+1}(\Omega))} \big) \\
& \quad + \tau^{q+1} \big(c^2 \Norm{\Delta \dpt^{(q+1)} u}{L^1(0, T; L^2(\Omega))} + c \Norm{\nabla \dpt^{(q+2)} u}{L^1(0, T; L^2(\Omega)^d)} \\
& \qquad \qquad + c \Norm{\nabla \dpt^{(q+1)} u}{L^{\infty}(0, T; L^2(\Omega)^d)}\big).
\end{alignat*}
\end{corollary}

\begin{remark}[Connection to a stabilized conforming method]
Recently, in~\cite{Ferrari_Perugia_Zampa:2026}, the relationship between 
the method in this section and the stabilized conforming method proposed in~\cite{Zank:2021} was established. 
    \eremk
\end{remark}

\begin{remark}[Effective number of degrees of freedom]
At first sight, the space--time method~\eqref{eq:French-Peterson} requires solving a coupled linear system with twice as many degrees of freedom as discretizations of the wave equation in a second-order formulation. However, on each partial cylinder~$\Qn$, one can use identity~\eqref{eq:French-Peterson-identity-vht-dpt-uht} to explicitly write
\begin{equation*}
\vht(\bx, t) = \Pi_{q-1}^t \vht (\bx, t) + \alpha_n(\bx) \Leg_q(t) = \dpt \uht (\bx, t) + \alpha_n(\bx) \Leg_q(t),
\end{equation*}
where~$\alpha_n(\bx) = \frac{2q+1}{\tau_n}\int_{\In} \vht(\bx, t) \Leg_q (t) \dt$, showing that the additional degrees of freedom are associated with~$\alpha_n$ only.

An alternative representation follows from the interpolatory properties of~$\Pi_{q-1}^t$ discussed in Remark~\ref{rem:equivalence-L2-orthogonal-Legendre-interpolant}. More precisely, let~$\{\mu_i^{(n)}\}_{i = 1}^{q}$ denote the nodes of the~$q$-point Gauss--Legendre quadrature rule, set~$\mu_0^{(n)} := \tnmo$, and let~$\{\widetilde{\Lag}_i\}_{i = 0}^q$ be the Lagrange polynomials associated with the nodes~$\mu_0^{(n)} < \mu_1^{(n)} < \dots < \mu_{q}^{(n)}$. One can then explicitly write
\begin{equation*}
\vht(\bx, t) = \sum_{i = 0}^{q} \vht(\bx, \mu_i^{(n)}) \widetilde{\Lag}_i(t) = \vht(\bx, \tnmo) \widetilde{\Lag}_0(t) + \sum_{i = 1}^{q} \dpt \uht (\bx, \mu_i^{(n)}) \widetilde{\Lag}_i(t),
\end{equation*}
which makes it clear that the additional degrees of freedom are associated solely with~$\vht(\cdot, \tnmo)$. These values can be computed directly from the approximation of~$v$ on the previous time slab (for~$n > 1$) or from the discrete initial condition (for~$n = 1$).
\eremk
\end{remark}

\section{DG time discretization for the Hamiltonian formulation of the wave equation (Johnson, 1993)}
\label{sec:Johnson-wave}
The following scheme was proposed by Johnson in~\cite{Johnson:1993}, but it was formulated and analyzed only for linear approximations in space and time (i.e., for $p = 1$ and~$q = 1$). Moreover, \emph{a priori} error estimates are obtained only at the discrete times~$\{\tn\}_{n =1}^{N}$ and require a coupling between the meshsize~$h$ and the time step~$\tau$ (see~\cite[Thm.~3.1 and Rem. 3.1]{Johnson:1993}); see also the analysis in~\cite{Rezaei_Sadpanah:2023}. In~\cite{Antonietti_Mazzieri_Migliorini:2023}, the method was applied to an elastodynamic equation, and the analysis was presented for high-order approximations under the assumption of a positive damping term. 

The space--time method reads: find~$(\uht, \vht) \in \oVdisc{q} \times \oVdisc{q}$ (with~$q \geq 1$) such that
\begin{subequations}
\label{eq:original-Johnson}
\begin{alignat}{3}
\nonumber
c^2 (\nabla \vht, \nabla \zht)_{\QT} & = c^2 \sum_{n = 1}^{N} (\dpt \nabla \uht, \nabla \zht)_{\Qn}  + c^2 \sum_{n = 1}^{N - 1} (\jump{\nabla \uht}_n, \nabla \zht(\cdot, \tn^+))_{\Omega}   \\
& \quad + c^2 (\nabla \uht, \nabla \zht)_{\SO} - c^2(\nabla u_0, \nabla \zht(\cdot, 0))_{\Omega}  & & \quad \forall \zht \in \oVdisc{q}, \\
\nonumber
\sum_{n = 1}^{N} (\dpt \vht, \wht)_{\Qn} & + \sum_{n = 1}^{N - 1} (\jump{\vht}_n, \wht(\cdot, \tn^+))_{\Omega}  + (\vht, \wht)_{\SO} \\
& + c^2(\nabla \uht, \nabla \wht)_{\QT} = (f, \wht)_{\QT} + (v_0, \wht(\cdot,  0))_{\Omega} & & \quad \forall \wht \in \oVdisc{q}.
\end{alignat}
\end{subequations}

For the convergence analysis, it is useful to study the stability of the following perturbed space--time formulation, which we write in compact form using the time reconstruction operator~$\Rt$ in Section~\ref{sec:time-reconstruction}: find~$(\uht, \vht) \in \oVdisc{q} \times \oVdisc{q}$ (with~$q \geq 1$) such that
\begin{subequations}
\label{eq:Johnson}
\begin{alignat}{3}
\nonumber
c^2 (\nabla \vht, \nabla \zht)_{\QT} & = c^2 (\dpt \Rt \nabla \uht, \nabla \zht)_{\QT} - c^2(\nabla u_0, \nabla \zht(\cdot, 0))_{\Omega} \\
\label{eq:Johnson-1}
& \quad + c^2(\nabla \Upsilon, \nabla \zht)_{\QT} & & \quad \forall \zht \in \oVdisc{q}, \\
\label{eq:Johnson-2}
(\dpt \Rt \vht, \wht)_{\QT} & + c^2(\nabla \uht, \nabla \wht)_{\QT} = (f, \wht)_{\QT} + (v_0, \wht(\cdot,  0))_{\Omega} & & \quad \forall \wht \in \oVdisc{q},
\end{alignat}
\end{subequations}
where~$\Upsilon \in L^2(0, T; H_0^1(\Omega))$ is a perturbation function that will represent a projection error in the convergence analysis. 

The forthcoming stability and convergence analysis follows the structure described in Section~\ref{sec:structure-theory}.
\subsection{Stability analysis}
\begin{lemma}[Weak partial bound]
\label{lemma:weak-bound-Johnson}
Any solution~$(\uht, \vht) \in \oVdisc{q} \times \oVdisc{q}$ to~\eqref{eq:Johnson} satisfies the following bound for~$n = 1, \ldots, N$: 
\begin{alignat*}{3}
\frac12 & \Norm{\vht(\cdot, \tn^-)}{L^2(\Omega)}^2 + \frac12 \sum_{m = 1}^{n - 1} \Norm{\jump{\vht}_m}{L^2(\Omega)}^2 + \frac14 \Norm{\vht}{L^2(\SO)}^2 \\
& + \frac{c^2}{2} \Norm{\nabla \uht(\cdot, \tn^-)}{L^2(\Omega)^d}^2 + \frac{c^2}{2} \sum_{m = 1}^{n - 1} \Norm{\jump{\nabla \uht}_m}{L^2(\Omega)^d}^2 + \frac{c^2}{4} \Norm{\nabla \uht}{L^2(\SO)^d}^2 \\
& \quad \le \Norm{f}{L^1(0, \tn; L^2(\Omega))} \Norm{\vht}{L^{\infty}(0, \tn; L^2(\Omega))} + c^2 \Norm{\nabla \Upsilon}{L^1(0, \tn; L^2(\Omega)^d)} \Norm{\nabla \uht}{L^{\infty}(0, \tn; L^2(\Omega)^d)} \\
& \quad \quad + \Norm{v_0}{L^2(\Omega)}^2 + c^2 \Norm{\nabla u_0}{L^2(\Omega)^d}^2.
\end{alignat*}
\end{lemma}
\begin{proof}
Without loss of generality, we prove the result for~$n = N$. 
Choosing~$\wht = \vht$ in~\eqref{eq:Johnson-2}, and using identity~\eqref{eq:identity-Rt-v} from Lemma~\ref{lemma:properties-Rt}, and the H\"older and the Young inequalities, we have
\begin{alignat*}{3}
\frac12  \Norm{\vht}{L^2(\ST)}^2 & + \frac12 \sum_{n = 1}^{N - 1} \Norm{\jump{\vht}_m}{L^2(\Omega)}^2 + \frac12 \Norm{\vht}{L^2(\SO)}^2   + c^2 (\nabla \uht, \nabla \vht)_{\QT} \\
& = (f, \vht)_{\QT} + (v_0, \vht(\cdot, 0))_{\Omega} \\
& \le \Norm{f}{L^1(0, T; L^2(\Omega))} \Norm{\vht}{L^{\infty}(0, T; L^2(\Omega))} + \Norm{v_0}{L^2(\Omega)}^2 + \frac14 \Norm{\vht}{L^2(\SO)}^2.
\end{alignat*}
As for the fourth term on the left-hand side, we can use~\eqref{eq:Johnson-1} and identity~\eqref{eq:identity-Rt-v} to obtain
\begin{alignat*}{3}
c^2 (\nabla \uht, \nabla \vht)_{\QT} 
& = \frac{c^2}{2} \Norm{\nabla \uht}{L^2(\ST)^d}^2 + \frac{c^2}{2} \sum_{n = 1}^{N  - 1} \Norm{\jump{\nabla \uht}_n}{L^2(\Omega)^d}^2 + \frac{c^2}{2} \Norm{\nabla \uht}{L^2(\SO)^d}^2 \\
& \quad  - c^2 (\nabla u_0, \nabla \uht(\cdot, 0))_{\Omega} + c^2(\nabla \Upsilon, \nabla \uht)_{\QT} \\
& \geq \frac{c^2}{2} \Norm{\nabla \uht}{L^2(\ST)^d}^2 + \frac{c^2}{2} \sum_{n = 1}^{N  - 1} \Norm{\jump{\nabla \uht}_n}{L^2(\Omega)^d}^2 + \frac{c^2}{4} \Norm{\nabla \uht}{L^2(\SO)^d}^2 \\
& \quad  - c^2 \Norm{\nabla u_0}{L^2(\Omega)^d}^2 - c^2\Norm{\nabla \Upsilon}{L^1(0, T; L^2(\Omega)^d)} \Norm{\nabla \uht}{L^{\infty}(0, T; L^2(\Omega)^d)},
\end{alignat*}
which leads to the desired result by using again the Young inequality. 
\end{proof}

We recall the functions~$\{\varphi_n\}_{n = 1}^{N}$ defined in~\eqref{lambdadef}, and the Lagrange interpolant~$\ItR$ associated with the left-sided Gauss--Radau points. 

\begin{proposition}[Continuous dependence on the data]
\label{prop:continuous-dependence-DG-Johnson}
Any solution~$(\uht, \vht) \in \oVdisc{q} \times \oVdisc{q}$ to~\eqref{eq:Johnson} satisfies
\begin{alignat*}{3}
\Norm{\vht}{L^{\infty}(0, T; L^2(\Omega))}^2 + c^2 \Norm{\nabla \uht}{L^{\infty}(0, T; L^2(\Omega)^d)}^2 & \lesssim \Norm{f}{L^1(0, T; L^2(\Omega))}^2 + c^2 \Norm{\nabla \Upsilon}{L^1(0, T; L^2(\Omega)^d)}^2 \\
& \quad + \Norm{v_0}{L^2(\Omega)}^2 + c^2\Norm{\nabla u_0}{L^2(\Omega)^d}^2,
\end{alignat*}
where the hidden constant depends only on~$q$.
\end{proposition}
\begin{proof}
For each~$n = 1, \ldots, N$, we define the test function
\begin{equation*}
\wht{}_{|_{Q_m}} : = \begin{cases}
\ItR(\varphi_n \vht) & \text{ if } m = n, \\
0 & \text{ otherwise}.
\end{cases}
\end{equation*}
We focus on the case~$n > 1$, as the argument works for~$n = 1$ with only minor modifications. 

Choosing~$\wht$ in~\eqref{eq:Johnson-2} as above  gives
\begin{equation}
\label{eq:aux-identity-Johnson}
\begin{split}
(\dpt \vht, \ItR(\varphi_n \vht))_{\Qn} + (\jump{\vht}_n, \ItR(\varphi_n \vht)(\cdot, \tn^+))_{\Omega}  + c^2 (\nabla \uht, \nabla \ItR(\varphi_n \vht))_{\Qn} \\
= (f, \ItR(\varphi_n \vht))_{\Qn}.
\end{split}
\end{equation}

Using identities~\eqref{eq:ItR-exact} and~\eqref{eq:identities-varphi-integration-1}, we can show that
\begin{alignat*}{3}
\nonumber
(\dpt \vht, & \ItR (\varphi_n \vht))_{\Qn} + (\jump{\vht}_{n - 1}, \ItR(\varphi_n \vht)(\cdot, \tnmo^+))_{\Omega} \\
& = \frac14 \Norm{\vht(\cdot, \tn^-)}{L^2(\Omega)}^2 + \frac12 \Norm{\jump{\vht}_{n - 1}}{L^2(\Omega)}^2 - \frac12 \Norm{\vht(\cdot, \tnmo^-)}{L^2(\Omega)}^2 + \frac{\lambda_n}{2} \Norm{\vht}{L^2(\Qn)}^2.
\end{alignat*}

As for the third term on the left-hand side of~\eqref{eq:aux-identity-Johnson}, using the exactness of the Gauss--Radau quadrature rule, we get
\begin{alignat*}{3}
c^2 (\nabla \uht, \nabla \ItR (\varphi_n \vht))_{\Qn} & = \sum_{i = 1}^{q+1} \omega_i^{(n)} \varphi_n(s_i^{(n)}) (\nabla \uht(\cdot, s_i^{(n)}), \nabla \vht(\cdot, s_i^{(n)}))_{\Omega} \\
& = c^2 (\nabla \ItR(\varphi_n \uht), \nabla \vht)_{\Qn},
\end{alignat*}
which, together with equation~\eqref{eq:Johnson-1}, identity~\eqref{eq:identities-varphi-integration-1} from Lemma~\ref{lemma:identities-varphi-integration}, and the Young inequality, gives
\begin{alignat*}{3}
c^2(\nabla \uht, \nabla \ItR(\varphi_n \vht))_{\Qn} 
& = (\dpt \nabla \uht, \nabla \ItR (\varphi_n \uht))_{\Qn} + c^2 (\jump{\nabla \uht}_{n-1}, \nabla \ItR (\varphi_n \uht)(\cdot, \tnmo^+))_{\Omega} \\
\nonumber
& \quad + c^2 (\nabla \Upsilon, \nabla \ItR (\varphi_n \uht))_{\Qn}\\
& = (\dpt \nabla \uht, \nabla \varphi_n \uht)_{\Qn} + c^2 (\jump{\nabla \uht}_{n-1}, \nabla \uht(\cdot, \tnmo^+))_{\Omega} \\
\nonumber
& \quad + c^2 (\nabla \Upsilon, \nabla \ItR (\varphi_n \uht))_{\Qn}\\
& = \frac{c^2}{4} \Norm{\nabla \uht(\cdot, \tn^-)}{L^2(\Omega)^d}^2 + \frac{c^2}{2} \Norm{\jump{\nabla \uht}_{n - 1}}{L^2(\Omega)^d}^2 \\
& \quad - \frac{c^2}{2} \Norm{\nabla \uht(\cdot, \tnmo^-)}{L^2(\Omega)^d}^2 + \frac{\lambda_n c^2}{2} \Norm{\nabla \uht}{L^2(\Qn)^d}^2 \\
& \quad + c^2 (\nabla \Upsilon, \nabla \ItR (\varphi_n \uht))_{\Qn}.
\end{alignat*}

Combining these two identities with~\eqref{eq:aux-identity-Johnson}, and using the H\"older inequality, the stability bound in Lemma~\ref{lemma:stab-ItR} of~$\ItR$, and the weak partial bound in Lemma~\ref{lemma:weak-bound-Johnson}, we get
\begin{alignat}{3}
\nonumber
\frac14 \Norm{\vht(\cdot, \tn^-)}{L^2(\Omega)}^2 & + \frac12 \Norm{\jump{\vht}_{n - 1}}{L^2(\Omega)}^2 + \frac{\lambda_n}{2} \Norm{\vht}{L^2(\Qn)}^2 \\
\nonumber
+ \frac{c^2}{4} \Norm{\nabla \uht(\cdot, \tn^-)}{L^2(\Omega)^d}^2 & + \frac12 \Norm{\jump{\nabla \uht}_{n - 1}}{L^2(\Omega)^d}^2 + \frac{\lambda_n c^2}{2} \Norm{\nabla \uht}{L^2(\Qn)^d}^2 \\
\nonumber
& \le \frac12 \Norm{\vht(\cdot, \tnmo^-)}{L^2(\Omega)}^2 + \frac{c^2}{2} \Norm{\nabla \uht(\cdot, \tnmo^-)}{L^2(\Omega)^d}^2 + (f, \ItR (\varphi_n \vht) )_{\Qn} \\
\nonumber
& \quad + c^2 (\nabla \Upsilon, \nabla \ItR(\varphi_n \uht))_{\Qn} \\
\nonumber
& \le \frac12 \Norm{\vht(\cdot, \tnmo^-)}{L^2(\Omega)}^2 + \frac{c^2}{2} \Norm{\nabla \uht(\cdot, \tnmo^-)}{L^2(\Omega)^d}^2 \\
\nonumber
& \quad + \CSI \Norm{f}{L^1(\In; L^2(\Omega))} \Norm{\vht}{L^{\infty}(\In; L^2(\Omega))} \\
\nonumber
& \quad + c^2 \CSI \Norm{\nabla \Upsilon}{L^1(\In; L^2(\Omega)^d)} \Norm{\nabla \uht}{L^{\infty}(\In; L^2(\Omega)^d)} \\
\nonumber
& \le \Norm{v_0}{L^2(\Omega)}^2 + c^2 \Norm{\nabla u_0}{L^2(\Omega)}^2 + \Norm{f}{L^1(0, \tnmo; L^2(\Omega))} \Norm{\vht}{L^{\infty}(0, \tnmo; L^2(\Omega))} \\
\nonumber
& \quad + \CSI \Norm{f}{L^1(\In; L^2(\Omega))} \Norm{\vht}{L^{\infty}(\In; L^2(\Omega))} \\
\nonumber
& \quad + c^2 \Norm{\nabla \Upsilon}{L^1(0, \tnmo; L^2(\Omega)^d)} \Norm{\nabla \uht}{L^{\infty}(0, \tnmo; L^2(\Omega)^d)} \\
\nonumber
& \quad + c^2 \CSI \Norm{\nabla \Upsilon}{L^1(\In; L^2(\Omega)^d)} \Norm{\nabla \uht}{L^{\infty}(\In; L^2(\Omega)^d)} \\
\nonumber
& \lesssim \Norm{v_0}{L^2(\Omega)}^2 + c^2 \Norm{\nabla u_0}{L^2(\Omega)^d}^2 + \Norm{f}{L^1(0, \tn; L^2(\Omega))} \Norm{\vht}{L^{\infty}(0, \tn; L^2(\Omega))} \\
\nonumber
& \quad + c^2 \Norm{\nabla \Upsilon}{L^1(0, \tn; L^2(\Omega)^d)}\Norm{\nabla \uht}{L^{\infty}(0, \tn; L^2(\Omega)^d)} ,
\end{alignat}
where it is crucial that the hidden constant is independent of the time interval~$\In$ and depends only on the degree of approximation~$q$.

From the inverse estimate in Lemma~\ref{lemma:L2-Linfty}, we deduce that
\begin{alignat*}{3}
\frac{1}{4 \Cinv} \Norm{\vht}{L^{\infty}(\In; L^2(\Omega))}^2 & \le \frac{\lambda_n}{2} \Norm{\vht}{L^2(\Qn)}^2, \\
\frac{c^2}{4 \Cinv} \Norm{\nabla \uht}{L^{\infty}(\In; L^2(\Omega)^d)}^2&  \le \frac{\lambda_n c^2}{2} \Norm{\nabla \uht}{L^2(\Qn)^d}^2,
\end{alignat*}
which allows us to show that
\begin{alignat*}{3}
\Norm{\vht}{L^{\infty}(\In; L^2(\Omega))}^2 + c^2 \Norm{\nabla \uht}{L^{\infty}(\In; L^2(\Omega)^d)}^2 
& \lesssim \Norm{v_0}{L^2(\Omega)}^2 + c^2 \Norm{\nabla u_0}{L^2(\Omega)^d}^2 \\
& \quad + \Norm{f}{L^1(0, T; L^2(\Omega))} \Norm{\vht}{L^{\infty}(0, T; L^2(\Omega))}\\
& \quad + c^2 \Norm{\nabla \Upsilon}{L^1(0, T; L^2(\Omega)^d)}\Norm{\nabla \uht}{L^{\infty}(0, T; L^2(\Omega)^d)}.
\end{alignat*}
For each of the terms on the left-hand side, one can choose an index~$n$ where it achieves its
maximum value. This, combined with Lemma~\ref{lemma:weak-bound-Johnson}, gives the desired result.
\end{proof}
The existence of a unique solution to~\eqref{eq:Johnson} is an immediate consequence of Proposition~\ref{prop:continuous-dependence-DG-Johnson}, as it corresponds to a linear system with a nonsingular matrix.

\subsection{Convergence analysis}
Recalling Definition~\ref{def:Pt} of the Thom\'ee projection operator~$\PtTh$, we define the space--time projection~$\Piht := \PtTh \circ \Rh$ and the error functions
\begin{alignat*}{3}
\ev := v - \vht = (v - \Piht v) + \Piht \ev \quad \text{ and } \quad \eu := u - \uht = (u - \Piht u) + \Piht \eu.
\end{alignat*}

\begin{theorem}[Estimates for the discrete error]
\label{thm:a-priori-Johnson}
Let~$u$ be the solution to the continuous weak formulation~\eqref{eq:second-order-weak-formulation-wave} and~$v = \dpt u$, and let~$(\uht, \vht) \in \oVdisc{q} \times \oVdisc{q}$ be the solution to the discrete space--time formulation~\eqref{eq:original-Johnson}. Assume that~$\Omega$ satisfies the elliptic regularity assumption~\eqref{eq:elliptic-regularity-Omega}, and that the continuous solution~$u$ satisfies
\begin{equation*}
\begin{split}
u \in H^1(0, T; & H_0^1(\Omega)) \cap W^{2, 1}(0, T; H^{r+1}(\Omega)) \\
& \text{ with } \Delta u \in W^{s+1, 1}(0, T; L^2(\Omega)) \text{ and } \nabla u \in W^{s+2, 1}(0, T; L^2(\Omega)^d),
\end{split}
\end{equation*}
for some~$0 \le r \le p$ and~$1 \le s \le q$. Then, the following estimate holds:
\begin{alignat}{3}
\nonumber
\Norm{\Piht \ev}{L^{\infty}(0, T; L^2(\Omega))} & + c \Norm{\nabla \Piht \eu}{L^{\infty}(0, T; L^2(\Omega)^d)}  + \Seminorm{\Piht \ev}{\sf J} + c \Seminorm{\nabla \Piht \eu}{\sf J} \\
\nonumber
& \quad \lesssim h^{r+1} \big( \Seminorm{\dptt u}{L^1(0, T; H^{r+1}(\Omega))}  + \Seminorm{v_0}{H^{r+1}(\Omega)} \big)\\
\nonumber
& \qquad + \tau^{q+1} \big( c^2 \Norm{\Delta \dpt^{(q+1)} u}{L^1(0, T; L^2(\Omega))}  + c \Norm{\nabla \dpt^{(q+2)} u}{L^1(0, T; L^2(\Omega)^d)} \big).
\end{alignat}
\end{theorem}
\begin{proof}
Adding and subtracting suitable terms in~\eqref{eq:Johnson-1} and using the orthogonality properties of~$\PtTh$ and~$\Rh$, Lemma~\ref{lemma:orthogonality-DG-PtTh}, and the fact that~$v = \dpt u$, we get
\begin{alignat}{3}
\nonumber
c^2 (\nabla  \Piht \ev, \nabla \zht)_{\QT} & - c^2 (\nabla \dpt \Rt \Piht \eu, \nabla \zht)_{\QT} \\
\nonumber
& = 
c^2 ( \nabla \Piht v, \nabla \zht)_{\QT} - c^2(\nabla \dpt \Rt \Piht u, \nabla \zht)_{\QT} \\
\nonumber 
& \quad + c^2 (\nabla u_0, \nabla \zht)_{\Omega} \\
\nonumber
& = 
c^2 ( \nabla \PtTh v, \nabla \zht)_{\QT} -c^2 (\nabla \dpt \Rt \PtTh u, \nabla \zht)_{\QT} \\
\nonumber
& \quad + c^2 (\nabla u_0, \nabla \zht)_{\Omega} \\
\label{eq:aux-error-Johnson-1}
& = - c^2 (\nabla (\Id - \PtTh) v, \nabla \zht)_{\QT}.
\end{alignat}

We can proceed similarly to get the following identity from~\eqref{eq:Johnson-2} and the consistency of the method:
\begin{alignat}{3}
\nonumber
(\dpt \Rt \Piht \ev, \wht)_{\QT} & + c^2 (\nabla \Piht \eu, \nabla \wht)_{\QT} \\
\nonumber
& = - (f, \wht)_{\QT} - (v_0, \wht(\cdot, 0))_{\Omega} + (\dpt \Rt \Piht v, \wht)_{\QT} \\
\nonumber
& \quad + c^2(\nabla \Piht u, \nabla \wht)_{\QT} \\
\nonumber
& = -(\dpt v, \wht)_{\QT} - c^2 (\nabla u, \nabla \wht)_{\QT} + (\dpt \Rh v, \wht)_{\QT} \\
\nonumber
& \quad - ((\Id - \Rh) v_0, \wht(\cdot, 0))_{\Omega}  + c^2(\nabla \PtTh u, \nabla \wht)_{\QT} \\
\nonumber
& = - ((\Id - \Rh) \dpt v, \wht)_{\QT} + c^2((\Id - \PtTh) \Delta u, \wht)_{\QT}  \\
\label{eq:aux-error-Johnson-2}
& \quad - ((\Id - \Rh) v_0, \wht(\cdot,0))_{\Omega}.
\end{alignat}

Putting together~\eqref{eq:aux-error-Johnson-1} and~\eqref{eq:aux-error-Johnson-2}, it follows that the discrete error pair~$(\Piht \eu, \Piht \ev) \in \oVdisc{q} \times \oVdisc{q}$ solves a problem of the form~\eqref{eq:Johnson} with data
\begin{equation*}
\widetilde{f} = - (\Id - \Rh) \dpt v + c^2 (\Id - \PtTh) \Delta u, \quad \Upsilon = \PtTh v - v, \quad \widetilde{u}_0 = 0, \quad \widetilde{v}_0 = - (\Id - \Rh)v_0. 
\end{equation*}

The error estimate then follows by using the continuous dependence on the data in Proposition~\ref{prop:continuous-dependence-DG-Johnson}, and the approximation properties of~$\PtTh$ and~$\Rh$ in Lemma~\ref{lemma:estimates-Ptau} (relying on Lemma~\ref{lemma:stab-Pt}) and~\ref{lemma:Estimates-Rh}, respectively. 
\end{proof}

\begin{corollary}[\emph{A priori} error estimates]
Under the assumptions of Theorem~\ref{thm:a-priori-Johnson}, and assuming that the continuous solution~$u$ to~\eqref{eq:second-order-weak-formulation-wave} is sufficiently regular, there hold
\begin{alignat*}{3}
\Norm{\ev}{L^{\infty}(0, T; L^2(\Omega))} & \lesssim h^{p+1} \Big(\Seminorm{\dptt u}{L^1(0, T; H^{p+1}(\Omega))} + \Seminorm{v_0}{H^{p+1}(\Omega)} + \Seminorm{\dpt u}{L^{\infty}(0, T; H^{p+1}(\Omega))} \Big) \\
& \quad + \tau^{q+1} \big( c^2 \Norm{\Delta \dpt^{(q+1)} u}{L^1(0, T; L^2(\Omega))} + c \Norm{\nabla \dpt^{(q+2)} u}{L^1(0, T; L^2(\Omega)^d)} \\
& \qquad \qquad + \Norm{\dpt^{(q + 2)} u}{L^{\infty}(0, T; L^2(\Omega))}\big), \\
c \Norm{\nabla \eu}{L^{\infty}(0, T; L^2(\Omega)^d)} & \lesssim  h^{p} \big(\Seminorm{\dptt u}{L^1(0, T; H^p(\Omega))} + \Seminorm{v_0}{H^p(\Omega)} + c \Seminorm{u}{L^{\infty}(0, T; H^{p+1}(\Omega))} \big) \\
& \quad + \tau^{q+1} \big(c^2 \Norm{\Delta \dpt^{(q+1)} u}{L^1(0, T; L^2(\Omega))} + c \Norm{\nabla \dpt^{(q+2)} u}{L^1(0, T; L^2(\Omega)^d)} \\
& \qquad \qquad + c \Norm{\nabla \dpt^{(q+1)} u}{L^{\infty}(0, T; L^2(\Omega)^d)}\big).
\end{alignat*}
\end{corollary}

\begin{remark}[Improved results]
Theorem~\ref{thm:a-priori-Johnson} does not require any constraint on the relation of~$\tau$ and~$h$, which improves the corresponding result in~\cite[Thm.~3.1 and Rem.~3.1]{Johnson:1993}. Moreover, compared to~\cite{Antonietti_Mazzieri_Migliorini:2023}, our stability and convergence analysis show that the presence of a damping term is not necessary. 
\eremk
\end{remark}

\begin{remark}[Computational aspects]
A numerical assessment comparing the method in Section~\ref{sec:French-Peterson-wave} with the present method is detailed in~\cite{Kocher_Bause:2014}. Furthermore, some theoretical and computational aspects of the scheme in Section~\ref{sec:plain-vanilla-wave} and the one described here have been recently compared in~\cite{Antonietti_Artoni_Ciaramella_Mazzieri:2025}. A technique for the solution to the linear system stemming from the space--time method~\eqref{eq:Johnson} was presented in~\cite{Banks_etal:2014}.
\eremk
\end{remark}

\section{DG--CG time discretization for the second-order formulation of the wave equation (Walkington, 2014)}
\label{sec:Walkington-wave}
We now study the DG--CG method proposed by Walkington in~\cite{Walkington:2014}, where the main idea for the analysis in the previous sections was introduced.

The discrete space--time formulation reads: find~$\uht \in \oVcont{q}$ (with~$q \geq 2$) such that~$\uht(\cdot, 0) = \Rh u_0$  and
\begin{equation}
\label{eq:Walkington}
\Bht(\uht, \wht) = \ell(\wht) \quad \forall \wht \in \oVdisc{q-1},
\end{equation}
where the bilinear form~$\Bht : \oVcont{q} \times \oVdisc{q-1} \to \R$ and the linear functional~$\ell : \oVdisc{q-1} \to \R$ are given by
\begin{alignat*}{3}
\nonumber
\Bht(\uht, \wht) & := \sum_{n = 1}^N (\dptt \uht, \wht)_{\Qn} + \sum_{n = 1}^{N - 1} (\jump{\dpt \uht}_n, \wht(\cdot, \tn^+))_{\Omega} + (\dpt \uht, \wht)_{\SO} \\
& \quad + c^2 (\nabla \uht, \nabla \wht)_{\QT}, \\
\ell(\wht) & :=  (f, \wht)_{\QT} + (v_0, \wht(\cdot, 0))_{\Omega}.
\end{alignat*}
The bilinear form~$\Bht(\cdot, \cdot)$ can be written in compact form using the definition of the time reconstruction operator~$\Rt$ in Section~\ref{sec:time-reconstruction} as
\begin{equation*}
\Bht(\uht, \wht) = (\dpt\Rtq \dpt \uht, \wht)_{\QT} + c^2 (\nabla \uht, \nabla \wht)_{\QT}. 
\end{equation*}

The name DG--CG refers to the fact that one can recover method~\eqref{eq:Walkington} as a combination of the DG method in Section~\ref{sec:Johnson-wave} and the CG method in Section~\ref{sec:French-Peterson-wave}. More precisely, this method can be written as: find~$(\uht, \vht) \in \oVcont{q} \times \oVdisc{q-1}$ (with~$q \geq 2$) such that~$\uht(\cdot, 0) = \Rh u_0$ and
\begin{subequations}
\begin{alignat}{3}
\label{eq:Walkington-mixed-1}
c^2 (\nabla \vht, \nabla \zht)_{\QT} & = c^2 ( \nabla \dpt \uht, \nabla \zht)_{\QT} & & \qquad \forall \zht \in \oVdisc{q-1}, \\
\nonumber
\sum_{n = 1}^{N} (\dpt \vht, \wht)_{\Qn}  + \sum_{n = 1}^{N - 1} (\jump{\vht}_n, & \wht(\cdot, \tn^+))_{\Omega} + (\vht, \wht)_{\SO}\\
\label{eq:Walkington-mixed-2}
 + c^2 (\nabla \uht, \nabla \wht)_{\QT} &  = (f, \wht)_{\QT} + (v_0, \wht(\cdot, 0))_{\Omega} & & \qquad \forall \wht \in \oVdisc{q-1}.
\end{alignat}
\end{subequations}
Due to the choice of the discrete spaces, equation~\eqref{eq:Walkington-mixed-1} leads to the identity~$\vht = \dpt \uht$. Method~\eqref{eq:Walkington} can then be recovered by substituting such an identity in~\eqref{eq:Walkington-mixed-2}.

The forthcoming stability and convergence analysis follows the structure described in Section~\ref{sec:structure-theory}.
\subsection{Stability analysis}
\begin{lemma}[Weak partial bound on~$\uht$]
\label{lemma:partial-bound-Walkington}
Any solution~$\uht \in \oVcont{q}$ to~\eqref{eq:Walkington} satisfies the following bound for~$n = 1, \ldots, N$:
\begin{alignat*}{3}
\nonumber
\frac12 & \Norm{\dpt \uht(\cdot, \tn^-)}{L^2(\Omega)}^2 + \frac12 \sum_{m = 1}^{n-1} \Norm{\jump{\dpt \uht}_m}{L^2(\Omega)}^2 + \frac14 \Norm{\dpt \uht}{L^2(\SO)}^2 \\
\nonumber
& + \frac{c^2}{2}\Norm{\nabla \uht}{L^2(\Sn)^d}^2 + \frac{c^2}{2} \Norm{\nabla \uht}{L^2(\SO)^d}^2 \\
& \quad \le \Norm{f}{L^1(0, \tn; L^2(\Omega))} \Norm{\dpt \uht}{L^{\infty}(0, \tn; L^2(\Omega))} + \Norm{v_0}{L^2(\Omega)}^2 + c^2 \Norm{\nabla u_0}{L^2(\Omega)^d}^2.
\end{alignat*}
\end{lemma}
\begin{proof}
Without loss of generality, we show the result for~$n = N$. We choose~$\wht = \dpt \uht \in \oVdisc{q-1}$ in~\eqref{eq:Walkington}. 

The following identity can be obtained using Lemma~\ref{lemma:properties-Rt}, and the fact that~$\uht (\cdot, 0) = \Rh u_0$: 
\begin{alignat}{3}
\nonumber
\Bht(\uht, \dpt \uht) 
& = \frac12 \Seminorm{\dpt \uht}{\sf J}^2 + \frac{c^2}{2} \Norm{\nabla \uht}{L^2(\ST)^d}^2 - \frac{c^2}{2} \Norm{\nabla \Rh u_0}{L^2(\Omega)^d}^2.
\end{alignat}

Consequently, using the H\"older and the Young inequalities, we get
\begin{alignat}{3}
\nonumber
\frac12 & \Seminorm{\dpt \uht}{\sf J}^2  + \frac{c^2}{2} \Norm{\nabla \uht}{L^2(\ST)^d}^2 \\
\nonumber
& = (f, \dpt \uht)_{\QT} + (v_0, \dpt \uht(\cdot, 0))_{\Omega} + \frac{c^2}{2} \Norm{\nabla \Rh u_0}{L^2(\Omega)^d} \\
\label{eq:aux-partial-bound-Walkington}
& \le \Norm{f}{L^1(0, T; L^2(\Omega))} \Norm{\dpt \uht}{L^{\infty}(0, T; L^2(\Omega))} + \Norm{v_0}{L^2(\Omega)}^2 + \frac14 \Norm{\dpt \uht}{L^2(\SO)}^2 + \frac{c^2}{2} \Norm{\nabla u_0}{L^2(\Omega)^d}^2.
\end{alignat}
The result then follows by adding~\eqref{eq:aux-partial-bound-Walkington} and the bound
\begin{equation*}
    \frac{c^2}{2} \Norm{\nabla \uht}{L^2(\SO)^d}^2 \le \frac{c^2}{2} \Norm{\nabla u_0}{L^2(\Omega)^d}^2.
\end{equation*}
\end{proof}

In this setting, one has to modify the definition of the weight functions in Section~\ref{sec:auxiliary-weight} as follows (see~\cite[Thm.~4.5]{Walkington:2014}): for~$n = 1, \ldots, N$,
\begin{equation*}
    \tvarphi_n (t) = 1 - \tlambda_n(t - \tnmo) \quad \text{ with~$\tlambda_n := \frac{\xi_q}{\tau_n}$ \ \ and\ \  $\xi_q := \frac{1}{4(2q +1)}$,}
\end{equation*}
which corresponds to the Taylor polynomial of degree~$1$ centered at~$\tnmo$ of~$\exp(-(t - \tnmo) \tlambda_n)$. 
Moreover, the following lemma is required (see~\cite[Cor. 4.4]{Walkington:2014} and~\cite[Lemma 3.3]{Gomez-Nikolic:2025}). 

\begin{lemma} \label{lemma: weight function}
For~$n = 1, \ldots, N$, the following bound holds:
\begin{alignat*}{3}
\Norm{(\Id - \Pi_{q - 1}^t)(\tvarphi_n \wht)(\cdot, \tnmo^+)}{L^2(\Omega)} & \le \frac{\sqrt{\tlambda_n}}{2} \Norm{\wht}{L^2(\Qn)} \qquad  \forall \wht \in \oVdisc{q-1}.
\end{alignat*}
Moreover, for any Hilbert space~${\HH}$ with inner product~$(\cdot, \cdot)_{\HH}$, there hold (see also Footnote~\ref{footnote:n=1})
\begin{subequations}
\begin{alignat}{3}
\nonumber
\int_{\In} (\dpt v_{\tau}, & \tvarphi_n v_{\tau})_{{\HH}} \dt  + (\jump{v_{\tau}}_{n-1}, \tvarphi_n(\tnmo) v_{\tau}(\tnmo^+))_{\HH} \\
\nonumber
& = \frac{1-\xi_q}{2} \Norm{v_{\tau}(\tn^-)}{{\HH}}^2 + \frac12 \Norm{\jump{v_{\tau}}_{n-1}}{{\HH}}^2 - \frac12 \Norm{v_{\tau}(\tnmo^-)}{{\HH}}^2 \\
\label{eq:identities-tvarphi-integration-1}
& \quad + \frac{\tlambda_n}{2} \Norm{v_{\tau}}{L^2(\In; {\HH})}^2 & & \quad \forall v_{\tau} \in \Pp{q}{\Tt} \otimes {\HH}, \\
\nonumber
\int_{\In} (\dpt v_{\tau}, & \tvarphi_n v_{\tau})_{{\HH}} \dt \\
\label{eq:identities-tvarphi-integration-2}
& = \frac{1- \xi_q}{2} \Norm{v_{\tau}(\tn)}{{\HH}}^2 - \frac12 \Norm{v_{\tau}(\tnmo)}{{\HH}}^2 + \frac{\tlambda_n}{2} \Norm{v_{\tau}}{L^2(\In; {\HH})}^2 & & \quad \forall v_{\tau} \in \Pp{q}{\Tt} \cap C^0[0, T] \otimes {\HH}.
\end{alignat}
\end{subequations}
\end{lemma}

\begin{proposition}[Continuous dependence on the data]
\label{prop:continuous-dependence-DG-Walkington}
Any solution~$\uht \in \oVcont{q}$ to~\eqref{eq:Walkington} satisfies
\begin{alignat*}{3}
\nonumber
\Norm{\dpt \uht}{L^{\infty}(0, T; L^2(\Omega))}^2 & + c^2 \Norm{\nabla \uht}{C^0([0, T]; L^2(\Omega)^d)}^2 + \Seminorm{\dpt \uht}{\sf J}^2 + c^2 \Norm{\nabla \uht}{L^2(\ST)^d}^2 \\
& \lesssim \Norm{f}{L^1(0, T; L^2(\Omega))}^2 + \Norm{v_0}{L^2(\Omega)}^2 + c^2\Norm{\nabla u_0}{L^2(\Omega)^d}^2,
\end{alignat*}
\end{proposition}
\begin{proof}
For~$n = 1, \ldots, N$, we choose the following test function in~\eqref{eq:Walkington} as
\begin{equation*}
\wht{}_{|_{Q_m}}  := \begin{cases}
\Pi_{q-1}^t (\tvarphi_n \dpt \uht) & \text{ if } m = n, \\
0 & \text{ otherwise.}
\end{cases}
\end{equation*}
We focus on the case~$n > 1$, as the argument for~$n = 1$ only requires minor modifications. 

From the definition of~$\Bht(\cdot, \cdot)$, we have
\begin{alignat*}{3}
\Bht(\uht, \wht) &  = (\dptt \uht, \Pi_{q-1}^t (\tvarphi_n \dpt \uht))_{\Qn} + (\jump{\dpt \uht}_{n - 1}, \Pi_{q-1}^t (\tvarphi_n \dpt \uht(\cdot, \tnmo^+)))_{\Omega} \\
& \quad + c^2 (\nabla \uht, \nabla \Pi_{q-1}^t(\tvarphi_n, \dpt \uht))_{\Qn} \\
& = (\dptt \uht, \tvarphi_n \dpt \uht)_{\Qn} + (\jump{\dpt \uht}_{n-1}, \dpt \uht(\cdot, \tnmo^+))_{\Omega} \\
& \quad - (\jump{\dpt \uht}_{n-1}, (\Id - \Pi_{q-1}^t)(\tvarphi_n \dpt \uht)(\cdot, \tnmo^+))_{\Omega} \\
& \quad + c^2 (\nabla \uht, \tvarphi_n \nabla \dpt \uht)_{\Qn} - c^2 (\nabla \uht, \nabla (\Id - \Pi_{q-1}^t)(\tvarphi_n \dpt \uht))_{\Qn} \\
& =: J_1 + J_2 + J_3 + J_4 + J_5. 
\end{alignat*}

Identity~\eqref{eq:identities-tvarphi-integration-1} from Lemma~\ref{lemma: weight function} gives
\begin{equation}
\label{eq:J1-Walkington}
\begin{split}
J_1 + J_2 & = \frac{1-\xi_q}{2} \Norm{\dpt \uht(\cdot, \tn^-)}{L^2(\Omega)}^2 + \frac12 \Norm{\jump{\dpt \uht}_{n-1}}{L^2(\Omega)}^2 \\
& \quad - \frac12 \Norm{\dpt \uht(\cdot, \tnmo^-)}{L^2(\Omega)}^2 + \frac{\tlambda_n}{2} \Norm{\dpt \uht}{L^2(\Qn)}^2.
\end{split}
\end{equation}

As for~$J_3$, we use the Young inequality and Lemma~\ref{lemma: weight function} to obtain
\begin{alignat}{3}
\nonumber
J_3 &  = (\jump{\dpt \uht}_{n-1}, (\Id - \Pi_{q-1}^t) (\tvarphi_n \dpt \uht)(\cdot, \tnmo^+))_{\Omega}  \\
\nonumber
& \geq - \frac14 \Norm{\jump{\dpt \uht}_{n-1}}{L^2(\Omega)}^2 - \Norm{(\Id - \Pi_{q-1}^t)(\tvarphi_n \dpt \uht)(\cdot, \tnmo^+)}{L^2(\Omega)}^2  \\
\label{eq:J3-Walkington}
& \geq - \frac14 \Norm{\jump{\dpt \uht}_{n-1}}{L^2(\Omega)}^2 - \frac{\tlambda_n}{4} \Norm{\dpt \uht}{L^2(\Qn)}^2. 
\end{alignat}

Using identity~\eqref{eq:identities-tvarphi-integration-2} for~$J_4$, we have
\begin{alignat}{3}
\nonumber
J_4 & = c^2(\nabla \uht, \tvarphi_n \nabla \dpt \uht)_{\Qn} \\
\label{eq:J4-Walkington}
& = \frac{(1 - \xi_q)c^2}{2} \Norm{\nabla \uht}{L^2(\Sn)^d}^2 - \frac{c^2}{2} \Norm{\nabla \uht}{L^2(\Snmo)^d}^2 + \frac{\tlambda_n}{2} \Norm{\nabla \uht}{L^2(\Qn)^d}^2.
\end{alignat}

One can easily adapt Lemma~\ref{lemma:bound-pi-time-varphi_n} to deduce that~$J_5 \geq 0$.  Combining bounds \eqref{eq:J1-Walkington}--\eqref{eq:J4-Walkington}, using the weak partial bound in Lemma~\ref{lemma:partial-bound-Walkington} and the stability in Lemma~\ref{lemma:stab-pi-time} of~$\Pi_{q-1}^t$, we get
\begin{alignat*}{3}
\frac{1 - \xi_q}{2} \Norm{\dpt \uht(\cdot, \tn^-)}{L^2(\Omega)}^2 & + \frac14 \Norm{\jump{\dpt \uht}_{n-1}}{L^2(\Omega)}^2 + \frac{\tlambda_n}{4} \Norm{\dpt \uht}{L^2(\Qn)}^2 \\
+ \frac{c^2(1 - \xi_q)}{2} & \Norm{\nabla \uht}{L^2(\Sn)^d}^2  + \frac{\tlambda_n}{2} \Norm{\nabla \uht}{L^2(\Qn)^d}^2 \\
& \le \frac12 \Norm{\dpt \uht(\cdot, \tnmo^-)}{L^2(\Omega)}^2 + \frac{c^2}{2} \Norm{\nabla \uht}{L^2(\Snmo)^d}^2 \\
& \quad + \Norm{f}{L^1(\In; L^2(\Omega))} \Norm{\Pi_{q-1}^t(\tvarphi_n \dpt \uht)}{L^{\infty}(\In; L^2(\Omega))} \\
& \le  \Norm{v_0}{L^2(\Omega)}^2 + c^2 \Norm{\nabla u_0}{L^2(\Omega)^d}^2 + \Norm{f}{L^1(0, \tnmo; L^2(\Omega))} \Norm{\dpt \uht}{L^{\infty}(0, \tnmo; L^2(\Omega))} \\
& \quad + \CS \Norm{f}{L^1(\In; L^2(\Omega))} \Norm{\dpt \uht}{L^{\infty}(\In; L^2(\Omega))} \\ 
& \lesssim \Norm{v_0}{L^2(\Omega)^2}^2 + c^2 \Norm{\nabla u_0}{L^2(\Omega)^d}^2 + \Norm{f}{L^1(0, T; L^2(\Omega))} \Norm{\dpt \uht}{L^{\infty}(0, T; L^2(\Omega))} ,
\end{alignat*}
where the hidden constant depends only on~$q$. Using the polynomial inverse estimate in Lemma~\ref{lemma:L2-Linfty}, we get
\begin{equation*}
\begin{split}
    \Norm{\dpt \uht}{L^{\infty}(\In; L^2(\Omega))}^2 & + c^2 \Norm{\nabla \uht}{L^{\infty}(\In; L^2(\Omega)^d)}^2 \\
    & \lesssim \Norm{v_0}{L^2(\Omega)^2}^2 + c^2 \Norm{\nabla u_0}{L^2(\Omega)^d}^2 + \Norm{f}{L^1(0, T; L^2(\Omega))} \Norm{\dpt \uht}{L^{\infty}(0, T; L^2(\Omega))}.
\end{split}
\end{equation*}
For each of the terms on the left-hand side, one
 can choose an index~$n$ where it achieves its maximum value. This, combined with Lemma~\ref{lemma:partial-bound-Walkington}, gives the desired result.
\end{proof}
The existence of a unique solution to~\eqref{eq:Walkington} is an immediate consequence of Proposition~\ref{prop:continuous-dependence-DG-Walkington}, as it corresponds to a linear system with a nonsingular matrix.

\subsection{Convergence analysis}
Recalling Definition~\ref{def:PtW} of the Walkington projection~$\PtW$, we define the space--time projection~$\Piht = \PtW \circ \Rh$ and the error functions
\begin{equation*}
\eu := u - \uht = (\Id - \Piht) u + \Piht \eu. 
\end{equation*}
\begin{theorem}[Estimates for the discrete error]
\label{thm:estimate-discrete-error-Walkington}
Let the assumptions of Proposition~\ref{prop:continuous-dependence-DG-Walkington} hold. Let~$u$ be the continuous weak solution to~\eqref{eq:second-order-weak-formulation-wave}, and~$\uht \in \oVcont{q}$ be the solution to~\eqref{eq:Walkington}.
Assume that~$\Omega$ satisfies the elliptic regularity assumption~\eqref{eq:elliptic-regularity-Omega}, and that the continuous solution~$u$ satisfies
\begin{equation*}
\begin{split}
u \in H^1(0, T; & H_0^1(\Omega)) \cap W^{2, 1}(0, T; H^{r+1}(\Omega)) \ \text{ with } \Delta u \in W^{s+1, 1}(0, T; L^2(\Omega)),
\end{split}
\end{equation*}
for some~$0 \le r \le p$ and~$1 \le s \le q$. Then, the following estimate holds:
\begin{equation*}
\begin{split}
\Norm{\dpt \Piht \eu}{L^{\infty}(0, T; L^2(\Omega))} & + c \Norm{\nabla \Piht \eu}{C^0([0, T]; L^2(\Omega)^d)} + \Seminorm{\dpt \Piht \eu}{\sf J} + c \Norm{\nabla \Piht \eu}{L^2(\ST)^d} \\
& \lesssim h^{r+1} \big(\Seminorm{\dptt u}{L^1(0, T; H^{r+1}(\Omega))} + \Seminorm{v_0}{H^{r+1}(\Omega)} \big) + c^2 \tau^{s+1} \Norm{\Delta \dpt^{(s+1)}u}{L^1(0, T; L^2(\Omega))}.
\end{split}
\end{equation*}
where the hidden constant depends only on~$q$. 
\end{theorem}
\begin{proof}
Since the space--time variational formulation~\eqref{eq:Walkington} is consistent, we have the identity
\begin{alignat}{3}
\nonumber
\Bht(\Piht \eu, \wht) & = -\Bht((\Id - \Piht)u, \wht) \\
\nonumber
& = - (\dpt\Rt \dpt (\Id - \Piht) u, \wht)_{\QT} - c^2 (\nabla (\Id - \Piht) u, \nabla \wht)_{\QT}.
\end{alignat}

The identity in Lemma~\ref{lemma:orthogonality-DG-PtW} immediately gives
\begin{alignat*}{3}
- (\dpt\Rt \dpt (\Id - \Piht) u, \wht)_{\QT} 
& = - ((\Id - \Rh) \dptt u, \wht)_{\QT} - ((\Id - \Rh) v_0, \wht(\cdot, 0))_{\Omega}.
\end{alignat*}

Moreover, using the orthogonality properties of~$\Rh$ and integration by parts in space, we obtain
\begin{alignat*}{3}
-c^2 (\nabla (\Id - \Piht) u, \nabla \wht)_{\QT} = - c^2 ((\Id - \PtW) \nabla u, \nabla \wht)_{\QT} = c^2 ((\Id - \PtW) \Delta u, \wht)_{\QT}.
\end{alignat*}
Therefore, the discrete error~$\Piht \eu$ solves a problem of the form of~\eqref{eq:Walkington} with data
\begin{equation*}
\widetilde{f} = -(\Id - \Rh) \dptt u + c^2 (\Id - \PtW) \Delta u, \quad \Piht \eu (\cdot, 0) = 0, \quad  \text{ and } \quad \widetilde{v}_0 = - (\Id - \Rh) v_0.
\end{equation*}
The estimate then follows by using the approximation properties of~$\PtW$ and~$\Rh$ in Lemmas~\ref{lemma:estimates-PtW} and~\ref{lemma:Estimates-Rh}, respectively. 
\end{proof}

\begin{corollary}[\emph{A priori} error estimates]
Under the assumptions of Theorem~\ref{thm:estimate-discrete-error-Walkington}, and assuming that the continuous solution~$u$ to~\eqref{eq:second-order-weak-formulation-wave} is sufficiently regular, there hold
\begin{alignat*}{3}
\Norm{\dpt \eu}{L^{\infty}(0, T; L^2(\Omega))} & \lesssim h^{p+1} \Big(\Seminorm{\dptt u}{L^1(0, T; H^{p+1}(\Omega))} + \Seminorm{v_0}{H^{p+1}(\Omega)} + \Seminorm{\dpt u}{L^{\infty}(0, T; H^{p+1}(\Omega))} \Big) \\
& \quad + \tau^{q} \big( c^2 \Norm{\Delta \dpt^{(q)} u}{L^1(0, T; L^2(\Omega))} + \Norm{\dpt^{(q + 1)} u}{L^{\infty}(0, T; L^2(\Omega))}\big), \\
c \Norm{\nabla \eu}{L^{\infty}(0, T; L^2(\Omega)^d)} & \lesssim  h^{p} \big(\Seminorm{\dptt u}{L^1(0, T; H^p(\Omega))} + \Seminorm{v_0}{H^p(\Omega)} + c \Seminorm{u}{L^{\infty}(0, T; H^{p+1}(\Omega))} \big) \\
& \quad + \tau^{q+1} \big(c^2 \Norm{\Delta \dpt^{(q+1)} u}{L^1(0, T; L^2(\Omega))} + c \Norm{\nabla \dpt^{(q+1)} u}{L^{\infty}(0, T; L^2(\Omega)^d)}\big).
\end{alignat*}
\end{corollary}

\begin{remark}[$q$-explicit estimates]
Stability bounds and a priori error estimates with~$q$-explicit constants were derived for the space--time method~\eqref{eq:Walkington} in~\cite{Dong-Mascotto-Wang:2025}. These results have also been extended to a damped wave equation in~\cite{Zhang_Dong_Yan:2025}.
\eremk
\end{remark}
\section{Concluding remarks}
In this work, we presented a unified variational framework for the analysis of Galerkin-type time discretizations applied to parabolic and hyperbolic partial differential equations. By employing nonstandard test functions, we established stability bounds and \emph{a priori} error estimates in~$L^{\infty}(0, T; \mathcal{X})$ norms. This approach successfully circumvents the limitations of standard energy arguments, particularly for high-order approximations of wave propagation problems.

The proposed analysis closes several theoretical gaps and improves upon existing results in the literature. Key features of this framework include, for most of the methods considered, the absence of exponential Gr\"onwall factors, applicability to arbitrary polynomial degrees in space and time, the lack of artificial constraints between the time step~$\tau$ and the mesh size~$h$, and robustness with respect to model parameters.

Future research directions include the extension of this framework to nonlinear problems, equations with varying coefficients, and varying meshes and polynomial degrees. Finally, a numerical comparison of these methods, focusing on convergence rates in different norms, computational efficiency, and conditioning of the system matrices, would provide valuable insight into their practical advantages.

The following list outlines additional settings where the variational structure of these methods has been exploited, illustrating their versatility and broad applicability; references are provided for interested readers.

\begin{itemize}
\item \textbf{Maximal parabolic regularity:} this is a fundamental property in the analysis of parabolic evolution equations, as it provides optimal control of the time derivative and the spatial operator in the natural $L^p(0,T; \mathcal{X})$ norms, and can be used to establish well-posedness and stability results for a broad class of nonlinear and parameter-dependent problems. An important feature of DG time discretizations is that this property can also be established at the discrete level; see~\cite{Leykekhman_Vexler:2018,Akrivis_Larsson:2026,Akrivis_Makridakis:2022,Takahito_Tomoya:2024,Kovacs_Li_Lubich:2016}.

\item \textbf{Moving domains:} Galerkin-type time discretizations are particularly well suited for evolving domains or interfaces, where the lack of a tensor-product-in-time structure in the continuous formulation makes standard time-stepping schemes less natural; see~\cite{Bonito_Kyza_Nochetto:2013a,Bonito_Kyza_Nochetto:2013b,Beirao_Gomez_Haile:2026,Balazsov_Feistauer:2015,Balazsov_Feistauer_Vlasak:2018}.

\item \textbf{$q$-explicit analyses:} although the main focus of this work is the presentation of the general framework and the derivation of $(h,\tau)$-\emph{a priori} error estimates, these ideas can be adapted to obtain~$q$-explicit constants, which are particularly relevant for understanding the behavior of the method for high-order time discretizations and establishing exponential convergence for suitable geometrically refined meshes. Relying on the inverse estimate in Lemma~\ref{lemma:L2-Linfty} leads to suboptimal convergence rates for~$L^{\infty}(0, T; \mathcal{X})$ errors (see~\cite{Dong-Mascotto-Wang:2025} for the method in Section~\ref{sec:Walkington-wave}). Alternatively, measuring the error in weaker norms yields sharper $q$-convergence rates, as established in the literature for the methods in Section~\ref{sec:Jamet-heat}~\cite{Schotzau_Schwab:2000}, Section~\ref{sec:Aziz-Monk-heat}~\cite{Li_Yi:2025}, Section~\ref{sec:plain-vanilla-wave}~\cite{Antonietti_Mazzieri_DalSanto_Quarteroni:2018,Antonietti_Mazzieri_Migliorini:2020}, Section~\ref{sec:French-Peterson-wave}~\cite{Long_Li_Yi:2026}, Section~\ref{sec:Johnson-wave}~\cite{Antonietti_Mazzieri_Migliorini:2023}, and Section~\ref{sec:Walkington-wave}~\cite{Zhang_Dong_Yan:2025}.

\item \textbf{Discrete compactness:} the variational structure of these schemes also allows for establishing convergence to continuous weak solutions under minimal regularity assumptions on the data via discrete compactness arguments, mimicking the continuous Faedo--Galerkin approach. Because DG time discretizations lack a global time derivative, specialized analytical techniques must be employed; see~\cite{Walkington:2010,Riviere_Walkington:2011,Li_Riviere_Walkington:2015,Gazca-Orozco_Kaltenbach:2025,Gomez:2026,Kirk_Rhebergen:2023}.

\item \textbf{\emph{A posteriori} error estimates:} the variational framework provides a natural basis for deriving computable \emph{a posteriori} error estimators, which are essential for assessing the accuracy of numerical solutions and for guiding adaptive space--time mesh refinement. Such estimates have been developed for the methods in Section~\ref{sec:Jamet-heat}~\cite{Makridakis_Nochetto:2006}, Section~\ref{sec:French-Peterson-wave}~\cite{Gomez:2025}, Section~\ref{sec:Johnson-wave}~\cite{Johnson:1993,Makridakis_Nochetto:2006}, and~Section~\ref{sec:Walkington-wave}~\cite{Dong-Mascotto-Wang:2025,Dong-Georgoulis-Mascotto-Wang:2026}.

\end{itemize}

\paragraph{Acknowledgements.} The author is a member of the Gruppo Nazionale Calcolo
Scientifico-Istituto Nazionale di Alta Matematica (GNCS-INdAM).
\bibliographystyle{plain}
\bibliography{references}
\end{document}